\documentclass{amsart}

\usepackage[utf8]{inputenc}
\usepackage[T1]{fontenc}

\usepackage[backend=biber, giveninits=true, style=numeric, maxcitenames=3, maxbibnames=9]{biblatex} 
\AtEveryBibitem{\clearfield{issn} \clearfield{url} \clearfield{isbn} \clearfield{doi} }
\AtEveryBibitem{\clearlist{language}}
\DeclareFieldFormat[article, inbook, incollection, inproceedings, thesis]{title}{{#1\isdot}}
\DeclareFieldFormat[article, book]{volume}{\mkbibbold{#1}}
\DeclareFieldFormat{pages}{#1}
\DeclareLabelalphaNameTemplate{
  \namepart[use=true, pre=true, strwidth=1, compound=true]{prefix}
  \namepart[compound=true]{family}
}
\renewbibmacro{in:}{}
\usepackage{xurl} 
\usepackage{amsmath,amsthm,amssymb,amsfonts}
\usepackage{mathtools,mathrsfs} 
\usepackage{tikz-cd}
\usepackage{bbold} 
\usepackage{MnSymbol} 
\usepackage[inline]{enumitem}
\usepackage{anyfontsize} 
\usepackage{stmaryrd} 

\usepackage{xcolor}
\definecolor{IndigoBlue}{RGB}{4,35,96}

\usepackage[colorlinks]{hyperref}
\hypersetup{
    colorlinks=true,
	linkcolor=IndigoBlue,
	citecolor=IndigoBlue,
	urlcolor=IndigoBlue,
	pdftitle = {Canonical tilting bundles: the first nonabelian examples},
	pdfauthor = {Wei Tseu}
}

\theoremstyle{plain}
\newtheorem{thm}{Theorem}[section]
\newtheorem{lem}[thm]{Lemma}

\newtheorem{cor}[thm]{Corollary}

\newtheoremstyle{note}%
  {}{}
  {}{}
  {\itshape}{.}
  { }{}
\theoremstyle{note}
\newtheorem{defn}[thm]{Definition}
\newtheorem{rmk}[thm]{Remark}

\newtheorem{clm}{Claim}

\makeatletter 
\newcommand{\smallbullet}{} 
\DeclareRobustCommand\smallbullet{%
  \mathord{\mathpalette\smallbullet@{0.5}}%
}
\newcommand{\smallbullet@}[2]{%
  \vcenter{\hbox{\scalebox{#2}{$\m@th#1\bullet$}}}%
}
\makeatother

\newcommand\simrightarrow{\stackrel{\textstyle\sim}{\smash{\longrightarrow}\rule{0pt}{0.3ex}}} 

\newcommand{\on}[1]{\operatorname{#1}}
\newcommand{\Gr}{\mathrm{Gr}}
\newcommand{\TGr}[2]{T^{*}\Gr(#1,#2)}

\newcommand{\Hom}{\on{Hom}}
\newcommand{\Ext}{\on{Ext}}
\newcommand{\Sym}{\on{Sym}}
\newcommand{\End}{\on{End}}

\makeatletter 
\newcommand{\extp}{\@ifnextchar^\@extp{\@extp^{\,}}}
\def\@extp^#1{\mathop{\bigwedge\nolimits^{\!#1}}}
\makeatother

\newcommand{\ldb}{{\mathchoice{\mbox{\rm [\hspace{-0.15em}[}}
		{\mbox{\rm [\hspace{-0.15em}[}}
		{\mbox{\scriptsize\rm [\hspace{-0.15em}[}}
		{\mbox{\tiny\rm [\hspace{-0.15em}[}}}}
\newcommand{\rdb}{{\mathchoice{\mbox{\rm ]\hspace{-0.15em}]}}
		{\mbox{\rm ]\hspace{-0.15em}]}}
		{\mbox{\scriptsize\rm ]\hspace{-0.15em}]}}
		{\mbox{\tiny\rm ]\hspace{-0.15em}]}}}}

\newcommand{\mmid}{\mathrel{\mathchoice
		{\Vert} 
		{\Vert} 
		{\scriptstyle\Vert} 
		{\scriptscriptstyle\Vert} 
}}

\newcommand{\bC}{\mathbb{C}}
 
\newcommand{\bP}{\mathbb{P}}
\newcommand{\be}{\mathbb{e}} 
\newcommand{\bbf}{\mathbb{f}} 
\newcommand{\bS}{\mathbb{S}} 
\newcommand{\bT}{\mathbb{T}} 
\newcommand{\bD}{\mathbb{D}} 

\newcommand{\GL}{\mathrm{GL}}
\newcommand{\gl}{\mathfrak{gl}}
\newcommand{\fsl}{\mathfrak{sl}}

\newcommand{\fZ}{\mathfrak{Z}} 

\newcommand{\cE}{\mathcal{E}} 
\newcommand{\cF}{\mathcal{F}} 
\newcommand{\cU}{\mathcal{U}}
\newcommand{\cT}{\mathcal{T}}
\newcommand{\cR}{\mathcal{R}}
\newcommand{\cQ}{\mathcal{Q}}
\newcommand{\cV}{\mathcal{V}}
\newcommand{\cC}{\mathcal{C}}
\newcommand{\cO}{\mathcal{O}} 

\newcommand{\rH}{\mathscr{H}}

\newcommand{\tG}{\tilde{G}}
\newcommand{\tH}{\tilde{H}}

\newcommand{\sV}{\mathscr{V}}
\newcommand{\sW}{\mathscr{W}}
\newcommand{\sL}{\mathscr{L}}

\newcommand{\Cx}{\bC^{\times}}

\newcommand{\la}{\lambda}

\newcommand{\Db}{\mathnormal{D}^{\mathrm{b}}} 

\title{Canonical tilting bundles: the first nonabelian examples}

\author{Wei Tseu}

\date{}

\begin{document}

\begin{abstract}
We present an explicit construction of tilting bundles on cotangent bundles of Grassmannians of 2-planes.
This construction is based on Kapranov's exceptional collection for the underlying Grassmannians, and utilizes specific iterative extensions. 
The resulting tilting bundle exhibits invariance under the derived equivalence for the stratified Mukai flop through the geometric categorical $\mathfrak{sl}(2)$ action, providing a categorical lift of the K-theoretic canonical basis up to shifts. 
\end{abstract}

\maketitle

\section{Introduction}
We will work over $\bC$ throughout this paper, though all statements and arguments hold over any algebraically closed field of characteristic zero.

\subsection{Background}
Suppose $X$ is a smooth algebraic variety projective over an irreducible, affine, normal base $Y$. 
The pair $X \to Y$ is called a \textit{symplectic resolution} if there exists a non-degenerate symplectic form on the smooth locus of $Y$ that extends to a symplectic form on $X$.

Let $\Db(X) = \Db\mathrm{Coh}(X)$ be the derived category of bounded complexes of coherent sheaves on $X$.
We refer to standard references, such as the textbook \cite{Huy}, for definitions of various notions related to derived categories. 
An object $A$ is a \textit{classical generator} of $\Db(X)$ if the smallest full triangulated subcategory containing $A$, closed under isomorphisms and direct summands, is $\Db(X)$ itself.

\begin{defn}
	A \textit{tilting bundle}\footnote{For an object to be tilting in a triangulated category, the algebra $\End(\cE)$ must also have finite global dimension. This holds automatically in our geometric setting due to the finite global dimension of $\mathrm{Coh}(X)$; see \cite{KR20}.} 
	on $X$ is a vector bundle $\cE$ with no higher self-extensions $\Ext^{>0} (\cE, \cE) = 0$ that classically generates the derived category $\Db(X)$. 
\end{defn}

The tilting bundle brings a noncommutative perspective to the geometry of the variety $X$: it induces an equivalence
\begin{equation*}
    \begin{tikzcd}        
    R\Hom(\cE, -): \Db(X) \arrow[r, shift left] & \Db(\End(\cE)): - \otimes^{L} \cE  \arrow[l,shift left] 
    \end{tikzcd}
\end{equation*}
between $\Db(X)$ and the bounded derived category $\Db(\End(\cE))$ of finitely generated right modules over the usually noncommutative algebra $\End(\cE)$. 

Due to Kaledin \cite{Kaledin}, the existence of tilting bundles is known for symplectic resolutions via deformation quantization.
In the recently emerging \textit{symplectic duality} (see, e.g., \cite{symduality}), these resolutions of symplectic singularities appear in pairs with various symmetric geometric properties and have numerous connections to representation theory. 
In particular, the \textit{Higgs branch} of symplectic duality contains a rich family of symplectic varieties known as \textit{Nakajima quiver varieties}. 
As the simplest example, it includes the cotangent bundle $\TGr{n}{\bC^{N}}$ of the Grassmannian $\Gr(n,\bC^{N})$ of $n$-planes in $\bC^{N}$. 
It is well-known that, in the abelian case ($n=1$), the pullback of Beilinson's collection \cite{Bei} $\{\cO, \cdots, \cO(N-1)\}$ from the zero section $\bP^{N-1}$ to $T^{*}\bP^{N-1}$ constitutes to a tilting generator; see, e.g., \cite{Hara}. 

In this paper, we provide an explicit construction of tilting bundles on $\TGr{2}{\bC^{N}}$, motivated by the geometric representation theory of $\fsl(2)$ on the Nakajima quiver varieties $\bigsqcup_{n} \TGr{n}{\bC^{N}}$. 
As an illustration, we describe the bundle for the $N=4$ case below before stating the main theorem. 
Notably, the quantization approach is further developed by Webster \cite{Web,Web2} to explicitly construct the tilting bundles for the \textit{Coulomb branch} of symplectic duality, including all $\TGr{n}{\bC^{N}}$. 
This is detailed in a recent joint work with Suter \cite{Web24} for $N=4$, which coincides with our example below up to tensoring with the line bundle $\cO(1)$. 
Earlier, Toda and Uehara \cite{TU} provided an implicit construction of tilting bundles on $\TGr{2}{\bC^{4}}$ via an inductive argument based on the generator $\bigoplus_{i=0}^{4} \cO(i)$. 
We also refer to \cite[\S 4]{Web24} for a detailed comparison. 

\subsection{Example}
Let $V_{2}$ be the rank two tautological subbundle of the trivial bundle $\bC^{4}$ on $\Gr(2,\bC^{4})$. 
Kapranov's exceptional collection\footnote{Here we shifted the original collection $\{ \cO, \cO(-1), V_{2}, \cO(-2), V_{2}\otimes\cO(-1), \Sym^{2}V_{2} \}$ in \cite{Kap} by $\cO(1)\otimes (-)$ for our purposes.}  \cite{Kap} 
\begin{equation}\label{eq:Kap}
    \left\{ \cO(1), \cO, V_{2}^{\vee}, \cO(-1), V_{2}, \Sym^{2}V_{2}\otimes\cO(1) \right\}
\end{equation}
gives rise to a tilting bundle on $\Gr(2,\bC^{4})$ by taking the direct sum of (\ref{eq:Kap}).  
Pull these bundles in the collection back to $\TGr{2}{\bC^{4}}$ along the projection, and still denote them by the same letters. 
It was already noticed in \cite{Kawa} that although the collection still generates $\Db(\TGr{2}{\bC^{4}})$, there are nonvanishing higher extensions between these bundles due to the presence  of the cotangent fibres.  

In \cite{CKL-sl2}, Cautis, Kamnitzer, and Licata construct a natural equivalence 
\begin{equation*}
    \mathbb{T}: \Db(\TGr{n}{\bC^{N}}) \simrightarrow \Db(\TGr{N-n}{\bC^{N}})
\end{equation*}
for the stratified Mukai flop $\TGr{n}{\bC^N} \dashrightarrow \TGr{N-n}{\bC^N}$, as we will recall in \S\ref{sec:sl2}.
When $n=1$, Hara \cite{Hara} proves a result equivalent to Beilinson's collection staying stable under this equivalence, i.e.\ $\bT(\cO(-l)) = \cO(l)$ for $l = -1,0,\cdots, N-2$. 
Inspired by this fact, we examine the collection (\ref{eq:Kap}) under $\bT$ (which is an autoequivalence of $\Db(\TGr{2}{\bC^{4}})$, and find that 
\[
\bT(\cO(1)) = \det(\bC^{4}/V_{2})^{\vee} \cong \cO(-1), \quad \bT(\cO) = \cO, \quad \bT(V_{2}^{\vee}) = (\bC^{4}/V_{2})^{\vee}. 
\]
If we identify $\TGr{2}{\bC^{4}}$ with $\TGr{\bC^{4}}{2}$, the cotangent bundle of the Grassmannian of 2-dimensional quotients $V_{2}^{\prime}$ of $\bC^{4}$, via the canonical isomorphism $\bC^{4}/V_{2}\cong V_{2}^{\prime}$, then $\cO(1) = \det(V_{2})^{\vee}$, $\cO$, and $V_{2}^{\vee}$ on $\TGr{2}{\bC^{4}}$ are mapped to $\det(V_{2}^{\prime})^{\vee}$, $\cO$, and $(V_{2}^{\prime})^{\vee}$ on $\TGr{\bC^{4}}{2}$ respectively. 
In other words, these bundles are in some sense invariant under the flop equivalence $\bT$. 
We will justify the use of the term `invariance' in terms of a bar involution on the K-theory of  $\TGr{2}{\bC^{4}}$ later. 
For now, we can directly verify that these bundles have no higher extensions, so we keep them in the collection. 

Unfortunately, the other three bundles from (\ref{eq:Kap}) do not satisfy the invariance. 
However, the bundle $V_{2}$ sits in the short exact sequence $0\to V_{2} \to \bC^{4} \to \bC^{4}/V_{2} \to 0$, so the pair $\{ \cO, V_{2}\}$ generates the same category as $\{ \cO, \bC^{4}/V_{2}\}$. 
It turns out that $\cE_{0} := \bC^{4}/V_{2}$ is invariant
\[
\bT(\bC^{4}/V_{2}) = V_{2}. 
\]
To find replacements for $\cO(-1) = \det(V_{2})$ and $\Sym^{2}V_{2}\otimes \cO(1)$, we construct nontrivial extension bundles $\cE_{1}$ and $\cE_{-1}$ sitting in exact triangles of the form
\begin{align*}
    \det(\bC^{4}/V_{2}) [-1] \longrightarrow \det(V_{2}) \longrightarrow  &\cE_{1},\\
    \det(V_{2})^{\vee}\otimes \det(\bC^{4}/V_{2}) [-1] \longrightarrow \cO \longrightarrow  &\cE_{-1},
\end{align*}
whose classes are the unique (up to a scalar multiple) $\GL(\bC^{4})$-equivariant ones. 
The reader may ignore the equivariant structure and consider $\det(\bC^{4}/V_{2}) \cong \cO(1)$ and $\det(V_{2})^{\vee}\otimes \det(\bC^{4}/V_{2}) \cong \cO(2)$, then both extensions correspond to a generator of the one-dimensional piece
\begin{equation}\label{eq:choice}
    \bC \cong H^{1}_{\Gr(2,\bC^{4})} \left( \cO(-2) \otimes \cT_{\Gr(2,\bC^{4})} \right) \subset H^{1}_{\TGr{2}{\bC^{4}}} \left( \cO(-2) \right).
\end{equation}
There are multiple reasons for considering these extensions. 
First, we can relate the bundles to the original collection (\ref{eq:Kap}) and deduce the generation property. 
For example, the bundle $\det(V_{2})^{\vee}\otimes \det(\bC^{4}/V_{2})$, together with $\Sym^{2}V_{2}\otimes \cO(1)$ and $\cO$, appear as the associated graded pieces of certain natural filtrations on $\cO(1)\otimes \extp^{2}\bC^{4}$ and $V_{2}^{\vee} \otimes \extp^{1} \bC^{4}$ (see Lemma~\ref{lem:gen} for details). 
Thus, the collection $\{\cO(1),\cO,V_{2}^{\vee}, \Sym^{2}V_{2}\otimes \cO(1)\}$ generates the same category as $\{\cO(1),\cO,V_{2}^{\vee}, \det(V_{2})^{\vee}\otimes \det(\bC^{4}/V_{2})\}$. 
Second, the bundles $\cE_{1}$ and $\cE_{-1}$ are now invariant under $\bT$; their images sit in similar triangles  
\begin{align*}
    \det(\bC^{4}/V_{2})[-1] \longrightarrow \det(V_{2}) \longrightarrow  &\bT(\cE_{1}), \\
    \cO[-1] \longrightarrow \det(V_{2})\otimes \det(\bC^{4}/V_{2})^{\vee}  \longrightarrow &\bT(\cE_{-1})
\end{align*}
by one of our main results (Theorem~\ref{thm:2}). 
Finally, the new collection 
\[
\left\{\cO(1), \cO, V_{2}^{\vee}, \cE_{1}, \cE_{0}, \cE_{-1}
\right\}
\]
has no higher extensions any more. 
This is proved in \S\ref{sec:extvan} by chasing the long exact sequences, where the lowest-degree choice (\ref{eq:choice}) of the extension classes plays a crucial role in showing that certain connecting morphisms are surjective or isomorphic. 

\subsection{Main results}
Following these ideas, we discover a construction for general $N$ that involves iterative extensions. 
The main result (Theorem~\ref{thm:tilting}) of this paper is as follows.. 

\begin{thm}\label{thm:main}
    A tilting bundle $\cE$ on $\TGr{2}{\bC^{N}}$ is given by the direct sum of the following $\binom{N}{2}$ indecomposable $\GL(\bC^{N})$-equivariant vector bundles.  
    \begin{enumerate}[wide, labelindent=0pt, label = (\roman*)]
        \item The Schur functors\footnote{The definition of a Schur functor is recalled in \S\ref{sec:Schur}. As $V_{2}$ is of rank two, this Schur functor $\bS^{\la}V_{2}$ is, in particular, isomorphic to $\Sym^{\la_{1}-\la_{2}} V_{2} \otimes \det(V_{2})^{\la_{2}}$. } $\cE_{\la} = \bS^{\la}V_{2}$ for $-1 \leq \la_{2}\leq \la_{1} \leq N-4$. 
        
        \item For each $-1 \leq k \leq N-3$, an iterated nontrivial extension $\cE_{k}$, which is taken as the convolution of a complex of the form
        \begin{equation*}
            \cV_{k,-n^{-}_{k}}[-n^{+}_{k}-n^{-}_{k}] \longrightarrow \cdots \longrightarrow \cV_{k,i} [-n_{k}^{+} + i] \longrightarrow \cdots \longrightarrow \cV_{k,n^{+}_{k}}[0], 
        \end{equation*}
        where $n^{+}_{k} = \lfloor (N-3-k)/2 \rfloor$, $n^{-}_{k} = \lfloor (k+1)/2 \rfloor$, and 
        \[
        \cV_{k,i} = \det(V_{2})^{k+i} \otimes \extp^{N-3-k-2i} \bC^{N}/V_{2}
        \]
        for $i = -n_{k}^{-},\cdots, 0, \cdots, n_{k}^{+}$. 
    \end{enumerate}
\end{thm}
The differentials in the above complex are chosen to be the unique (up to a scalar multiple) $\GL(\bC^{N})$-equivariant ones as in \S\ref{sec:con}, and the convolution exists uniquely. 
In \S\ref{sec:gen}, we prove that the bundle $\cE$ is a classical generator of $\Db(\TGr{2}{\bC^{N}})$ by reducing the problem to the generation of Kapranov's collection. 
We then complete the proof of Theorem~\ref{thm:main} by computing the higher extension groups case by case in \S\ref{sec:extvan}.

Let $V_{N-2}$ be the rank $N-2$ tautological bundle on $\TGr{N-2}{\bC^{N}}$ and $V_{2}^{\prime} = \bC^{N}/V_{N-2}$. 
The images of $\cE_{\la}, \cE_{k}$ under the flop equivalence
\begin{equation*}
    \mathbb{T}: \Db(\TGr{2}{\bC^{N}}) \simrightarrow \Db(\TGr{N-2}{\bC^{N}})
\end{equation*}
are calculated by Theorem~\ref{thm:flop} as follows. 

\begin{thm}\label{thm:2}
    We have $\bT (\cE_{\la}) = \bS^{\la}V_{2}^{\prime}$, and $\bT (\cE_{k})$ is the convolution of a complex 
    \begin{equation*}
        \cV_{k,n^{+}_{k}}^{\prime}[-n^{-}_{k}-n^{+}_{k}] \longrightarrow \cdots \longrightarrow  \cV_{k,i}^{\prime}[-n^{-}_{k}-i] \longrightarrow \cdots \longrightarrow \cV_{k,-n^{-}_{k}}^{\prime}[0], 
    \end{equation*}
    where $\cV_{k,i}^{\prime} = \det(V_{2}^{\prime})^{k+i} \otimes \extp^{N-3-k-2i}  V_{N-2}$.     
\end{thm}

Again, the convolution is uniquely determined by its $\GL(\bC^{N})$-equivariant structure.
This theorem is proved by computing the defining Rickard complex of $\bT$ in \S\ref{sec:flop}, with the proof relying heavily on the Lascoux resolutions from \S\ref{sec:cal}. 

After choosing an appropriate $\Cx$-equivariant structure, we prove that the endomorphism algebra $\End(\cE)$ is a Koszul algebra in \S\ref{sec:Koszul}. 

\begin{thm}\label{thm:3}
	The algebra $\End(\cE)$ is Koszul. 
\end{thm}

Finally, our use of the term \textit{canonical} is motivated by a conjecture of Hikita \cite[\S 3.6]{Hikita}, which states that an invariant basis under a \textit{bar involution} on the equivariant K-theory of a nice symplectic resolution should admit a categorical lift to a tilting bundle. 
This conjecture generalizes a similar result for Springer resolutions, which was pioneered by Lusztig and proven in the celebrated work \cite{BM} of Bezrukavnikov and Mirkovi{\'c}. 

For a simply-laced Nakajima quiver variety (e.g.\ the cotangent bundle of a Grassmannian), the bar involution was already constructed by Varagnolo and Vasserot \cite{VV03}. 
In the last part (\S\ref{sec:inv}) of this paper, we prove that the K-theory classes of the indecomposable summands of $\cE$ are invariant under this bar involution and are contained in the canonical basis, as a corollary of Theorem~\ref{thm:2} and Theorem~\ref{thm:3}. 
For dimension reasons, this implies that our collection $\{\cE_{\la},\cE_{k}\}_{\la,k}$ provides a categorical lift of the canonical basis, up to shifts.

\subsection*{Acknowledgments}
I am grateful to Travis Schedler for his guidance and support. 
Special thanks to Tatsuyuki Hikita and Ben Webster for generously sharing their respective works with me during the early stages of this project. 
I would also like to thank Wahei Hara, Alice Rizzardo, Ed Segal, and Richard Thomas for helpful comments and discussions. 
This work was supported by the Engineering and Physical Sciences Research Council [EP/S021590/1]; The EPSRC Centre for Doctoral Training in Geometry and Number Theory (The London School of Geometry and Number Theory), University College London.

\section{Preliminaries}
In this section, we collect some facts about vector bundles, cohomology, and resolutions, and recall the construction of the geometric categorical $\fsl(2)$ action on cotangent bundles of Grassmannians.

Fix an $N$-dimensional vector space $\bC^{N}$, and consider the Grassmannian $\Gr(n,N)$ of $n$-dimensional subspaces of $\bC^{N}$. 
We will denote the rank $n$ tautological vector bundle on $\Gr(n,N)$, as well as its pullback to the cotangent bundle $\TGr{n}{N}$, by $V_{n}$. 
The Grassmannian admits an open covering by affine spaces of the form 
\[
U_{W} = \{ V \in \Gr(n,N) \mid V \cap  W =0 \} \cong \Hom(V_{n},W),
\]
where $W$ is an $(N-n)$-dimensional subspace complementary to $V_{n}\in \Gr(n,N)$ and the isomorphism assigns to any linear map $V_{n} \to W$ its graph.
Since $W \cong \bC^{N}/V_{n}$, these local charts identify the tangent bundle to $\Gr(n,N)$ with the Hom bundle $\Hom(V_{n},\bC^{N}/V_{n})$ via gluing.
Similarly, the cotangent bundle $\TGr{n}{N}$ is identified with the total space of $\Hom(\bC^{N}/V_{n}, V_{n})$ over $\Gr(n,N)$.
For details, see, for example, \cite{3264}.

\subsection{Quiver varieties}\label{sec:quiver}
The cotangent bundle of a Grassmannian can also be described as a GIT quotient.  
Suppose $V_n$ is an $n$-dimensional vector space, and let
\[
M = M(n,N) = \Hom(V_n,\bC^{N}) \oplus \Hom(\bC^{N}, V_n)
\]
be the space of `double framed quiver representations', on which $\GL(V_n)$ acts by change of basis:  $g\cdot (a,b) = (ag^{-1},gb)$, $g\in \GL(V_n), (a,b)\in M$. 
By definition \cite{Nak94}, the \textit{Nakajima quiver variety} $\TGr{n}{N}$ is the GIT quotient 
\begin{equation}\label{eq:Nak}
	\mu^{-1}(0)\sslash_{\chi} \GL(V_n) = \on{Proj} \bigoplus_{m \geq 0} \bC [\mu^{-1}(0)]^{\GL(V_n), \chi^{m}}, 
\end{equation}
where $\mu: M \to \gl(V_n), (a,b)\mapsto b\circ a$ is the moment map and $\chi(g) = \det(g)^{-1}$ is a character of $\GL(V_n)$. 
The quotient (\ref{eq:Nak}) also admits a natural projection to the categorical quotient 
\begin{equation}\label{eq:cotag}
	\on{Spec} \bC [\mu^{-1}(0)]^{\GL(V_n)} \cong \left\{ f = a \circ b \in \End(\bC^{N}) \mid f^{2}=0, \,\, \on{rk}(f) \leq n \right\}, 
\end{equation}
which is equivalent to the blow-down of the zero section of $\TGr{n}{N}$. 
For this reason, each cotangent vector to $\Gr(n,N)$ can be viewed as a square-zero endomorphism of $\bC^{N}$. 

Suppose we take the opposite stability condition $\chi^{-1}$ in the GIT quotient (\ref{eq:Nak}), the quiver variety becomes the cotangent bundle of the Grassmannian $\Gr(N,n)$ of $n$-dimensional quotients of $\bC^{N}$, i.e.\
\[
\mu^{-1}(0)\sslash_{\chi^{-1}} \GL(V_n) \cong \TGr{N}{n}. 
\]
Now, as the vector space $M$ carries a natural $\GL(\bC^N)$-module structure via $h\cdot (a,b) = (ha,bh^{-1})$, $h\in\GL(\bC^N)$, there is an induced action of $\GL(\bC^N)$ on either $\TGr{n}{N}$ or $\TGr{N}{n}$.
In fact, the quiver varieties $\TGr{n}{N}$ and $\TGr{N}{n}$ are related by a canonical isomorphism as follows.
Consider the map
\begin{equation*}
	(-)^\dagger: M\longrightarrow M, \quad (a,b) \longmapsto (b^{\on{t}}, -a^{\on{t}}),
\end{equation*}
as well as the group automorphism of $\GL(\bC^N) \times \GL(V_n)$ given by
\begin{equation*}
	(-)^\dagger: (g,h) \longmapsto ((g^{\on{t}})^{-1}, (h^{\on{t}})^{-1}),
\end{equation*}
where $(-)^{\on{t}}$ is the transpose. 
The two maps are compatible with the group actions $((g,h) \cdot (a,b) )^{\dagger} = (g,h)^{\dagger} \cdot (a,b)^{\dagger}$, hence they induce a $\GL(\bC^N)$-equivariant isomorphism between the quiver varieties (\cite[\S 4.6]{VV03})
\[
\dagger: \TGr{n}{N} \simrightarrow \TGr{N}{n}.
\]

There is another canonical isomorphism between quiver varieties, known as the `reflection functor', due to Lusztig \cite{LMN-L}, Maffei \cite{LMN-M}, and Nakajima \cite{LMN-N}.  
Consider the locally closed subvariety $Z$ of $M(n,N)\oplus M(N-n,N)$ consisting of all vectors $((a,b), (a',b'))$ such that the sequence 
\[
0\longrightarrow V_n \xlongrightarrow{a} \bC^N \xlongrightarrow{b'} V_{N-n} \longrightarrow 0
\] is exact, $b\circ a = 0 = b' \circ a'$, and $a\circ b = a' \circ b'$. 
The smooth quotient $Z/(\GL(V_n)\times \GL(V_{N-n}))$ turns out to be isomorphic to both $\TGr{n}{N}$ and $\TGr{N}{N-n}$ through the obvious $\GL(\bC^N)$-equivariant projections (\cite[\S 3.1]{LMN-M})
\[
\begin{tikzcd}
	\TGr{n}{N} & \arrow[l,"p_n"'] Z/(\GL(V_n)\times \GL(V_{N-n})) \arrow[r,"p_{N-n}"] & \TGr{N}{N-n}.
\end{tikzcd}
\]
We will refer to the induced isomorphism 
\begin{equation*}
	S_{w_0} = p_{N-n}\circ p_{n}^{-1}: \TGr{n}{N} \simrightarrow \TGr{N}{N-n}
\end{equation*}
of quiver varieties as \textit{the LMN isomorphism}.

\subsection{The grading}\label{sec:grading}
The cotangent bundle of a Grassmannian also carries a natural $\Cx$-action induced by the dilation on the vector space $M$. 
Given the description (\ref{eq:cotag}), this action of $\Cx$ acts by squared dilation on the cotangent fibres. 

We shall consider $\Cx$-equivariant sheaves in this paper. 
Following \cite[\S 2.4]{Cautis}, we adopt the (cohomological) \textit{internal degree shift} $\langle \cdot \rangle$ on the corresponding $\mathbb{Z}$-grading. 
For example, the group $\End(\cO) \cong H^{0}_{T^{*}\Gr}(\cO)$ consists of maps $\cO \to \cO$ of various degrees $0,2,4,\cdots$, and a degree, say, two map of $\End(\cO)_{2}$ becomes equivariant after shifting $\cO \to \cO \langle 2 \rangle$.
Here the target sheaf is concentrated in degree $-2$, and a scalar $t\in\Cx$ acts by power $-2$ on its local sections: $t\cdot f = t^{-2}(t\cdot f_{0})$, where $f\in \cO\langle 2\rangle (U)$ and $f_{0}$ is the same function but viewed in $\cO (U)$.  
As the degree shift is cohomological, we have the convention 
\[
\Hom(\cO, \cO\langle 2\rangle)_{j} = \Hom(\cO,\cO)_{j+2}. 
\]

\subsection{Schur functors}\label{sec:Schur}

Assume $V$ is a rank $n$ vector bundle on a smooth projective variety $X$.
We now recall the definition of a Schur functor of $V$.

A partition, or a non-increasing sequence of natural numbers, $\la = (\la_{1}, \cdots,\la_{n})$ is regarded as a Young diagram with $\la_{1}$ boxes in the first row, $\la_{2}$ boxes in the second row, and so on. 
Its conjugate Young diagram $\la^{\prime} = (\la_{1}^{\prime}, \cdots,\la_{l}^{\prime})$ counts the boxes in each column of $\la$. 
For example, for $\la = (3,1)$, the conjugate is $\la^{\prime} = (2,1,1)$.

\begin{defn}[{\cite[\S 4]{Kr}}]
	For a non-increasing sequence\footnote{From the definition, we can easily see that the Schur functor vanishes if the Young diagram $\la$ has more than $n$ nonzero rows. Thus, we may assume without loss of generality that $\la\in \mathbb{N}^n$.} $\la = (\la_{i}) \in \mathbb{N}^{n}$, define the \textit{Schur functor} $\bS^{\la} V$ as the image of the composition map
	\[
	\extp^{\la_{1}^{\prime}}V \otimes \cdots \otimes \extp^{\la_{l}^{\prime}} V 
	\xlongrightarrow{\Delta^{\otimes l}} V^{\otimes \sum \la_{i}^{\prime}} \xlongrightarrow{s_{\la^{\prime}}} V^{\otimes \sum \la_{i}} \xlongrightarrow{\on{m}} \Sym^{\la_{1}} V \otimes \cdots \otimes \Sym^{\la_{n}} V, 
	\]
	where 
	\begin{enumerate}[wide, labelindent=0pt, label = (\roman*)]
		\item The comultiplication map $\Delta: \extp^{d} V \to V^{\otimes d}$ sends any vector $v_{1}\wedge\cdots\wedge v_{d}$ to $\sum_{\sigma\in S_{d}} \on{sgn}(\sigma) v_{\sigma(1)}\otimes \cdots \otimes v_{\sigma(d)}$, and $\on{m}:V^{\otimes d} \to \Sym^{d}V$ is the natural multiplication map. 
		\item The permutation $s_{\la^{\prime}}$ acts on the $\sum \la_{i}^{\prime} = \sum \la_{i}$ indices by $\la_{1}^{\prime}+ \cdots + \la_{i-1}^{\prime} +j \mapsto \la_{1} + \cdots + \la_{j-1} + i$, where $1\leq j \leq \la_{i}^{\prime}$ and $1\leq i \leq \la_{j}$. 
	\end{enumerate}
\end{defn}

Note that the Schur functor defined here is intended to realize the irreducible representation of $\GL(n)$ with highest weight $\la$ (when $X$ is a point), whereas the one defined in \cite{Wey} (denoted by $L_{\la}$) has highest weight $\la^{\prime}$; see \cite[Theorem 2.2.10]{Wey}. 

Examples of Schur functors include 
\begin{equation*}
	\Sym^{d} V \cong \bS^{\la} V , \quad \la = (d); \quad 
	\extp^{d} V \cong \bS^{\la} V, \quad \la = ((1)^{d}). 
\end{equation*}
For a non-increasing sequence of integers $\la \in \mathbb{Z}^{n}$, we define 
\[
\bS^{\la} V = \bS^{(\la_{1}-\la_{n},\cdots, \la_{n}-\la_{n})} V \otimes \det(V)^{\la_{n}}. 
\]
In particular, this is an isomorphism when $\la \in \mathbb{N}^{n}$ (see Pieri's formula below). 
Using this notation, we have a canonical isomorphism \cite[p.83]{Wey}
\begin{equation}\label{eq:isomschur}
	\bS^{\la} V^{\vee} \cong \bS^{-\la} V, \quad -\la := (-\la_{n},\cdots, -\la_{1}). 
\end{equation}
The following special cases of the Littlewood--Richardson rule will be used frequently in this paper. 

\begin{thm}[{\cite[\S 2.3]{Wey}}]\label{thm:LR}
	\begin{enumerate}[wide, labelindent=0pt, label = (\arabic*)]    
		\item (Pieri's formula) For a Young diagram $\la$, 
		\[
		\bS^{\la}V \otimes \extp^{d} V \cong \bigoplus_{\mu} \bS^{\mu} V, \quad 
		\bS^{\la}V \otimes \Sym^{d} V \cong \bigoplus_{\nu} \bS^{\nu} V,
		\]
		where the sum is over all Young diagrams $\mu$ (resp.\ $\nu$) which are obtained by adding $d$ boxes to $\la$ with no two in the same row (resp.\ with no two in the same column). 
		\item Suppose $W$ is another rank $m$ vector bundle, then 
		\[
		\extp^{d}(V\otimes W) \cong \bigoplus_{\mu} \bS^{\mu} V \otimes \bS^{\mu^{\prime}} W, \quad 
		\Sym^{d}(V\otimes W) \cong \bigoplus_{\nu} \bS^{\nu} V \otimes \bS^{\nu} W,
		\]
		where the first sum is over all partitions $\mu$ of $d$ with at most $n$ rows and $m$ columns, and the second sum is over all Young diagrams $\nu$ with at most $\min\{n,m \}$ rows\footnote{When $\nu$ has more than $n$ or $m$ nonzero rows, at least one of $\bS^{\nu} V$, $\bS^{\nu} W$ vanishes. For this reason, we may also say that the second direct sum is taken over all Young diagrams.}.  
	\end{enumerate}
\end{thm}

\subsection{The Borel--Weil--Bott theorem}\label{sec:BWB}
Suppose $\cE$ is a rank $N$ vector bundle on a smooth projective variety $X$.
For $0\leq n\leq N$, consider the relative Grassmannian bundle $\pi: \Gr(n,\cE) \to X$. 
There is a short exact sequence of tautological vector bundles
    \begin{equation*}
        0\longrightarrow \cR \longrightarrow \cE \longrightarrow \cQ \longrightarrow 0, 
    \end{equation*}
where $\cR$ is the rank $n$ tautological subbundle and $\cQ = \cE/\cR$ is the rank $N-n$ quotient bundle. 
The Borel--Weil--Bott theorem tells us how to calculate the direct image of a Schur functor along $\pi$. 

\begin{thm}[{\cite[\S 4.1]{Wey}}]
    For any two non-increasing sequences of integers $\alpha = (\alpha_{1}, \cdots, \alpha_{n})$ and $\beta = (\beta_{1}, \cdots, \beta_{N-n})$, the direct image $\pi_{*}$ of the Schur functor $\bS^{\alpha}\cR^{\vee}\otimes \bS^{\beta} \cQ^{\vee}$ is equal to 
    \begin{equation*}
        \bS^{w \,\smallbullet\, (\alpha,\beta)} \cE^{\vee} [- \ell(w)]
    \end{equation*}
    if there is a unique Weyl group element $w\in S_{N}$ such that $w \,\smallbullet\, (\alpha,\beta) = w((\alpha,\beta)+\rho) - \rho$ is non-increasing, and zero otherwise. 
\end{thm}

Here $(\alpha,\beta)$ is the concatenation $(\alpha_{1},\cdots, \alpha_{n},\beta_{1},\cdots,\beta_{N-n})$ and $\rho = (N-1, \cdots, 0)$ is the (shifted) half sum of positive roots of $\GL(N)$.
In practice, the dot action from a transposition of the form $(n, n+1)$ operates as
\begin{equation*}
    (\alpha_{1},\cdots, \alpha_{n},\beta_{1},\cdots,\beta_{N-n}) \longmapsto (\alpha_{1},\cdots, \alpha_{n-1},\beta_{1}-1, \alpha_{n}+1, \beta_{2}, \cdots,\beta_{N-n})
\end{equation*}
whenever $\beta_{1}- \alpha_{n} >1$. 

\begin{rmk}\label{rmk:BWB}
	Using the isomorphism (\ref{eq:isomschur}), the direct image of $\bS^{\alpha}\cR \otimes \bS^{\beta} \cQ$ is equal to $\bS^{w \,\smallbullet\, (\beta, \alpha)} \cE [- \ell(w)]$ when the concatenation $(\beta, \alpha)$ is in the orbit of a dominant weight under the dot action. 
\end{rmk}

\subsection{The geometric categorical action}\label{sec:sl2}
The theory of the geometric categorical $\fsl(2)$ action is developed by Cautis, Kamnitzer, and Licata in a series of papers \cite{CKL-Duke,CKL-sl2, Cautis}. 

\begin{defn}\label{def:Hecke}
	For $0\leq m \leq n \leq N$, the \textit{Hecke correspondence} 
	\begin{equation*}
		\begin{tikzcd}
			\TGr{m}{N} & \arrow[l, "\pi_{m}"'] \fZ_{m,n} \arrow[r,"\pi_{n}"] & \TGr{n}{N}
		\end{tikzcd}
	\end{equation*} 
	is the total space of the Hom bundle $\Hom(\bC^{N}/V_{n},V_{m})$ over the partial flag variety $\mathrm{Fl}(m,n;N)$. 
	Here $\pi_{m}, \pi_{n}$ are the natural projections. 
\end{defn}

For $0\leq m \leq n \leq N$, consider the integral functors (\cite{CKL-sl2})
\begin{equation*}
	\begin{tikzcd}        
		\be^{n,m}: \Db(\TGr{n}{N}) \arrow[r,shift left] & \Db(\TGr{m}{N}): \bbf^{m,n}  \arrow[l,shift left]
	\end{tikzcd}
\end{equation*}
that are induced respectively by the kernels
\begin{align*}
	\cE^{n,m} = \cO_{\fZ_{m,n}} &\otimes \det(\bC^{N}/V_{n})^{m-n} \otimes \det(V_{m})^{n-m} \langle m(n-m) \rangle,\\ 
	\cF^{m,n} = \cO_{\fZ_{m,n}} &\otimes \det(V_{n}/V_{m})^{N-m-n} \langle (n-m)(N-n) \rangle
\end{align*}
via the Fourier--Mukai transforms
\begin{equation*}
	\be^{n,m}(-)  = \pi_{m,*} \left( \pi_{n}^{*}(-) \otimes \cE^{n,m} \right), \quad
	\bbf^{m,n}(-) = \pi_{n,*} \left( \pi_{m}^{*} (-) \otimes \cF^{m,n} \right). 
\end{equation*}
We also denote these functors and kernels by $\be^{(n-m)}, \bbf^{(n-m)}$ and $\cE^{(n-m)}$, $\cF^{(n-m)}$ when $n,m$ are obvious from the context.
The functors $\be,\bbf$ are both left and right adjoints of each other up to degree shifts (\cite[\S 2.1]{CKL}). 
We will mainly use the following counit 
\begin{align*}     
    \varepsilon: \bbf^{m,n}\be^{n,m}[-(n-m)(N-n-m)]\langle (n-m)(N-n-m) \rangle \longrightarrow \on{id}. 
\end{align*}
Moreover, the functor $\be$ or $\bbf$ composes with itself as follows: for $l\leq m\leq n$, 
\begin{align*}
    \be^{m,l} \be^{n,m} &\simeq \be^{n,l}\otimes H^{*}(\Gr(n-m, n-l)), \\
    \bbf^{m,n} \bbf^{l,m} &\simeq \bbf^{l,n}\otimes H^{*}(\Gr(m-l, n-l)). 
\end{align*}
We always shift the degree so that the cohomology ring is symmetric with respect to the 0th degree.
For example,
\[
H^{*}(\bP^{n}) \cong \bC[n]\langle -n \rangle \oplus \bC[n-2] \langle 2-n \rangle \oplus \cdots \oplus \bC[2-n] \langle n-2 \rangle \oplus \bC[-n]\langle n \rangle.
\] 

Now, assume $2n\leq N$. For each $0\leq i\leq n$, we consider the functors $\be^{(i)} = \be^{n,n-i}$, $\bbf^{(N-2n+i)} = \bbf^{n-i,N-n}$, and the composition $\Theta_{i} = \bbf^{(N-2n+i)} \be^{(i)} \langle i \rangle$. 
A differential map
\begin{equation*}
    d_{i}: \Theta_{i}[-i] \longrightarrow \Theta_{i-1}[-(i-1)]
\end{equation*}
can be defined in the following way
\[
\bbf^{(N-2n+i)} \be^{(i)}\langle 1 \rangle [-1] \longrightarrow \bbf^{(N-2n+i-1)} \bbf^{(1)}  \be^{(1)} \be^{(i-1)} \xlongrightarrow{\varepsilon}  \bbf^{(N-2n+i-1)} \be^{(i-1)}.
\] 
Here we first include $\be^{(i)}$ or $\bbf^{(N-2n+i)}$ into the lowest degree in the decomposition
\begin{align*}
    \be^{(1)}\be^{(i-1)} &\simeq \be^{(i)}\otimes H^{*}(\Gr(i-1,i)),\\
    \bbf^{(N-2n+i-1)} \bbf^{(1)} &\simeq \bbf^{(N-2n+i)} \otimes H^{*}(\Gr(1,N-2n+i))
\end{align*}
respectively. 
Next, we apply the counit to the two middle terms
\begin{equation*}
    \varepsilon: \bbf^{(1)}\be^{(1)} \longrightarrow \on{id}[N-2n+2i-1] \langle -N+2n -2i +1 \rangle. 
\end{equation*}
This construction gives rise to the \textit{Rickard complex} $\Theta_{\bullet} = (\Theta_{i}[-i], d_{i})$. 

A Postnikov system of the complex $\Theta_{\bullet}$ is a collection of distinguished triangles
\begin{equation*}
 \cC_{1}[-1] \xlongrightarrow{a_{1}} \Theta_{1}[-1] \xrightarrow{d_{1}} \Theta_{0}, \quad
 \cC_{i}[-i] \xlongrightarrow{a_{i}} \Theta_{i}[-i] \xlongrightarrow{b_{i}} \cC_{i-1}[-i+1]
\end{equation*}
such that $a_{i-1} \circ b_{i} = d_{i}$ for $2\leq i \leq n$. 
When it exists, the object $\cC_{n}$ is called a (right) \textit{convolution} of the complex. 

\begin{thm}[\cite{CKL-sl2}]\label{thm:CKL}
    The Rickard complex has a unique convolution, which defines a natural equivalence 
    \begin{equation*}
        \mathbb{T} = \bT_{n} = \on{Conv}(\Theta_{\bullet}): \Db(\TGr{n}{N}) \simrightarrow \Db(\TGr{N-n}{N}) 
    \end{equation*}
    for the local model of type $A$ stratified Mukai flops. 
\end{thm}

\subsection{Lascoux resolutions} \label{sec:torsion}
Let $G$ be a linearly reductive group acting on a vector space $T$, and let $U$ be a subspace of $T$ preserved by a parabolic subgroup $P$ of $G$.
Given a representation $V$ of $P$, we have an induced vector bundle $\cV= G \times_{P} V$ over $G/P$.
We will use the same letter to denote its pullback to the total space $T\times G/P$, as well as to the subspace $U\times_{P} G$. 
We are interested in the direct image of $\cV$ along 
\begin{equation*}
    \pi: U\times_{P}G \longrightarrow T, \quad (u,g)\longmapsto gu.
\end{equation*}
Once we factor the map $\pi$ through 
\begin{equation*}
    \begin{tikzcd}
        U\times_{P}G \arrow[r, hook, "i"] & T \times G/P \arrow[r, "q"] & T,
    \end{tikzcd}
\end{equation*}
where $i(u,g) = (gu, gP)$ and $q$ is the projection, we can apply the Koszul complex construction and compute the sheaf cohomology over $G/P$. 
The following theorem provides a Lascoux type resolution of this image (see also \cite[\S 6.1]{Wey}). 

\begin{thm}[{\cite[Theorem A.16]{DS}}]\label{thm:torsion}
    Suppose $\pi_{*}\cV$ is a sheaf concentrated in a single degree, then it has a $G$-equivariant resolution $\cF^{\bullet}$ where
    \begin{equation*}
        \cF^{-n} = \bigoplus_{\ell \geq n} R^{\ell -n}q_{*} \left(\cV \otimes \extp^{\ell} (T/U)^{\vee} \right).
    \end{equation*}
\end{thm}

Note that this construction also works in families: instead of vector spaces $U\subset T$, we may work with vector bundles $\cU\subset \cT$ relatively over a base scheme. 
For example, the projection $\pi_{n}: \fZ_{m,n} \to \TGr{n}{N}$ from the Hecke correspondence (Definition~\ref{def:Hecke}) is constant along the tautological direction $V_{n}\hookrightarrow \bC^{N}$.
Thus, the relative part of $\pi_{n}$ above $\Gr(n,N)$ is given exactly by the above map $\cU \times_{P} G \to \cT$ for
\[
\cT = \Hom(\bC^{N}/V_{n}, V_{n}),\quad \cU = \Hom(\bC^{N}/V_{n},V_{m}), \quad G/P \cong \Gr(V_{m},V_{n}).
\]

\section{Tilting bundle}

\subsection{The construction}\label{sec:con}
Let $G = \GL(\bC^{N})$ and $\tG = G\times\Cx$. 

\begin{thm}\phantomsection\label{thm:tilting}
    A tilting bundle $\cE$ on $\TGr{2}{N}$ is given by the direct sum of the following $\binom{N}{2}$ indecomposable $\tG$-equivariant vector bundles.  
    \begin{enumerate}[wide, labelindent=0pt, label = (\roman*)]
        \item The Schur functors $\cE_{\la} = \bS^{\la}V_{2}$ for $-1 \leq \la_{2}\leq \la_{1} \leq N-4$. 
        
        \item For each $-1 \leq k \leq N-3$, the convolution $\cE_{k} (= \cE_{N-3,k})$ of a complex of the form
        \begin{equation}\label{eq:complex}
            \cV_{k,-n^{-}_{k}}[-n^{+}_{k}-n^{-}_{k}] \xlongrightarrow{d} \cdots \xlongrightarrow{d} \cV_{k,i} [-n_{k}^{+} + i] \xlongrightarrow{d} \cdots \xlongrightarrow{d} \cV_{k,n^{+}_{k}}[0], 
        \end{equation}
        where $n^{+}_{k} = \lfloor (N-3-k)/2 \rfloor$, $n^{-}_{k} = \lfloor (k+1)/2 \rfloor$, and 
        \begin{equation}\label{eq:shift}
            \cV_{k,i} = \det(V_{2})^{k+i} \otimes \extp^{N-3-k-2i} \bC^{N}/V_{2}  \langle - (N-2-k-2i) \rangle
        \end{equation}
        for $i=-n_{k}^{-},\cdots,0,\cdots, n_{k}^{+}$.
    \end{enumerate}
\end{thm}

The differentials of the complex (\ref{eq:complex}) are chosen in Corollary~\ref{lem:diff} below.
As we will see, each extension class is uniquely determined by its equivariant structure, which can also be read from the internal degree shift of each $\cV_{k,i}$ in (\ref{eq:shift}). 
By Corollary~\ref{lem:exist}, \textit{the} convolution $\cE_{k}$ is well-defined. 
The notation $\cE_{k} = \cE_{N-3,k}$ suggests that these convolutions are intended to replace the Schur functors $\bS^{N-3,k}V_{2}$, $-1\leq k \leq N-3$ from Kapranov's collection $\{ \bS^{\la}V_{2} \mid -1 \leq \la_2 \leq \la_1 \leq N-3 \}$.

\begin{lem}\label{lem:vanext}
	For any $i,j$, we have
	\begin{equation*}
    	\Ext^{> \delta_{1j} }(\cV_{k,i}, \cV_{k,i+j}) = 0, \quad \Ext^{1}  = (\cV_{k,i}, \cV_{k,i+1})^{\tG} \cong \bC.
	\end{equation*}
\end{lem}

\begin{proof}
	These groups are calculated by
	\begin{align*}
		\Ext^{\bullet}_{T^{*}\Gr}(\cV_{k,i}, \cV_{k,i+j}) & \cong H^{\bullet}_{\Gr} \left( \cV_{k,i}^{\vee}\otimes \cV_{k,i+j}\otimes \on{Sym}( \bC^{N}/V_{2} \otimes V_{2}^{\vee} ) \right) \\
		& \cong \bigoplus_{\mu} H^{\bullet}_{\Gr} \left( \cV_{k,i}^{\vee}\otimes \cV_{k,i+j}\otimes \bS^{\mu}V_{2}^{\vee}\otimes \bS^{\mu}\bC^{N}/V_{2} \right)
	\end{align*}
	over all Young diagrams $\mu =(\mu_1, \mu_2)$.
	By Pieri's formula (Theorem~\ref{thm:LR}), there is a decomposition
	\[
	\extp^{N-3-k-2i} (\bC^{N}/V_{2})^{\vee} \otimes\extp^{N-3-k-2i-2j} \bC^{N}/V_{2} \cong \bigoplus_{\nu} \bS^{\nu} (\bC^{N}/V_{2})^{\vee},
	\] 
	where each $\nu$ is of the form $(1,\cdots,1,0,\cdots,0,-1,\cdots,-1)$ with $|\nu|=2j$ and at least $2j$ of its coordinates equal to $1$ (when $j\geq 0$, of course). 
	To apply the Borel--Weil--Bott theorem, we consider the concatenation weight 
	\begin{equation}\label{eq:conwtdif}
		(\mu_{1}-j,\mu_{2}-j, \nu\otimes -\mu),
	\end{equation}
	where by abuse of notation, we use $\nu\otimes -\mu$ to denote all possible weights appearing in the decomposition of the tensor product $\bS^{\nu} (\bC^{N}/V_{2})^{\vee} \otimes \bS^{(0,\cdots,0,-\mu_{2},-\mu_{1})} (\bC^{N}/V_{2})^{\vee}$ according to the Littlewood--Richardson rule (\cite[\S 2.3]{Wey}). 
	When $\mu_{2}-j \geq 0$, the weight (\ref{eq:conwtdif}) is either already dominant or never in a dominant orbit. 
	So there are no higher extensions in this case. 
	Next we assume $\mu_{2}-j < 0$.
	To obtain a nonzero $H^{>0}$, we need to permute the coordinate $\mu_{2}-j$ rightward step by step under the dot actions of consecutive transpositions to obtain a dominant weight.
	Because there are at least $2j$ coordinates $1$ in $\nu$, there are at least $2j-2$ (resp.\ $2j-1$) coordinates $1$ in any $\nu\otimes -\mu$ if $\mu_{2}\neq 0$ (resp.\ $\mu_{2}=0$) by the Littlewood--Richardson rule. 
	This means unless $j=1$, whenever $\mu_{2}-j$ is permuted into $0$ by $j-\mu_{2}$ consecutive transpositions, there are still $1$'s after it, and in that case the weight cannot be further permuted into a dominant weight. 
	However, if $j=1$ (and necessarily $\mu_{2}=0$ to meet the assumption $\mu_{2}-j <0$), we can obtain dominant weights of the form $(\mu_{1}-1, 0,\cdots, 0, -\mu_{1}+1)$ for all $\mu_{1}\geq 1$, with $\nu = (1,1,0,\cdots,0)$ and $\nu\otimes -\mu = (1,0,\cdots, 0,-\mu_{1}+1)$.  
	These weights contribute to all
	\begin{equation*}\label{eq:Ext1}
		\Ext^{1}_{T^{*}\Gr}(\cV_{k,i},\cV_{k,i+1}) \cong \bigoplus_{\mu_{1} \geq 1} \bS^{(\mu_{1}-1,0,\cdots,0,-\mu_{1}+1)} (\bC^{N})^{\vee}. 
	\end{equation*}
	The $G$-invariant module $\bS^{(0,\cdots,0)}(\bC^{N})^{\vee}$ only occurs with multiplicity one at $\mu_{1}=1$.
	It is also of $\Cx$-weight two, which aligns with our choice of the internal degree shift in (\ref{eq:shift}). 
\end{proof}

\begin{cor}\label{lem:diff}
	There exists a unique (up to scalar multiples) nonzero $\tG$-equivariant morphism 
	\[d: \cV_{k,i}[-1]\longrightarrow \cV_{k,i+1}\] 
	for every $k,i$. 
	Moreover, the morphisms satisfy $d\circ d[-1]=0$. 
\end{cor}

\begin{proof}
	It suffices to show  that
	\begin{equation*}
		\Ext^{1}(\cV_{k,i}, \cV_{k,i+1})^{\tG} \cong \bC, \quad \Ext^{2}(\cV_{k,i}, \cV_{k,i+2}) = 0. 
	\end{equation*}
	These have been covered by Lemma~\ref{lem:vanext}. 
\end{proof}

\begin{cor}\label{lem:exist}
    The convolution $\cE_{k}$ uniquely exists.
\end{cor}

\begin{proof}
    By \cite[Proposition 8.3]{CK}, it is enough to check that for all $i$ 
    \begin{equation*}
        \Ext^{1}(\cV_{k,i}, \cV_{k,i+1+j}) = 0 = \Ext^{2}(\cV_{k,i}, \cV_{k,i+2+j}), \quad j\geq 1.
    \end{equation*}
    Again, these have been proved in Lemma~\ref{lem:vanext}. 
\end{proof}

\subsection{Generation}\label{sec:gen}
Let $\mathnormal{D}\mathrm{Qcoh}(T^{*}\Gr(2,N))$ be the unbounded derived category of quasi-coherent sheaves. 
On $\TGr{2}{N}$, there is a short exact sequence of tautological vector bundles
\begin{equation*}
	0 \longrightarrow V_{2} \longrightarrow \bC^{N} \longrightarrow \bC^N / V_{2} \longrightarrow 0. 
\end{equation*}
Any exterior power $\extp^{d} \bC^{N}$ then has a natural filtration with associated graded pieces \cite[\S 2.3]{Wey}
\begin{equation*}
	(\extp^2 V_2 \otimes \extp^{d-2}\bC^N / V_{2}) \oplus (\extp^1 V_2 \otimes \extp^{d-1} \bC^N / V_{2}) \oplus
	(\extp^{d} \bC^N / V_{2}). 
\end{equation*}
We will apply this filtration repeatedly to deduce the generation property.

\begin{lem}\label{lem:gen}
    The collection $\{\cE_{\la},\cE_{k}\}_{\la,k}$ spans $\mathnormal{D}\mathrm{Qcoh}(T^{*}\Gr(2,N))$. 
\end{lem}

This means, if $A$ is an object in $\mathnormal{D}\mathrm{Qcoh}(T^{*}\Gr)$ such that for all $\la, k$
    \begin{equation}\label{eq:gen}
        R \Hom(\cE_{\la},A) = 0 = R \Hom(\cE_{k},A),
    \end{equation}
    then $A=0$.

\begin{proof}
    We will show that (\ref{eq:gen}) implies $R\Hom(\bS^{N-3,k}V_{2},A) = 0$ for $-1\leq k \leq N-3$.
    As a result, we obtain $A=0$ since the collection $\{\cE_{\la}, \bS^{N-3,k}V_{2}\}_{\la, k}$ is a shift (by $-\otimes \cO(1)$) of the original Kapranov's collection $\{\bS^{\la} V_{2} \mid 0\leq \la_2 \leq \la_1 \leq N-2\}$, which in turn spans $\mathnormal{D}\mathrm{Qcoh}(T^{*}\Gr)$ due to the fact that it is pulled back from the underlying Grassmannian along the {affine} bundle map (see, e.g.,
    \ \cite[Lemma 3.1]{Hara-Abuaf}). 
    Set 
    \[
    \cV_{\mu_{1},\mu_{2},d} := \bS^{\mu_{1},\mu_{2}}V_{2} \otimes \extp^{d} \bC^{N}/V_{2}.
    \] 
    
    \begin{clm}
        Assume $R \Hom(\cE_{\la},A) = 0$ for all $\la$, then $R\Hom(\cV_{k,i},A) = 0$ for any $k$ and $i\neq 0$.
    \end{clm}
    
    \begin{proof}[Proof of claim]\renewcommand{\qedsymbol}{}
        In fact, we will use induction to prove a more general statement: 
        \begin{equation}\label{eq:claim}
            R\Hom(\cV_{\mu_{1},\mu_{2},d} ,A) = 0
        \end{equation} 
        if $N-2-d \leq \mu_{2} \leq \mu_{1}\leq N-4$ or $-1\leq \mu_{2}\leq \mu_{1}\leq N-4-d$. 
        In the former case, the descending induction begins at $d=N-2$, where
        \[\cV_{\mu_{1},\mu_{2},N-2} \cong \cE_{\mu_{1}-1,\mu_{2}-1} \otimes \det(\bC^{N}). \] 
        By assumption (\ref{eq:gen}), we immediately have $R\Hom(\cE_{\mu_{1}-1,\mu_{2}-1},A)=0$. 
        For any smaller $d$, there is a natural filtration on $\bS^{\mu_{1}-1,\mu_{2}-1} V_{2} \otimes \extp^{d+2} \bC^{N}$ whose associated graded pieces are
        \[\cV_{\mu_{1}-1,\mu_{2}-1,d+2} \oplus (\cV_{\mu_{1}-1,\mu_{2}-1,d+1} \otimes V_{2}) \oplus \cV_{\mu_{1},\mu_{2},d}. \]
        By induction, the functor $R\Hom(-,A)$ vanishes on the first three aforementioned bundles. 
        Hence, it also vanishes on the last one $R\Hom(\cV_{\mu_{1},\mu_{2},d},A) =0$ by the long exact sequences. 
        When $-1\leq \mu_{2}\leq \mu_{1}\leq N-4-d$, the induction is run increasingly over $d$. 
        The $d=0$ case is covered by the assumption of this claim as $\cV_{\mu_{1},\mu_{2},0} = \cE_{\mu_{1},\mu_{2}}$. 
        For $d>0$, there is a filtration on $\bS^{\mu_{1},\mu_{2}} V_{2} \otimes \extp^{d} \bC^{N}$ whose associated graded pieces are 
        \[ \cV_{\mu_{1}+1,\mu_{2}+1,d-2} \oplus (\cV_{\mu_{1},\mu_{2},d-1} \otimes V_{2}) \oplus \cV_{\mu_{1},\mu_{2},d}. \] 
        By induction, the first three aforementioned bundles have no nonzero $R\Hom$ to $A$, hence so does the last one.  
    \end{proof}

    Now, we notice that the bundles 
    \[\cV_{k,0} = \cV_{k,k,N-3-k},\quad \cV_{k+1,k,N-4-k},\quad \cV_{k,1} = \cV_{k+1,k+1,N-5-k}\]
    are the associated graded pieces of a filtration on $\det(V_{2})^{k}\otimes \extp^{N-3-k}\bC^{N}$.
    By (\ref{eq:claim}), we have $R\Hom(-,A)=0$ on the last two aforementioned bundles, hence
    \[R\Hom(\cV_{k,0},A)=0 \iff R\Hom(\cV_{k+1,k,N-4-k},A)=0. \] 
    Repeat this process, we see that $R\Hom(-,A)$ vanishes on $\cV_{k+1,k,N-4-k}$ iff it vanishes on $\cV_{k+2,k,N-5-k}$, and so on until we reach $\cV_{N-3,k,0} = \bS^{N-3,k}V_{2}$. 
    In the end, we have $R\Hom(\cV_{k,0},A) = 0$ iff $R\Hom(\bS^{N-3,k}V_{2},A) = 0$. 
\end{proof}

Now, by a general result \cite[\S\S2-3]{BVdB}, Lemma~\ref{lem:gen} implies that the bundle 
\[
\cE = \bigoplus_{\la}\cE_{\la}\oplus\bigoplus_{k}\cE_{k}
\]
is a classical generator of the compact objects $\on{Perf}(T^{*}\Gr) = \Db(T^{*}\Gr)$ inside the larger category $\mathnormal{D}\mathrm{Qcoh}(T^{*}\Gr)$. 

\begin{cor}\label{cor:gen}
	The bundle $\cE$ is a classical generator of $\Db(\TGr{2}{N})$. 
\end{cor}

\begin{rmk}
	Up to taking the dual and tensoring with a power of $\cO(1)$, the collection $\{\cE_{\la}, \cV_{k,0}\}_{\la, k}$ coincides with Kaneda's full strongly exceptional collection \cite{Kaneda} on $\Gr(2,\mathbb{k}^{N})$, which also appears as subquotients of the Frobenius image of $\cO_{\Gr}$ over a field $\mathbb{k}$ of large positive characteristic; see \cite[\S 16.6]{RSVdB}. 
\end{rmk}

\subsection{Ext vanishing}\label{sec:extvan}
We now compute the extension groups case by case. 

\begin{lem}\label{lem:van1}
	We have $\Ext^{>0}(\cE_{\la},\cE_{\la^{\prime}} ) =0$ for any $\la, \la^{\prime}$. 
\end{lem}

\begin{proof}
	We can compute
	\[
	\Ext^{\bullet}_{T^{*}\Gr} (\cE_{\la},\cE_{\la^{\prime}}) \cong \bigoplus_{\mu} H^{\bullet}_{\Gr}\left( \bS^{\la_{1},\la_{2}}V_{2}^{\vee}\otimes \bS^{-\la_{2}^{\prime},-\la_{1}^{\prime}}V_{2}^{\vee}\otimes \bS^{\mu}V_{2}^{\vee} \otimes \bS^{\mu} \bC^{N}/V_{2} \right)
	\]
	over all Young diagrams $\mu=(\mu_{1},\mu_{2})$. 
	By the Littlewood--Richardson rule, there is a decomposition 
	\[\bS^{\la_{1},\la_{2}}V_{2}^{\vee}\otimes \bS^{-\la_{2}^{\prime},-\la_{1}^{\prime}}V_{2}^{\vee} \cong \bigoplus_{\nu} \bS^{\nu} V_{2}^{\vee} \] 
	where each $\nu$ satisfies $\nu_{2} \geq \la_{2}-\la_{1}^{\prime} \geq 3-N$. 
	To apply the Borel--Weil--Bott theorem, we consider the concatenation weights $(\nu\otimes \mu, -\mu)$. Here we abuse the notation to consider all possible weights appearing in the decomposition of $\bS^{\nu}V_{2}^{\vee} \otimes \bS^{\mu}V_{2}^{\vee}$ according to the Littlewood--Richardson rule. 
	If $\nu_{2}\geq 0$, the weight is always dominant and there is no nonvanishing higher cohomology. 
	Otherwise, we may permute the weight to obtain a nonzero $H^{>0}$. 
	In that case, the permutation must have length at least $N-3$ when $\mu_{2} = 0$.  
	This requires at least $\nu_{2}+N-3 \leq -1$, which contradicts $\nu_{2}\geq 3-N$. 
	Similarly, when $\mu_{2}>0$, the permutation should have length at least $N-4$, and that requires at least  $\nu_{2}+\mu_{2}+N-4 \leq -1$, a contradiction again. 
\end{proof}

\begin{lem}\label{lem:van2}
	For any $\la,k$, we have
	\[
	\Ext^{>0} (\cE_{\la},\cE_{k} ) = 0 = \Ext^{>0} (\cE_{k},\cE_{\la} ).
	\]
\end{lem}
    
\begin{proof}
    We first compute 
    \[
    \Ext^{\bullet}_{T^{*}\Gr} (\cE_{\la},\cV_{k,i}) \cong \bigoplus_{\mu} H^{\bullet}_{\Gr} \left( \bS^{\la} V_{2}^{\vee}\otimes \cV_{k,i} \otimes \bS^{\mu}V_{2}^{\vee}\otimes \bS^{\mu} \bC^{N}/V_{2}\right), 
    \]
    which amounts to consider concatenation weights of the form 
    \begin{equation}\label{eq:conwtlk}
        \left(\la \otimes \mu \otimes (-k-i,-k-i), ((-1)^{N-3-k-2i}) \otimes -\mu \right).
    \end{equation}
    Here $((-1)^{n})$ denotes the weight that has $- 1$ on its last $n$ coordinates and zero elsewhere.
    We also abuse the notation $\la\otimes \mu$ to denote all possible weights appearing in the decomposition  
    \[\bS^{\la}V_{2}^{\vee} \otimes \bS^{\mu} V_{2}^{\vee}  \cong \bigoplus_{t} \bS^{\la_{1}+\mu_{1}-t, \la_{2}+\mu_{2}+t} V_{2}^{\vee}, \]
    $0\leq t \leq \min \{\la_{1}-\la_{2}, \mu_{1}-\mu_{2}\}$, by the Littlewood--Richardson rule.
    There are three situations where  (\ref{eq:conwtlk}) can be permuted into a dominant weight. 
    \begin{enumerate}[wide, labelindent=0pt, label = (\roman*)]
        \item The permutation reaches the longest possible length, i.e.\ when the coordinate $\la_{2} + \mu_{2} -k-i + t$ is permuted rightward by $k+2i+1$ consecutive transpositions into $-1$. 
        This gives rise to 
        \begin{equation}\label{eq:H0}
            H^{k+2i+1} \cong \bigoplus_{\mu } \bigoplus_{t} \bS^{(\la_{1}+\mu_{1}-k-i-t, -1,\cdots, -1, -1-\mu_{2},-1-\mu_{1})} (\bC^{N})^{\vee}, 
        \end{equation}
        where $\mu_{1}-t\geq k+i-1-\la_{1}$. 
        In this case, $i = -\la_{2}-\mu_{2}-2-t$ ranges from $-\la_{2}-2$ to $-n_{k}^{-}$ (if $k$ is even) or $-n_{k}^{-}+1$ (if $k$ is odd). 

        \item The permutation has length $k+2i$ and $\mu_{1}$ is necessarily nonzero. 
        In this case, the cohomology is 
        \begin{equation}\label{eq:H1}
            H^{k+2i} \cong \bigoplus_{\mu_{1}\neq 0} \bigoplus_{t} \bigoplus_{s=0,1} \bS^{(\la_{1}+\mu_{1}-k-i-t, -1,\cdots, -1, -1-\mu_{2}+1-s,-1-\mu_{1}+s)} (\bC^{N})^{\vee},
        \end{equation}
        where $\mu_{1}-t\geq k+i-1-\la_{1}$, and $s$ takes $0$ only if $\mu_{2}>0$ and $1$ only if $\mu_{1}>\mu_{2}$. 
        In this case, $i = -\la_{2}-\mu_{2}-1-t$ ranges from $-\la_{2}-1$ to $-n_{k}^{-}+1$. 

        \item The permutation has length $k+2i-1$ and $\mu_{2}$ is necessarily nonzero. 
        This gives 
        \begin{equation}\label{eq:H2}
            H^{k+2i-1} \cong \bigoplus_{\mu_{2}\neq 0} \bigoplus_{t} \bS^{(\la_{1}+\mu_{1}-k-i-t, -1,\cdots, -1, -1-\mu_{2}+1,-1-\mu_{1}+1)} (\bC^{N})^{\vee},
        \end{equation}
        where $\mu_{1}-t\geq k+i-1-\la_{1}$. 
        In this case, $i = -\la_{2}-\mu_{2}-t$ ranges from $-\la_{2}-1$ to $-n_{k}^{-}+1$ (if $k$ is even) or $-n_{k}^{-}+2$ (if $k$ is odd). 
    \end{enumerate}
    
    In summary, all nonvanishing higher extensions are of the form
    \begin{equation}\label{eq:nonvaext}
        \Ext^{k+2i+j}(\cE_{\la},\cV_{k,i}) \neq 0,
    \end{equation}
    where $-n_{k}^{-} + \lceil k/2 \rceil - \lfloor k/2 \rfloor \leq i \leq -\la_{2}-1$ and $j$ takes the following values: $j=-1,0$ at $i=-\la_{2}-1$; $j=0,1$ at $i=-n_{k}^{-}+1$ ($k$ odd); $j=1$ at $i=-n_{k}^{-}$ ($k$ even); $j=-1,0,1$ at all other $i$. 
    In particular, this implies that there are no higher extensions of $\cE_{\la}$ by $\cV_{k,i}$ for $i> -\la_{2}-1$.
    Thus, the problem is reduced to the extensions of $\cE_{\la}$ by the iterated cones     
    \[
    \cC_{\alpha}:= \on{Conv}\left( \cV_{k,-\la_{2}-1-\alpha}[-\alpha] \to \cdots \to \cV_{k, -\la_{2}-1}\right), \quad \alpha = 1,\cdots, -\la_{2}-1+n_{k}^{-}.
    \]
    We will use induction to show that $\Ext^{>0}(\cE_{\la}, \cC_{-\la_{2}-1+n_{k}^{-}}) = 0$. 
    
    Set 
    \[
    \rH^{\bullet}(-)  = \Ext^{\bullet}(\cE_{\la},-), \quad \kappa = k-2\la_{2}-2.
    \] 
    In fact, we only need to look at $\rH^{\leq \kappa}$ since $\rH^{>\kappa}(\cV_{k,i})=0$ for all $i$ by (\ref{eq:nonvaext}).
    Also from (\ref{eq:nonvaext}) we know that
    \[\rH^{\kappa-2(\alpha-1)}(\cV_{k,i}) = 0,\quad \rH^{\kappa-2(\alpha-1) -1}(\cV_{k,i}) = 0, \quad  \rH^{\kappa-2(\alpha-1)-2}(\cV_{k,i})= 0 \]  for $i\leq -\la_{2}-2-\alpha$.
    Hence, we can further reduce the problem to
        \begin{align*} 
        \rH^{\kappa-2(\alpha-1)}(\cC_{-\la_{2}-1+n_{k}^{-}}) & \cong \rH^{\kappa-2(\alpha-1)} (\cC_{\alpha}),\\ 
        \rH^{\kappa-2(\alpha-1)-1}(\cC_{-\la_{2}-1+n_{k}^{-}}) & \cong \rH^{\kappa-2(\alpha-1)-1} (\cC_{\alpha}).
        \end{align*}
    As $\rH^{\kappa-2\alpha+2}(\cV_{k,-\la_{2}-1-\alpha}) = 0$, we have a long exact sequence
    \begin{align*}
        \cdots \longrightarrow \rH^{\kappa-2\alpha}(\cV_{k,-\la_{2}-1-\alpha}) &\xlongrightarrow{\delta_{1}}  \rH^{\kappa -2\alpha+1}(\cC_{\alpha-1}) \longrightarrow  
          \rH^{\kappa-2\alpha+1}(\cC_{\alpha}) \longrightarrow 
          \\ \rH^{\kappa-2\alpha+1}(\cV_{k,-\la_{2}-1-\alpha}) &\xlongrightarrow{\delta_{0}} 
        \rH^{\kappa-2\alpha+2}(\cC_{\alpha-1}) \longrightarrow
        \rH^{\kappa-2\alpha+2}(\cC_{\alpha}) \longrightarrow 
        0. 
    \end{align*}
    From (\ref{eq:nonvaext}) again, we know that $\rH^{\kappa-2\alpha+1}$ and $\rH^{\kappa-2\alpha+2}$ vanish on $\cV_{k,i}$ for $i > -\la_{2}-\alpha$.
    Hence, they also vanish on $\cC_{\alpha-2}$, and we can deduce 
    \begin{equation}\label{eq:HH}
    \rH^{\kappa-2\alpha+1} (\cC_{\alpha-1})  \cong \rH^{\kappa-2\alpha+1} (\cV_{k,-\la_{2}-\alpha})
    \end{equation}
    from the long exact sequence of the triangle $\cC_{\alpha-2} \to \cC_{\alpha-1} \to \cV_{k,-\la_{2}-\alpha}$. 
    Now by Claim \ref{clm:conn} below, the connecting morphism 
    \[\delta_{1}: \rH^{\kappa-2\alpha}(\cV_{k,-\la_{2}-1-\alpha}) \longrightarrow  \rH^{\kappa-2\alpha+1} (\cV_{k,-\la_{2}-\alpha})
    \]
    is surjective. 
    On the other hand, we have $\rH^{\kappa-2\alpha+3}(\cC_{\alpha-1}) =0$ by induction, so a short exact sequence
    \[
    0 \longrightarrow \rH^{\kappa-2\alpha+2}(\cC_{\alpha-1}) \longrightarrow \rH^{\kappa-2\alpha+2}(\cV_{k,-\la_{2}-\alpha}) \longrightarrow \rH^{\kappa-2\alpha+3}(\cV_{k,-\la_{2}-\alpha+1}) \longrightarrow 0
    \]
    can be extracted from the long exact sequence of the same triangle. 
    Here the right end term comes from the isomorphism (\ref{eq:HH}) by replacing the index $\alpha$ there with $\alpha-1$.  
    By Claim \ref{clm:conn} again, there is also a short exact sequence 
    \[
    0 \longrightarrow \rH^{\kappa-2\alpha+1}(\cV_{k,-\la_{2}-1-\alpha}) \longrightarrow \rH^{\kappa-2\alpha+2}(\cV_{k,-\la_{2}-\alpha}) \longrightarrow \rH^{\kappa-2\alpha+3}(\cV_{k,-\la_{2}-\alpha+1}) \longrightarrow 0. 
    \]
    As the connecting $\delta_{0}$ defines a compatible morphism between the two short exact sequences, it must be an isomorphism. 
    As a result, we assert 
    \[\rH^{\kappa-2\alpha+2}(\cC_{-\la_{2}-1+n_{k}^{-}})= 0=\rH^{\kappa-2\alpha+1}(\cC_{-\la_{2}-1+n_{k}^{-}}), \]
    and the induction goes on to lower degrees. 
    
    \begin{clm}\label{clm:conn}
        The differential $d$ induces a short exact sequence
        \begin{align*}
            0 \longrightarrow \rH^{k+2(i-1)+1}(\cV_{k,i-1}) \xlongrightarrow{\rH^{\bullet} (d)} \rH^{k+2i}(\cV_{k,i}) \xlongrightarrow{\rH^{\bullet} (d)} \rH^{k+2(i+1)-1}(\cV_{k,i+1}) \longrightarrow 0
        \end{align*}
        for each $k,i$.
    \end{clm}
    Note that the degrees here are written in the form of $k+2i+j$ as in (\ref{eq:nonvaext}). 

    \begin{proof}[Proof of claim]\renewcommand{\qedsymbol}{}
        These groups are computed explicitly in (\ref{eq:H0}) (by replacing the index $i$ there with $i-1$), (\ref{eq:H1}), and (\ref{eq:H2}) (by replacing the index $i$ there with $i+1$) respectively. 
        In all three cases, $\mu_{2}+t$ is consistently equal to $-\la_{2}-1-i$, which means for a fixed $0 \leq t \leq \min\{\la_{1}-\la_{2},\mu_{1}-\mu_{2}\}$, the digit $\mu_{2}$ is uniquely determined. 
        The middle term (\ref{eq:H1}) is the direct sum of two parts for $s=0$ ($\mu_{2}> 0$) and $s=1$ ($\mu_{1}> \mu_{2}$)
        \begin{align}
            \bigoplus_{\mu_{1} \geq \mu_{2}> 0} \bigoplus_{t} \bS^{(\la_{1}+\mu_{1}-k-i-t, -1,\cdots, -1, -1-\mu_{2}+1,-1-\mu_{1})} (\bC^{N})^{\vee}, \label{eq:i+1} \\
            \bigoplus_{\mu_{1} > \mu_{2} \geq 0} \bigoplus_{t}  \bS^{(\la_{1}+\mu_{1}-k-i-t, -1,\cdots, -1, -1-\mu_{2},-1-\mu_{1}+1)} (\bC^{N})^{\vee}. \label{eq:i-1}
        \end{align}
        By a change of variable $\mu_{1} \mapsto \mu_{1}-1$ in (\ref{eq:i+1}), a change of variable $\mu_{1} \mapsto \mu_{1}+1$ in (\ref{eq:i-1}), and rearranging the direct sums over $t$, we see that the direct sum of  (\ref{eq:i+1}) and  (\ref{eq:i-1}) is equivalent to the combination of the formula (\ref{eq:H2}) for  $i+1$ and the formula (\ref{eq:H0}) for $i-1$. 
        In other words, we have an isomorphism of $G$-modules
        \[
        \rH^{k+2i}(\cV_{k,i}) \cong \rH^{k+2(i-1)+1}(\cV_{k,i-1}) \oplus \rH^{k+2(i+1)-1}(\cV_{k,i+1}).
        \]
        This `change of variable' is then brought by $\rH^{\bullet}(d)$ via the cup product with the class
        \[d\in H^{1}_{\Gr} \left(\det(V_{2}^{\vee})^{-1} \otimes \extp^{2}(\bC^{N}/V_{2})^{\vee} \otimes \Sym^{1}(V_{2}^{\vee} \otimes  \bC^{N}/V_{2}) \right),\]
        as specified in Lemma~\ref{lem:vanext} and Corollary~\ref{lem:diff}. 
    \end{proof}

    The same argument applies to show the vanishing 
    \[
    \Ext^{>0} (\bT(\cE_{k}),\bT(\cE_{\la}))=0
    \]
    by using Theorem~\ref{thm:flop}, because the groups 
    \[
    \Ext^{\bullet}_{T^{*}\Gr} (\cV_{k,i}^{\prime},\bS^{\la}V_{2}^{\prime}) \cong \bigoplus_{\mu} H^{\bullet}_{\Gr} \left( \bS^{\la} V_{2}^{\prime} \otimes (\cV_{k,i}^{\prime})^{\vee} \otimes \bS^{\mu}V_{2}^{\prime} \otimes \bS^{\mu} V_{N-2}^{\vee} \right)
    \]
    can be computed on $\Gr(N-2,N)$ via exactly the same concatenation weights as in (\ref{eq:conwtlk}).
    The only difference is that the resulting module from the Borel--Weil--Bott theorem is now a Schur functor of $\bC^{N}$ rather than of its dual $(\bC^{N})^{\vee}$, since we would change the order in which the concatenation is formed; see Remark~\ref{rmk:BWB}.    
    \end{proof}

    \begin{lem}\label{lem:van3}
    	We have $\Ext^{>0}(\cE_{k},\cE_{k^{\prime}} ) =0$ for any $k, k^{\prime}$. 
    \end{lem}
    
    \begin{proof}
    Assume $k> k^{\prime}$, and consider
    \[
    \Ext^{\bullet}_{T^{*}\Gr} (\cV_{k,i},\cV_{k^{\prime},j}) \cong \bigoplus_{\mu} H^{\bullet}_{\Gr} \left(\cV_{k,i}^{\vee}\otimes \cV_{k^{\prime},j} \otimes \bS^{\mu}V_{2}^{\vee} \otimes \bS^{\mu} \bC^{N}/V_{2}\right)
    \]
    and the corresponding concatenation weights
    \begin{equation}\label{eq:conwtkkp}
        \left(\mu_{1} + k+i-k^{\prime}-j, \mu_{2} + k+i-k^{\prime}-j, ((1)^{N-3-k-2i}) \otimes ((-1)^{N-3-k^{\prime}-2j}) \otimes -\mu \right). 
    \end{equation}
    When $d := k+i -k^{\prime}-j \geq 0$, the weight (\ref{eq:conwtkkp}) is either already dominant or never in a dominant orbit. 
    On the other hand, if $d<0$, we may permute the coordinate $\mu_{2}+d$ rightward to obtain a nonzero $H^{>0}$.
    However, in this case, as 
    \begin{align*}
        N-3-k-2i - (N-3-k^{\prime}-2j) = -2d + k- k^{\prime} \geq -d +2,
    \end{align*}
    there is at least one coordinate $1$ remaining in (\ref{eq:conwtkkp}) when $\mu_{2} + d$ is permuted by consecutive transpositions into $0$.  
    Hence, the weight (\ref{eq:conwtkkp}) is not in a dominant orbit and there are no higher extensions of $\cV_{k,i}$ by $\cV_{k^{\prime},j}$ for all $i,j$.  
    Consequently, 
    \[\Ext^{>0}(\cE_{k}, \cE_{k^{\prime}}) = 0. \]
    
    To see $\Ext^{>0}(\cE_{k^{\prime}},\cE_{k}) = 0$, we apply the equivalence $\bT$ and prove 
    \[\Ext^{>0}(\bT(\cE_{k^{\prime}}), \bT(\cE_{k})) = 0 \] 
    on $\TGr{N-2}{N}$.
    By Theorem~\ref{thm:flop}, this amounts to compute     
    \[
    \Ext^{\bullet}_{T^{*}\Gr} (\cV_{k^{\prime}, j}^{\prime}, \cV_{k,i}^{\prime}) \cong \bigoplus_{\mu} H^{\bullet}_{\Gr} \left( (\cV_{k^{\prime}, j}^{\prime} )^{\vee} \otimes  \cV_{k,i}^{\prime} \otimes \bS^{\mu} V_{2}^{\prime}  \otimes \bS^{\mu} V_{N-2}^{\vee}\right).
    \]
    The previous argument applies without any change since over $\Gr(N-2,N)$, these groups are calculated by  the same concatenation weights as in (\ref{eq:conwtkkp}); see Remark~\ref{rmk:BWB}.

    It remains to deal with the self-extensions of $\cE_{k}$. 
    In the proof of Lemma~\ref{lem:vanext}, we have shown that the only nontrivial $\Ext^{>0}$ between the bundles $\{\cV_{k,i}\}_{i}$ is the $\Ext^{1}$ of each $\cV_{k,i}$ by $\cV_{k,i+1}$.  
    Moreover, we have chosen the differential $d: \cV_{k,i} [-1] \to \cV_{k,i+1}$ to be the unique $G$-equivariant one. 
    The convolution $\cE_{k}$ is formed as the iterated cones 
    \[
    \cC_{\alpha} := \on{Conv}\left(\cV_{k,n_{k}^{+}-\alpha}[-\alpha] \to \cdots \to \cV_{k,n_{k}^{+}}\right), \quad \alpha = 1,\cdots, n_{k}^{+} + n_{k}^{-}. 
    \]
    We will use induction to show that every $\Ext^{>0}(\cC_{\alpha},\cC_{\alpha}) = 0$.  
    By definition, there is a distinguished triangle
    \[
    \cV_{k,n_{k}^{+}-\alpha} [-1] \xlongrightarrow{d} \cC_{\alpha -1} \longrightarrow \cC_{\alpha},
    \]
    where $d$ is essentially the differential $\cV_{k,n_{k}^{+}-\alpha} [-1] \to \cV_{k, n_{k}^{+}-\alpha +1}$. 
    Now, the functor $\Ext^{>0}(\cC_{\alpha-1},-)$ vanishes on $\cC_{\alpha-1}$ by induction, and it also vanishes on  $\cV_{k,n_{k}^{+}-\alpha}$ by using the fact that $\Ext^{>0} (\cV_{k,i},\cV_{k,j}) =0$ for $i>j$. 
    So we deduce $\Ext^{>0} (\cC_{\alpha-1}, \cC_{\alpha}) = 0$ from the corresponding long exact sequence. 
    On the other hand, we have   
    \[
    \Ext^{\geq 2} (\cV_{k,n_{k}^{+}-\alpha}, \cC_{\alpha-1}) = 0 , \quad \Ext^{\geq 1} (\cV_{k,n_{k}^{+}-\alpha}, \cV_{k,n_{k}^{+}-\alpha}) = 0,
    \]
    which immediately imply $\Ext^{\geq 2} (\cV_{k,n_{k}^{+}-\alpha}, \cC_{\alpha}) = 0$. 
    Thus it remains to verify that the connecting morphism 
    \[
    \cdots \longrightarrow \End (\cV_{k,n_{k}^{+}-\alpha}) \xlongrightarrow{\delta} \Ext^{1} (\cV_{k,n_{k}^{+}-\alpha}, \cC_{\alpha-1} ) \longrightarrow \cdots 
    \]
    is surjective.
    As it is induced by the functor $\Hom(\cV_{k,n_{k}^{+}-\alpha},- )$, the map $\delta$ assigns to $\on{id}$  the differential $d$ and realizes $\Ext^{1}(\cV_{k,n_{k}^{+}-\alpha}, \cC_{\alpha-1})$ as a module over the algebra $\End(\cV_{k,n_{k}^{+}-\alpha})$. 
    This module is in fact isomorphic to $\Ext^{1}(\cV_{k,n_{k}^{+}-\alpha}, \cV_{k,n_{k}^{+}-\alpha+1})$, which is computed explicitly in the proof of Lemma~\ref{lem:vanext} as  
    \[
    \bigoplus_{n} H^{1}_{\Gr} \left(\det(V_{2}^{\vee})^{-1} \otimes \extp^{2}(\bC^{N}/V_{2})^{\vee} \otimes \Sym^{n}V_{2}^{\vee} \otimes \Sym^{n} \bC^{N}/V_{2}\right). 
    \]
    Here the factors $\Sym^{n}V_{2}^{\vee}\otimes \Sym^{n}\bC^{N}/V_{2}$ come from the tensor algebra structure of 
    \[
    \bigoplus_{n} H^{0}_{\Gr} \left(\Sym^{n}V_{2}^{\vee}\otimes \Sym^{n}\bC^{N}/V_{2}\right) \subset \End(\cV_{k,n_{k}^{+}-\alpha}),
    \]
    so $\Ext^{1}(\cV_{k,n_{k}^{+}-\alpha}, \cC_{\alpha-1})$ is generated by $d$ as an $\End(\cV_{k,n_{k}^{+}-\alpha})$-module. 
\end{proof}

\begin{proof}[Proof of Theorem~\ref{thm:tilting}]
	Combining Corollary~\ref{cor:gen}, Lemma~\ref{lem:van1}, Lemma~\ref{lem:van2}, and Lemma~\ref{lem:van3}, we prove that the $\cE$ is a tilting bundle. 
\end{proof}

\subsection{Koszulity}\label{sec:Koszul}
We now study the noncommutative algebra $A = \End(\cE)$. 
Let $\Db(A)$ be the bounded derived category of finitely generated right modules over $A$.
Recall we have introduced a grading on $A$ that is induced by a natural $\Cx$-action on the cotangent bundle of the Grassmannian in \S\ref{sec:grading}. 

\begin{defn}[\cite{BGS}]
    A graded ring $R = \bigoplus_{j\geq 0} R_{j}$ is \textit{Koszul} if $R_{0}$ is semisimple and for any two irreducible graded $R$-modules $L_{1}$, $L_{2}$ concentrated in degree zero, we have $\Ext^{i}_{R}(L_{1},L_{2})_{j} =0$ for $i\neq j$. 
\end{defn}

\begin{thm}\label{thm:Koszul}
    The algebra $A$ is non-negatively graded and its degree zero part consists of scalar multiples of the identityentity
    \[
    \Ext^{0}(\cE_{a},\cE_{b})_{0} \cong  \delta_{a,b} \bC, \quad \forall a,b. 
    \]
\end{thm}

\begin{proof}
	We will calculate the degrees case by case in Lemma~\ref{lem:K1}, Lemma~\ref{lem:K2}, and Lemma~\ref{lem:K3}. 
\end{proof}

\begin{lem}\label{lem:K1}
	For any $\la, \la^{\prime}$, we have 
	\[
	\Ext^{0}(\cE_{\la},\cE_{\la^{\prime}}) = \Ext^{0}(\cE_{\la},\cE_{\la^{\prime}})_{\geq 0},\quad  \Ext^{0}(\cE_{\la},\cE_{\la^{\prime}})_{0} \cong \delta_{\la,\la^{\prime}}\bC. 
	\]
\end{lem}

\begin{proof}
	We first compute the degrees of 
	\[
	\Ext^{0}_{T^{*}\Gr} (\cE_{\la},\cE_{\la^{\prime}}) \cong \bigoplus_{\mu} H^{0}_{\Gr} \left( \bS^{\la}V_{2}^{\vee}\otimes \bS^{\la^{\prime}}V_{2} \otimes \bS^{\mu}V_{2}^{\vee} \otimes \bS^{\mu} \bC^{N}/V_{2}\right),
	\]
	where the direct sum is over all Young diagrams $\mu = (\mu_{1},\mu_{2})$. 
	By the Littlewood--Richardson rule, we have a decomposition of the tensor product 
	\[
	\bS^{\la}V_{2}^{\vee}\otimes \bS^{\la^{\prime}}V_{2} \cong \bigoplus_{t} \bS^{\la_1 - \la_{2}^{\prime} -t, \la_{2} - \la_{1}^{\prime} +t} V_{2}^{\vee},
	\]
	where $0\leq t\leq \min\{\la_1 -\la_2 , \la_{1}^{\prime} - \la_{2}^{\prime}\}$. 
	When $\la_{2}-\la_{1}^{\prime}+t\geq 0$, the cohomology group is nonvanishing, and it has degree 
	\[
	(\la_1 - \la_{2}^{\prime} -t) + (\la_{2} - \la_{1}^{\prime} +t) + 2|\mu| \geq 0. 
	\]
	The degree equals zero only if $\mu=0$, that is to say, the cohomology is taken over the underlying Grassmannian  $H^{0}_{\Gr}(\bS^{\la}V_{2}^{\vee}\otimes \bS^{\la^{\prime}}V_{2})$. 
	The $\Gr(2,N)$ is a GIT quotient of the vector space $\Hom(V_2, \bC^N)$, with semistable locus $\{a: V_{2}\to \bC^{N} \mid \on{rk}(a) =2\}$.
	Its complement in the linear stack $[\Hom(V_{2},\bC^{N})/\GL(2)]$ has codimension at least two, and the local cohomology $H^{0}, H^{1}$ relative to that complement vanishes \cite[\S 2.5]{Cautis}.
	Thus, from the long exact sequence of relative cohomology, the group $H^{0}$ is identical either computed over the Grassmannian or over the linear stack.
	In the latter case, it is equivalent to the invariants 
	\[
	\bigoplus_{n\geq 0} ( \bS^{\la}V_{2}^{\vee}\otimes \bS^{\la^{\prime}}V_{2} \otimes  \Sym^{n} (V_{2} \otimes \bC^{\vee}) )^{\GL(2)}. 
	\]
	The degree $n=0$ piece is one-dimensional when $\la = \la^{\prime}$ and zero otherwise; see \cite[Lemma A.13]{DS}.
	
	On the other hand, when $\la_{2}-\la_{1}^{\prime}+t < 0$, the cohomology group is nonvanishing only if $|\mu| \geq -(\la_{2}-\la_{1}^{\prime}+t)$.
	The corresponding degree is
	\[
	(\la_1 - \la_{2}^{\prime} -t) + (\la_{2} - \la_{1}^{\prime} +t) + 2|\mu| \geq \la_1 -\la_{2} + \la_{1}^{\prime} - \la_{2}^{\prime} -2t \geq 0. 
	\]
	When $\la_1 - \la_{2}^{\prime} -t < 0$, the first inequality has to be strict so that the cohomology is nonvanishing. 
	When $\la_1 - \la_{2}^{\prime} -t \geq 0$, the second inequality must be strict for the following reason: the expression $\la_1 -\la_{2} + \la_{1}^{\prime} - \la_{2}^{\prime} -2t$ is zero iff $\la_1 -\la_{2} = t = \la_{1}^{\prime} - \la_{2}^{\prime}$, which yields a contradiction between $\la_{2}-\la_{1}^{\prime}+t < 0$ and $\la_1 - \la_{2}^{\prime} -t \geq 0$.
\end{proof}

\begin{lem}\label{lem:K2}
	For any $\la, k$, we have 
	\[
	\Ext^{0}(\cE_{\la},\cE_{k}) = \Ext^{0}(\cE_{\la},\cE_{k})_{> 0}, \quad \Ext^{0}(\cE_{k},\cE_{\la}) = \Ext^{0}(\cE_{k},\cE_{\la})_{> 0}.
	\]
\end{lem}

\begin{proof}
	We first compute the degrees of each $\Ext^{0}(\cE_{\la},\cV_{k,i}) = H^{0}_{T^{*}\Gr}(\bS^{\la}V_{2}^{\vee} \otimes \cV_{k,i})$, which amounts to consider global sections $H^{0}_{\Gr}(-)$ of 
	\[
	F_{\mu} := \bS^{\la}V_{2}^{\vee} \otimes \cV_{k,i} \otimes \bS^{\mu}V_{2}^{\vee} \otimes \bS^{\mu} \bC^{N}/V_{2}
	\]
	for all Young diagrams $\mu = (\mu_{1},\mu_{2})$.
	The tensor product $\bS^{\la}V_{2}^{\vee} \otimes \det(V_{2})^{k+i}\otimes \bS^{\mu}V_{2}^{\vee}$ has a decomposition into Schur functors of $V_{2}$ with highest weights $(\la_{1}-k-i+\mu_{1}-t, \la_{2}-k-i+\mu_{2}+t)$ for $0\leq t\leq \on{min}\{\la_{1}-\la_{2}, \mu_{1}- \mu_{2}\}$.
	By the Borel--Weil--Bott theorem, the group $H^{0}_{\Gr}(F_{\mu})$ is computed through the concatenation weights
	\begin{equation}\label{eq:Koszul-Borel--Weil--Bott}
		(\la_{1}-k-i+\mu_{1}-t,\la_{2}-k-i+\mu_{2}+t,\underbrace{0,\cdots,0}_{k+2i+1},\underbrace{-1,\cdots,-1}_{N-3-k-2i})
	\end{equation}
	for all $t$.
	\begin{enumerate}[wide, labelindent=0pt, label = (\roman*)]
		\item When $k+2i+1 \geq 2$, we claim that $H^{0}_{\Gr}(F_{\mu})$ always has positive degrees.
		The degree of $H^{0}_{\Gr}(F_{\mu})$ (if nonvanishing) is given by
		\begin{equation}\label{eq:wt}
			|\la|-2(k+i)+1 + 2|\mu|
		\end{equation}
		by our choice (\ref{eq:shift}) of the internal degree shift of $\cV_{k,i}$. 
		If $\la_{2}-k-i \geq 0$, $H^{0}_{\Gr}(F_{0})$ is already nonvanishing, and its degree is obviously positive.
		Now assume $\la_{2}-k-i<0$ and (\ref{eq:Koszul-Borel--Weil--Bott}) is dominant. 
		\begin{enumerate}[label = (\alph*)]
			\item If $k+2i+1>2$, or $k+2i+1=2$ and $\la_{2}-k-i = -1$, the least $|\mu|$ occurs when $\la_{2}-k-i+\mu_{2}+t = 0$. 
			\item If $k+2i+1=2$ and $\la_{2}-k-i = -2$, the least\footnote{i.e.\ $\mu=(1,1)$ or $\mu=(2,0)$ (when $\la_1 -\la_2 \geq 2$).} $|\mu| =2$.
			\item If $k+2i+1=2$ and $\la_{2}-k-i < -2$, the least $|\mu|$ occurs when $\la_{2}-k-i+\mu_{2}+t = -1$ and $\mu_{2} >0$.
		\end{enumerate}
		In all of these cases, it is easy to check that the degree (\ref{eq:wt}) of $H^{0}_{\Gr}(F_{\mu})$ is positive. 
		
		\item When $k+2i+1=1$, i.e.\ $i=-n_{k}^{-}$ and $N+k$ even. 
		The degree (\ref{eq:wt}) is non-negative in general, and the degree zero part $H^{0}_{\Gr}(F_{\mu})_{0} \cong \det(\bC^{N})$ occurs when $\la_{1}-k-i=-1$, $\la_{2}-k-i =-2$ and $\mu = (1,0)$.  
		Now, consider the long exact sequence 
		\[
		\cdots \longrightarrow\Ext^{0}(\cE_{\la}, \cE_{k}) 
		\longrightarrow \Ext^{0}(\cE_{\la}, \cV_{k,-n_{k}^{-}}) \longrightarrow 
		\Ext^{1}(\cE_{\la}, \cC) \longrightarrow 0, 
		\]
		where $\cC = \cC_{n_{k}^{+}+n_{k}^{-}-1}$ is the second last iterated cones in the definition of $\cE_{k}$, taken before the last $\cV_{k,-n_{k}^{-}}$. 
		By our previous calculation (\ref{eq:nonvaext}), we know that $\Ext^{1}(\cE_{\la}, \cC) \cong \Ext^{1}(\cE_{\la}, \cV_{k,1-n_{k}^{-}})$ has non-negative degrees, with its degree zero part also isomorphic to $\det(\bC^{N})$ at $\mu=(1,1)$. 
		Thus, we conclude that the $\tG$-equivariant connecting morphism $\rH^{\bullet}(d)$ induces an isomorphism between the irreducible modules (cf.\ Claim~\ref{clm:conn})
		\begin{equation*}
			\Ext^{0}(\cE_{\la}, \cV_{k,-n_{k}^{-}})_{0} \simrightarrow  \Ext^{1}(\cE_{\la}, \cC)_{0},
		\end{equation*}
		and consequently, the degrees of $\Ext^{0}(\cE_{\la}, \cE_{k})$ are all positive. 
		
		\item When $k+2i+1 =0$, i.e.\ $i=-n_{k}^{-}$ and $N+k$ odd.
		There are two situations in which the degree (\ref{eq:wt}) is non-positive. 
		\begin{enumerate}[label = (\alph*)]
			\item If $\la_{1}-k-i =-1$, $\la_{2}-k-i=-2$ and $\mu=(1,0)$, then  $H^{0}_{\Gr}(F_{\mu})_{0} \cong \det(\bC^{N})\otimes \bC^{N}$.
			\item If  $\la_{1}-k-i =-1$ (resp.\ $0$), $\la_{2}-k-i=-1$ and $\mu=(0,0)$, then $H^{0}_{\Gr}(F_{\mu})_{-1} \cong \det(\bC^{N})$ (resp.\ $H^{0}_{\Gr}(F_{\mu})_{0} \cong  \extp^{N-1} (\bC^{N})$). 
		\end{enumerate}
		Using the same argument, we can show that these irreducible graded modules are isomorphic to $\Ext^{1}(\cE_{\la}, \cC_{n_{k}^{+}+n_{k}^{-}-1})_{0}$ via $\rH^{\bullet}(d)$.
		So, the group $\Ext^{0}(\cE_{\la}, \cE_{k})$ is positively graded. 
	\end{enumerate}
		
	To compute the degrees of $\Ext^{0}(\cE_{k},\cE_{\la})$, we apply the $\tG$-equivariant equivalence $\bT$ using Theorem~\ref{thm:flop}. 
	It then amounts to compute the degrees of the global sections of 
	\[
	F_{\mu}^{\prime} = (\cV_{k,i}^{\prime})^{\vee} \otimes \bS^{\la}V_{2}^{\prime}\langle 2(N-4-\la_{1}-\la_{2})\rangle \otimes \bS^{\mu}V_{N-2}^{\vee} \otimes \bS^{\mu}V_{2}^{\prime}
	\]
	on $\Gr(N-2,N)$.
	Here, $\cV_{k,i}^{\prime}$ is the bundle $\det(V_{2}^{\prime})^{k+i} \otimes \extp^{N-3-k-2i}  V_{N-2}$ shifted by $ \langle N-4-3k-2i \rangle$ and  $\mu=(\mu_{1},\mu_{2})$ is any Young diagram. 
	By Remark~\ref{rmk:BWB}, the group $H^{0}_{\Gr}(F_{\mu}^{\prime})$ is computed through the same concatenation weights as in (\ref{eq:Koszul-Borel--Weil--Bott}), and as we noted before, the only difference is that the resulting module from the Borel--Weil--Bott theorem is now a Schur functor of $\bC^{N}$ rather than of its dual. 
	Taking into account the internal degree shifts, we find that the degree of $H^{0}_{\Gr}(F_{\mu}^{\prime})$ (if nonvanishing) is also given by $|\la|-2(k+i)+1+2|\mu|$. 
	Now, all previous arguments repeat. 
\end{proof}

\begin{lem}\label{lem:K3}
	For any $k, k^{\prime}$, we have 
	\[
	\Ext^{0}(\cE_{k},\cE_{k^{\prime}}) = \Ext^{0}(\cE_{k},\cE_{k^{\prime}})_{\geq 0},\quad  \Ext^{0}(\cE_{k},\cE_{k^{\prime}})_{0} \cong \delta_{k,k^{\prime}}\bC. 
	\]
\end{lem}

\begin{proof}
    We first consider $\Ext^{0}(\cE_{k},\cE_{k^{\prime}})$ with $k>k^{\prime}$. 
    The cohomology $H^{0}_{\Gr}$ of 
    \[
    \cV_{k,i}^{\vee} \otimes \cV_{k^{\prime},j} \otimes \bS^{\mu}V_{2}^{\vee} \otimes \bS^{\mu}\bC^{N}/V_{2}
    \]
    is computed from the concatenation weight as in (\ref{eq:conwtkkp}). 
    It has degree 
    \begin{equation}\label{eq:Koszul-Borel--Weil--Bott2}
    	2(k-k^{\prime} +i-j) + 2|\mu|
    \end{equation} 
    if the cohomology is nonvanishing. 
    The case of $k-k^{\prime} +i-j >0$ is obvious. 
    Otherwise, the concatenation (\ref{eq:conwtkkp}) is dominant only if $|\mu|\geq -2(k-k^{\prime} +i-j) +2 \delta_{k-k^{\prime} +i-j,0}$. 
    A simple calculation shows that the degree (\ref{eq:Koszul-Borel--Weil--Bott2}) is also positive in this case. 
    Thus, each group $\Ext^{0}(\cV_{k,i},\cV_{k^{\prime},j})$ is positively graded. 

    Again, to compute the degrees of $\Ext^{0}(\cE_{k^{\prime}},\cE_{k})$, we apply the equivalence $\bT$ by using Theorem~\ref{thm:flop}.
    The group $H^{0}(\cV_{k^{\prime},j}^{\prime},\cV_{k,i}^{\prime})$ then can be computed from the same concatenation weights as in (\ref{eq:conwtkkp}), and the same argument applies. 

    It remains to show that each $\End(\cE_{k})$ is non-negatively graded and $\End(\cE_{k})_{0} \cong \bC$. 
    The convolution $\cE_{k}$ is the iterated cones 
    \[
    \cC_{\alpha} = \on{Conv}\left(\cV_{k,n_{k}^{+}-\alpha}[-\alpha] \to \cdots \to \cV_{k,n_{k}^{+}}\right), \quad \alpha = 1,\cdots, n_{k}^{+} + n_{k}^{-}. 
    \]
    We will use induction to prove the statement for each $\End(\cC_{\alpha})$.  
    First of all, it is not hard to see that $\Ext^{0}(\cV_{k,i},\cV_{k,j})$ is positively graded, except when $i=j$, where we have $\End(\cV_{k,i})_{0}\cong \bC$. 
    Next, consider the long exact sequence
    \[
    0\longrightarrow \Ext^{0}(\cV_{k,n_{k}^{+}-\alpha}, \cC_{\alpha}) \longrightarrow \End(\cC_{\alpha}) \longrightarrow \Ext^{0}(\cC_{\alpha-1},\cC_{\alpha}) \longrightarrow 0. 
    \]
    Here we used the fact that $\Ext^{1}(\cV_{k,n_{k}^{+}-\alpha},\cC_{\alpha}) = 0$ from the proof of Lemma~\ref{lem:van3}, which is deduced from the surjectivity of the connecting morphism 
    \[
    \delta: \End(\cV_{k,n_{k}^{+}-\alpha}) \longrightarrow \Ext^{1}(\cV_{k,n_{k}^{+}-\alpha}, \cC_{\alpha-1}).
    \]
    It turns out that $\delta$ induces an isomorphism $\bC\simrightarrow \bC$ between the degree zero parts, so $\Ext^{0}(\cV_{k,n_{k}^{+}-\alpha}, \cC_{\alpha})$ is concentrated in positive degrees. 
    On the other hand, we have known that $\Ext^{0}(\cC_{\alpha-1},\cC_{\alpha})$ is positively graded, and by induction, $\End(\cC_{\alpha-1})$ is non-negatively graded with $\End(\cC_{\alpha-1})_{0} \cong \bC$. 
    As a result, we see that $\Ext^{0}(\cC_{\alpha-1},\cC_{\alpha})$, hence also $\End(\cC_{\alpha})$, is concentrated in non-negative degrees, and its degree zero part is one-dimensional. 
\end{proof}

\begin{cor}
	The algebra $A$ is Koszul. 
\end{cor}

\begin{proof}
	Assume $L_{1}$, $L_{2}$ are two irreducible graded $A$-modules concentrated in degree zero.
	Since $A_{0} \cong \bigoplus_{\la,k} \bC $ is semisimple, we have $\Ext^{i}_{A}(L_{1},L_{2})_{j} = 0$ for $i>j$ immediately by \cite[Lemma 2.1.2]{BGS}. 
	Following Kaledin's argument from \cite[\S 5.5]{BM}, we conclude the other half of vanishing $\Ext^{i}_{A}(L_{1},L_{2})_{j} = 0$, $i<j$ by using the Serre duality (induced via the equivalence $\Db(\TGr{2}{N}) \simeq \Db(A)$)
	\[
	\Ext^{i}_{A}(L_{1},L_{2})_{j} \cong \Ext^{d-i}_{A}(L_{2},L_{1})_{d-j}^{\vee}.
	\]
	Here $d=\dim \TGr{2}{N}$ and we used the fact that $\TGr{2}{N}$ has trivial canonical bundle $\cO\langle -d\rangle$ (see, e.g., \cite[Lemma 3.4]{CK1}). 
\end{proof}

\section{Invariance}

\subsection{Flop equivalence}\label{sec:flop}
We begin this section by studying the behaviour of the tilting bundle $\cE$ under the derived equivalence $\bT$ (see \S\ref{sec:sl2}) for the stratified Mukai flop $\TGr{2}{N} \dashrightarrow \TGr{N-2}{N}$. 
As before, the quotient bundle $\bC^{N}/V_{N-2}$ on $\TGr{N-2}{N}$ is denoted by $V_{2}^{\prime}$ . 

\begin{thm}\label{thm:flop}
	We have $\bT (\cE_{\la}) = \bS^{\la}V_{2}^{\prime} \langle 2(N-4-\la_{1}-\la_{2}) \rangle$, and $\bT (\cE_{k})$ is given by the convolution of a complex 
	\begin{equation*}
		\cV_{k,n^{+}_{k}}^{\prime}[-n^{-}_{k}-n^{+}_{k}] \xlongrightarrow{d} \cdots \xlongrightarrow{d}  \cV_{k,i}^{\prime}[-n^{-}_{k}-i] \xlongrightarrow{d} \cdots \xlongrightarrow{d} \cV_{k,-n^{-}_{k}}^{\prime}[0], 
	\end{equation*}
	where $\cV_{k,i}^{\prime} = \det(V_{2}^{\prime})^{k+i} \otimes \extp^{N-3-k-2i}  V_{N-2} \langle N-4-3k-2i \rangle$. 
\end{thm}

Like $\cE_{k}$, the convolution $\bT(\cE_{k})$ is uniquely determined by its $\tG$-equivariant structure, which can be read from the internal degree shift on each $\cV_{k,i}^{\prime}$.
We shall prove this theorem in the rest of this subsection by spelling out the Rickard complex on these bundles.
The proof is based on various Lascoux resolutions of the images of the functors $\be$ and $\bbf$, which we place at the end (\S\ref{sec:cal}).

\begin{lem}\label{lem:counit}
	The counit $\varepsilon: \bbf^{(1)}\be^{(1)}[3-N]\langle N-3 \rangle \to \on{id}$ on $\cE_{\la}$ (with $\la_{2}=-1$) is given by an isomorphism $\cE_{\la} \simrightarrow \cE_{\la}$.
\end{lem} 

\begin{proof}
	Assume $\la_{2}=-1$. 
	By Lemma~\ref{lem:eimage}, we have $\be^{(1)}(\cE_{\la}) = \bS^{\la_{1}+2}V_{1} \otimes \det(\bC^{N})^{\vee} \langle 1 \rangle$. 
	By Lemma~\ref{lem:fimage} further, we know that $\bbf^{(1)}\be^{(1)} (\cE_{\la}) [3-N]\langle N-3 \rangle$ has a resolution 
	\[\cF^{-n} = \begin{cases}
		\bS^{-1-n,\la_{1}+1} V_{2} \otimes \extp^{1-n} (\bC^{N}/V_{2})^{\vee}\langle 2-2n\rangle, & 3-N\leq n \leq -\la_{1}-2,\\
		\bS^{\la_{1},-1-n} V_{2} \otimes \extp^{-n} (\bC^{N}/V_{2})^{\vee} \langle -2n \rangle, & -\la_{1}-1 \leq n \leq 0. 
	\end{cases}\]
	By a general result \cite[Lemma~1.6]{Kap2}, the morphism $\varepsilon: \cF^{\bullet} \to \cE_{\la}$ will be given by a map of chain complexes once we check that
	$\Ext^{>0}(\cF^{-n},\cE_{\la}) = 0$ for all $n$. 
	This is computed by
	\[
	\Ext^{\bullet}_{T^{*}\Gr} (\cF^{-n}, \cE_{\la}) \cong \bigoplus_{\mu} H^{\bullet}_{\Gr} \left( (\cF^{-n})^{\vee} \otimes \cE_{\la}\otimes \bS^{\mu} V_{2}^{\vee} \otimes \bS^{\mu} \bC^{N}/V_{2}\right)
	\]
	over all Young diagrams $\mu=(\mu_{1},\mu_{2})$. 
	We first consider the case $-\la_{1}-1\leq n \leq 0$. 
	By the Littlewood--Richardson rule, we have a decomposition 
	\[
	\bS^{\la_{1},-1-n}V_{2}^{\vee} \otimes \bS^{1,-\la_{1}}V_{2}^{\vee} \otimes \bS^{\mu}V_{2}^{\vee} \cong \bigoplus_{\nu} c_{\nu} \bS^{\nu}V_{2}^{\vee},
	\]
	where each $\nu$ satisfies $\nu_{2}\geq -1-n-\la_{1} +\mu_{2} \geq 3-n-N+\mu_{2}$.
	To apply the Borel--Weil--Bott theorem, we consider concatenation weights of the form
	\begin{equation}\label{eq:conwtunit}
		(\nu, ((-1)^{-n})\otimes -\mu).  
	\end{equation}
	Here we use the notation $((-1)^{m})$ to denote the weight that has $- 1$ on its last $m$ coordinates and zero elsewhere, and we consider all possible weights appearing in the decomposition of the tensor product $\extp^{-n} (\bC^{N}/V_{2}) \otimes \bS^{\mu} \bC^{N}/V_{2}$ according to Pieri's formula.  
	When the last coordinate of $\nu$ is permuted rightward to become $-1$, there is still at least one digit $0$ after it in (\ref{eq:conwtunit}) since there are at least $N-3+n$ (if $\mu_{2}=0$) or $N-4+n$ (if $\mu_{2}\neq 0$) zeros in $((-1)^{-n}\otimes -\mu)$ while $\nu_{2}\geq 3-n-N+\mu_{2}$. 
	For this reason, the weight (\ref{eq:conwtunit}) can never be permuted into a dominant weight. 
	Similarly, for $3-N\leq n \leq -\la_{1}-2$, there is a decomposition 
	\[\bS^{-1-n,\la_{1}+1} V_{2}^{\vee} \otimes \bS^{1,-\la_{1}} V_{2}^{\vee} \otimes \bS^{\mu}V_{2}^{\vee}   \cong \bigoplus_{\nu} c_{\nu} \bS^{\nu} V_{2}^{\vee},
	\] 
	where each $\nu_{2}\geq 1+\mu_{2}>0$. 
	Thus, the weights $(\nu, ((-1)^{1-n}) \otimes -\mu)$ are always dominant, and there are no higher extensions of $\cF^{-n}$ by $\cE_{\la}$.  
	In conclusion, the morphism $\varepsilon: \cF^{\bullet} \to \cE_{\la}$ is given by a single map $\iota: \cF^{0} = \cE_{\la} \to \cE_{\la}$. 
	Moreover, the map is nonzero because otherwise the differential $d_{1}:\Theta_{1}[-1] \to \Theta_{0}$ which is built upon $\varepsilon$ will vanish on $\cE_{\la}$.
	In that case, the image of the indecomposable $\cE_{\la}$ under the equivalence $\bT$ will become decomposable, which is absurd. 
	We next show that this map is actually a scalar multiple of the identity.  
	
	Recall that $\pi_{2}:\fZ_{1,2}\to \TGr{2}{N}$ is the projection from the Hecke correspondence (\S\ref{sec:sl2}). 
	Using the adjunction $\pi_{2,*} \dashv \pi_{2}^{!}$, we see that $\iota$ actually comes from a map 
	\begin{equation}\label{eq:iotamap}
		\bS^{\la_{1}+1} V_{1} \longrightarrow \bS^{\la_{1}+1,0} V_{2} 
	\end{equation}
	over $\fZ_{1,2}$. 
	Here we used the formula
	\begin{align*}
		\omega_{\pi_{2}} = \det(V_{1})^{4-N}\otimes \det(V_{2})^{N-2} \otimes \det(\bC^{N})^{-1} [3-N]\langle 2N-4\rangle,
	\end{align*}
	which can be found in \cite[p.96]{CKL-sl2} or \cite[\S 4.2]{CKL}.  
	In fact, the map (\ref{eq:iotamap}) is the tautological one because when we resolved $\be^{(1)}(\cE_{\la})$ in the proof of Lemma~\ref{lem:eimage}, we calculated the cohomology over the associated graded pieces (\ref{eq:assocgr}) of $\cE_{\la}$ arising from the tautological sequence 
	$0 \to V_{1} \to V_{2} \to V_{2}/V_{1} \to 0 $.
	In that calculation, the only nonvanishing term comes from the piece $\bS^{\la_{1}}V_{1} \otimes \bS^{-1} V_{2}/V_{1}$ for $d=0$ in (\ref{eq:assocgr}). 
	
	Thus, the map (\ref{eq:iotamap}) is constant over the tautological direction $V_{2}\hookrightarrow \bC^{N}$, and so is $\iota$. 
	Following \cite[Lemma A.13]{DS}, we consider the subspace of maps in $\Ext^{0}_{\fZ_{1,2}}(\cE_{\la},\cE_{\la})$ that are constant over $\Gr(V_{2},\bC^{N})$, and find that it is isomorphic to the invariants 
	\begin{equation*}
		\left( \cE_{\la}^{\vee}\otimes \cE_{\la} \otimes \Sym \left( V_{1}\otimes V_{2}^{\vee} \oplus (\bC^{N}/V_{2})\otimes V_{1}^{\vee} \right) \right)^{\GL(V_{1})\times \GL(V_{2})} \cong \bC. 
	\end{equation*}
	The nonzero map $\iota$ then must be an isomorphism. 
\end{proof}

\begin{lem}\label{lem:la}
	For any $\la$, we have $\bT (\cE_{\la}) = \bS^{\la}V_{2}^{\prime} \langle 2(N-4-\la_{1}-\la_{2}) \rangle$. 
\end{lem}

\begin{proof}
	When $\la_{2}\neq -1$, we have $\be^{(1)} (\cE_{\la}) =0$ by Lemma~\ref{lem:eimage}.
	The Rickard complex then degenerates to only one term $\Theta_{0}$, which is $\bbf^{(N-4)}(\cE_{\la}) = \bS^{\la}V_{2}^{\prime} \langle 2(N-4-\la_{1}-\la_{2})\rangle$ by Lemma~\ref{lem:fil}. 
	When $\la_{2}=-1$, we know from Lemma~\ref{lem:eimage} that $\be^{(2)}(\cE_{\la})=0$ and $\be^{(1)}(\cE_{\la}) = \bS^{\la_{1}+1}V_{1}\otimes \det(\bC^{N}/V_{1})^{-1} \langle 1\rangle$.
	So, the image of $\cE_{\la}$ under $\bT$ is the cone 
	\[
	\on{Cone}\left( \bbf^{(N-3)}\left(\bS^{\la_{1}+1}V_{1}\otimes \det(\bC^{N}/V_{1})^{-1}\langle 1\rangle \right)[-1]\langle 1\rangle \xlongrightarrow{d_{1}} \bbf^{(N-4)}(\cE_{\la})\right). 
	\]
	As $-1\leq \la_{1} \leq N-4$, we can apply Lemma~\ref{lem:fimage} to $\bS^{\la_{1}+2}V_{1}$, take the tensor product of the resolution of $\bbf^{(n-3)}(\bS^{\la_{1}+2}V_{1})$ thus obtained with $\det(\bC^{N})^{-1}$, and finally shift it by $\langle 2\rangle [-1]$. 
	In this way, we resolve the sheaf on the left-hand side in the above cone by
	\begin{equation*}
		\cF_{1}^{-n} = \begin{cases}
			\bS^{\la_{1},n} V_{2}^{\prime} \otimes \extp^{n+1} V_{N-2}^{\vee} \langle 2(N-4-\la_{1}-n) \rangle, & -1 \leq n \leq \la_{1},\\
			\bS^{n,\la_{1}+1}V_{2}^{\prime} \otimes \extp^{n+2} V_{N-2}^{\vee} \langle 2(N-5-\la_{1}-n) \rangle,  & \la_{1}+1\leq n \leq N-4. 
		\end{cases}
	\end{equation*}
	By Lemma~\ref{lem:fil}, the sheaf on the right-hand side in the cone has a resolution with exactly the same terms in non-positive degrees
	\begin{equation*}
		\cF_{0}^{-n} = \begin{cases}
			\bS^{\la_{1},n} V_{2}^{\prime} \otimes \extp^{n+1} V_{N-2}^{\vee} \langle 2(N-4-\la_{1}-n) \rangle, & 0 \leq n \leq \la_{1},\\
			\bS^{n,\la_{1}+1}V_{2}^{\prime} \otimes \extp^{n+2} V_{N-2}^{\vee} \langle 2(N-5-\la_{1}-n) \rangle,  & \la_{1}+1\leq n \leq N-4. 
		\end{cases}
	\end{equation*}
	We claim that the differential $d_{1}: \cF_{1}^{\bullet} \to \cF_{0}^{\bullet}$ is given by isomorphisms $\cF_{1}^{-n} \simrightarrow \cF_{0}^{-n}$ at each $n\geq 0$. 
	Recall that $d_{1}: \bbf^{(N-3)}\be^{(1)}[-1]\langle 1\rangle \to \bbf^{(N-4)}$ is defined by including $\bbf^{(N-3)}$ into the lowest degree factor $\bbf^{(N-3)}[N-4]\langle 4-N \rangle$ in the decomposition 
	\[\bbf^{(N-4)} \bbf^{(1)} \simeq \bbf^{(N-3)}\otimes H^{*}(\bP^{N-4}),
	\]
	followed by applying the counit $\varepsilon: \bbf^{(1)}\be^{(1)}[3-N]\langle N-3 \rangle\to \on{id}$. 
	The map $\varepsilon$ is a (nonzero) scalar multiple of the identity (on $\cE_{\la}$) by Lemma~\ref{lem:counit}, hence so is the morphism $\bbf^{(N-4)}(\varepsilon)$. 
	As $d_{1}$ is the inclusion into the lowest degrees, it is a scalar multiple of the identity on each $\cF_{1}^{-n}$ for $n\geq 0$. 
	
	Given this claim, the terms of $\cF_{1}^{\bullet}$ and $\cF_{0}^{\bullet}$ in non-positive degrees are cancelled in the cone, which leaves
	\[
	\on{Cone} \left(\cF_{1}^{\bullet}\xlongrightarrow{d_{1}} \cF_{0}^{\bullet} \right) = \cF_{1}^{1} = \bS^{\la_{1},-1}V_{2}^{\prime} \langle 2(N-4-\la_{1}-(-1)) \rangle.
	\]
\end{proof}

Next we calculate the image of $\cE_{k}$ under $\bT$. 
First note that except for $k=0,-1$ where there is no term $\cV_{k,-1}$ in the complex (\ref{eq:complex}), we have 
\[
\be^{(1)}(\cV_{k,0}) = \bS^{k}V_{1}[-k]\langle k+1-N \rangle = \be^{(1)} (\cV_{k,-1}[-1]) 
\]
by Lemma~\ref{lem:eimage}.
As $d:\cV_{k,-1}[-1]\to \cV_{k,0}$ is $\tG$-equivariant, the induced map $\be^{(1)}(d): \bS^{k}V_{1} \to \bS^{k}V_{1}$ is also $\tG$-equivariant hence a scalar multiple of the identity.  
Moreover, it is nonzero because otherwise the counit $\varepsilon$ will vanish on $d$, and the equivalence $\bT$ will map the nontrivial class $d$ into a trivial one, leading to a contradiction.
So, in the cone $\on{Cone}(\bT(d))$, the terms $\Theta_{1}(\cV_{k,-1}[-1])$ and $\Theta_{1}(\cV_{k,0})$ are cancelled. 
We also know from Lemma~\ref{lem:eimage} that $\be^{(1)}$ vanishes on all other $\cV_{k,i}$, $i\neq 0,-1$, hence $\bT(\cE_{k})$ is reduced to the convolution
\begin{align}
	& \on{Conv} \left(\Theta_{0}(\cV_{k,-n_{k}^{-}})[-n_{k}^{+} -n_{k}^{-}] \to \cdots \to \Theta_{0}(\cV_{k,n_{k}^{+}})\right), \quad k\neq 0,-1, \label{eq:itcone1} \\
	& \on{Conv} \left(\bT(\cV_{k,0})[-n_{k}^{+}] \to \Theta_{0}(\cV_{k,1})[-n_{k}^{+} +1] \to \cdots \to \Theta_{0}(\cV_{k,n_{k}^{+}})\right), \,\, k=0,-1. \label{eq:itcone2}
\end{align}
In order to calculate the image of $\Theta_{0} = \bbf^{(N-4)}$ on $\cV_{k,i}$, we use the filtration  (\ref{eq:filtration})
\begin{align*}
	0\longrightarrow F^{1}\cV_{k,i} &\longrightarrow \cV_{k,i} \longrightarrow F^{2/1}\cV_{k,i}\longrightarrow 0,\\
	0\longrightarrow F^{0}\cV_{k,i} &\longrightarrow F^{1}\cV_{k,i} \longrightarrow F^{1/0}\cV_{k,i}\longrightarrow 0
\end{align*}
on the correspondence $\fZ_{2,N-2}$.
The image $\bbf^{(n-4)}(\cV_{k,i})$ is thus realized as an iterated extension sheaf
\[
\on{Conv}\left( \pi_{N-2,*}(F^{2/1}\cV_{k,i})[-2] \longrightarrow \pi_{N-2,*}(F^{1/0}\cV_{k,i})[-1] \longrightarrow \pi_{N-2,*}(F^{0}\cV_{k,i})\right). 
\]

\begin{lem}
	On $\fZ_{2,N-2}$, the differential $d:\cV_{k,i}[-1] \to \cV_{k,i+1}$ is uniquely determined by two $\tG$-equivariant maps between the associated graded pieces
	\[
	\iota_{2/1}:F^{2/1}\cV_{k,i}[-2] \longrightarrow F^{1/0}\cV_{k,i+1}[-1], \quad 
	\iota_{1/0}:F^{1/0}\cV_{k,i}[-1] \longrightarrow F^{0}\cV_{k,i+1}. 
	\]
\end{lem}

\begin{proof}
	The differential $d\in\Ext^{1}(\cV_{k,i},\cV_{k,i+1})$ is the unique $\tG$-equivariant class taken over $\TGr{2}{N}$. 
	As $\pi_{2,*}\cO =\cO$ along the projection $\pi_{2}:\fZ_{2,N-2}\to \TGr{2}{N}$, the $\tG$-invariants are also one-dimensional over $\fZ_{2,N-2}$ by the projection formula. 
	In fact, the correspondence $\fZ_{2,N-2}$ is the total space of $\Hom(V_{2}^{\prime}, V_{2})$ over the partial flag variety $\mathrm{Fl}(2,N-2;N)$, and the extension group can be computed by applying the Borel--Weil--Bott Theorem (\cite[\S 4.1]{Wey}) to the cohomology
	\[
	H^{\bullet}_{\on{Fl}}\left( \cV_{k,i}^{\vee}\otimes \cV_{k,i+1} \otimes \Sym(V_{2}^{\prime}\otimes V_{2}^{\vee}) \right) .
	\]
	A direct calculation shows that the only $\tG$-equivariant extensions between the associated graded pieces of $\cV_{k,i}$ and $\cV_{k,i+1}$ are
	\begin{equation*}
		\Ext^{1}(F^{2/1}\cV_{k,i},F^{1/0}\cV_{k,i+1})^{\tG} \cong \bC, \quad \Ext^{1} (F^{1/0}\cV_{k,i},F^{0}\cV_{k,i+1})^{\tG} \cong \bC, 
	\end{equation*}
	and $\Ext^{2}(F^{2/1}\cV_{k,i},F^{0}\cV_{k,i+1})^{\tG} \cong \bC$.
	This implies $\Ext^{\bullet}(\cV_{k,i},F^{2/1}\cV_{k,i+1})^{\tG} = 0$, hence
	\[
	\bC  d  = \Ext^{1}(\cV_{k,i},\cV_{k,i+1})^{\tG} \cong \Ext^{1}(\cV_{k,i},F^{1}\cV_{k,i+1})^{\tG} .
	\]
	Moreover, from the long exact sequence of the filtration $0\to F^{0}\to F^{1} \to F^{1/0} \to 0$, we obtain
	\begin{align*}
		\Ext^{1}(F^{1}\cV_{k,i},F^{1}\cV_{k,i+1})^{\tG} \cong & \Ext^{1}(F^{1}\cV_{k,i},F^{0}\cV_{k,i+1})^{\tG}  =: \bC \iota_{1/0} \\
		\cong & \Ext^{1}(F^{1/0}\cV_{k,i},F^{0}\cV_{k,i+1})^{\tG}.
	\end{align*}
	In summary, there is a commutative diagram of triangles
	\begin{equation}\label{eq:fildia}
		\begin{tikzcd}
			F^{2/1}\cV_{k,i}[-2] \arrow[r] \arrow[d, "\iota_{2/1}"] & 
			F^{1}\cV_{k,i}[-1] \arrow[r] \arrow[d, "\iota_{1/0}"] &
			\cV_{k,i}[-1] \arrow[d, "d"] \\
			F^{1/0} \cV_{k,i+1}[-1] \arrow[r] & 
			F^{0} \cV_{k,i+1} \arrow[r] &
			F^{1}\cV_{k,i+1} ,
		\end{tikzcd}
	\end{equation}
	where $\iota_{2/1}$ is determined up to a scalar multiple. 
\end{proof}

Given (\ref{eq:fildia}), the cone of $\pi_{N-2,*}(d)$ can be realized as the totalization 
\begin{equation}\label{eq:total}
	\on{Conv}\left(  
	\pi_{N-2,*}(F^{2/1}\cV_{k,i})[-2] \longrightarrow
	\substack{\pi_{N-2,*}(F^{1/0}\cV_{k,i+1})[-1]\\ 
		\oplus \\
		\pi_{N-2,*}(F^{1/0}\cV_{k,i})[-1]} \longrightarrow 
	\pi_{N-2,*}(F^{0}\cV_{k,i+1})
	\right). 
\end{equation}
In Lemma~\ref{lem:fil}, we found locally free resolutions of these sheaves $\pi_{N-2,*}(F^{j/j-1}\cV_{k,i})$. 
We now use them to spell out this convolution. 

\begin{lem}\label{lem:neq0-1}
	For $k\neq 0,-1$, we have 
	\[
	\bT(\cE_{k}) = \on{Conv} \left(\cV_{k,n^{+}_{k}}^{\prime}[-n^{-}_{k}-n^{+}_{k}] \to \cdots \to  \cV_{k,i}^{\prime}[-n^{-}_{k}-i] \to \cdots \to \cV_{k,-n^{-}_{k}}^{\prime} \right). 
	\]
\end{lem}

\begin{proof}
	When $i+1\leq 0$, the sheaf $\pi_{N-2,*}(F^{1/0}\cV_{k,i}[-1])$ has a resolution $\cF^{\bullet}_{i,1}$ which is concentrated in positive degrees with terms
	\begin{equation*}
		\cF^{-n}_{i,1} = \left( \bS^{k+i,k+i+n} V_{2}^{\prime} \oplus \bS^{k+i-1,k+i+n+1} V_{2}^{\prime} \right) \otimes \extp^{N-3-k-2i-n} V_{N-2} \langle 4-N -3k-2i-2n \rangle,
	\end{equation*}
	for $-k-2i-1 \leq n \leq -1$. 
	Using the same formula from Lemma~\ref{lem:fil}, the sheaf $\pi_{N-2,*}(F^{2/1}\cV_{k,i-1}[-2])$ is resolved by $\cF^{\bullet}_{i-1,2}$, where
	\begin{equation*} 
		\cF^{-n}_{i-1,2} = \bS^{k+i-1,k+i+n+1} V_{2}^{\prime} \otimes \extp^{N-3-k-2i-n} V_{N-2} \langle 4-N -3k-2i-2n \rangle,
	\end{equation*}
	for $-k-2i-1\leq n \leq -2$.
	Similarly, the sheaf $\pi_{N-2,*}(F^{0}\cV_{k,i+1})$ has a resolution $\cF^{\bullet}_{i+1,0}$ with terms
	\begin{equation*} 
		\cF^{-n}_{i+1,0} = \bS^{k+i,k+i+n} V_{2}^{\prime} \otimes \extp^{N-3-k-2i-n} V_{N-2}\langle 4-N -3k-2i-2n \rangle,
	\end{equation*}
	for $-k-2i-1 \leq n \leq 0$. 
	By our construction, the (nonzero) $\tG$-equivariant morphism $\pi_{N-2,*}(\iota_{2/1}): \cF^{\bullet}_{i-1,2} \to \cF^{\bullet}_{i,1}$ is unique up to a scalar multiple. 
	An obvious example of such an equivariant morphism is given by the inclusion $\cF^{-n}_{i-1,2} \hookrightarrow \cF^{-n}_{i,1}$ into the second summand in $\cF^{-n}_{i,1}$ at each $-k-2i-1\leq n \leq -2$. 
	For the same reason, the equivariant morphism $\pi_{N-2,*}(\iota_{1/0})$ is given by the identity map on the first summand in $\cF_{i,1}^{-n}$ onto $\cF_{i+1,0}^{-n}$ at each $-k-2i-1\leq n \leq -1$. 
	
	When $i+1>0$, the resolutions are given by a slightly different formula in Lemma~\ref{lem:fil}, where the terms are mostly concentrated in negative degrees, but the same argument applies. 
	
	Now we can cancel the common terms that appear in same degrees in $\cF^{\bullet}_{i-1,2}$, $\cF^{\bullet}_{i,1}$, and $\cF^{\bullet}_{i+1,0}$ in the totalization (\ref{eq:total}), and after that we are left with 
	\[
	\cF^{0}_{i+1,0} = \bS^{k+i,k+i}V_{2}^{\prime} \otimes \extp^{N-3-k-2i}V_{N-2} \langle 4-N-3k-2i\rangle.
	\]
	This cancellation is carried out at each step of the iterated cones (\ref{eq:itcone1}).
	Taking into account the shift $\langle 2(N-4) \rangle$ coming from the kernel of $\bbf^{(N-4)}$, the convolution (\ref{eq:itcone1}) becomes $\on{Conv}( \cV_{k,n^{+}_{k}}^{\prime}[-n^{-}_{k}-n^{+}_{k}] \to \cdots \to \cV_{k,-n^{-}_{k}}^{\prime})$.
\end{proof}

We next calculate the convolution (\ref{eq:itcone2}) for $k=0,-1$. 
If we take the cones iteratively from the right end, we encounter the same situation as in Lemma~\ref{lem:neq0-1} except in the last cone.

\begin{lem}\label{lem:conv0}
	We have 
	\[
	\bT(\cE_{0}) = \on{Conv} \left(\cV_{0,n^{+}_{0}}^{\prime}[-n^{+}_{0}] \to \cdots \to  \cV_{0,i}^{\prime}[-i] \to \cdots \to \cV_{0,0}^{\prime} \right). 
	\]
\end{lem}

\begin{proof}
	By Lemma~\ref{lem:eimage}, we have $\be^{(2)} (\cV_{0,0})=0$ and $\be^{(1)}(\cV_{0,0}) = \bS^{0}V_{1} \langle 1-N \rangle$, so the image $\bT(\cV_{0,0})$ is the cone
	\begin{equation}\label{eq:V00cone}
		\on{Cone}\left(\bbf^{(N-3)}(\cO)[-1]\langle 2-N \rangle \xlongrightarrow{d_{1}} \bbf^{(N-4)}(\cV_{0,0}) \right). 
	\end{equation}
	By Lemma~\ref{lem:fimage}, we can resolve the term on the left-hand side in this cone by 
	\begin{equation}\label{eq:F1V00}
		\cF_{1}^{-n} = \bS^{n+1,0}V_{2}^{\prime} \otimes \extp^{N-4-n} V_{N-2} \langle N-6-2n \rangle, \quad -1\leq n \leq N-4. 
	\end{equation}
	On the other hand, we have already seen that the term on the right-hand side is an extension sheaf sitting in the triangle
	\[
	\pi_{N-2,*}(F^{2/1}\cV_{0,0})[-1] \longrightarrow \pi_{N-2,*}(F^{1/0}\cV_{0,0})\longrightarrow \bbf^{(N-4)}(\cV_{0,0}), 
	\]
	where each direct image $\pi_{N-2,*}$ in this triangle has a resolution by Lemma~\ref{lem:fil}. 
	More precisely, the sheaf $\pi_{N-2,*}(F^{1/0}\cV_{0,0})$ has a resolution
	\begin{equation}\label{eq:F10V00}
		\cF^{-n}_{0,1} = \left(\bS^{n+1,0}V_{2}^{\prime} \oplus \bS^{n,1} V_{2}^{\prime} \right) \otimes \extp^{N-4-n} V_{N-2}\langle 2- N -2n \rangle
	\end{equation}
	for $ 0\leq n \leq N-4$. 
	Comparing (\ref{eq:F10V00}) (further shifted by $\langle 2(N-4) \rangle$, which comes from the kernel of $\bbf^{(N-4)}$) with (\ref{eq:F1V00}), we expect to cancel their common terms in the cone (\ref{eq:V00cone}). 
	
	By Lemma~\ref{lem:fimage}, we have a resolution of $\bbf^{(1)}\be^{(1)}[3-N]\langle N-3 \rangle (\cV_{0,0})$ by  
	\begin{equation*}
		\cF^{-n} = \bS^{-n,0} V_{2} \otimes \extp^{n+N-3} \bC^{N}/V_{2}\langle 2-N-2n\rangle, \quad 3-N \leq n \leq 0. 
	\end{equation*}
	Like we did in the proof of Lemma~\ref{lem:counit}, we can check that $\Ext^{>0}(\cF^{-n},\cV_{0,0}) = 0$, and then claim that the counit $\varepsilon: \bbf^{(1)}\be^{(1)}[3-N]\langle N-3 \rangle \to \on{id}$ is given by a single nonzero map $\iota: \cF^{0}=\cV_{0,0}\to \cV_{0,0}$. 
	By using the adjunction $\pi_{2,*} \dashv \pi_{2}^{!}$, we see that it comes from a map 
	\[V_{1}^{\vee} \longrightarrow  (\bC^{N}/V_{2})^{\vee} \langle 2 \rangle \]
	over $\fZ_{1,2}$, which is exactly the dual of the tautological map $\bC^{N}/V_{2}\to V_{1}$ (see the paragraph of (\ref{eq:dV00}) where we took the exterior power $\extp^{\ell=1}$ of the dual of the tautological direction $\bC^{N}/V_{2}\to V_{1}$ in the Koszul complex construction when computing the cohomology). 
	In particular, this implies that the map $\iota$ is constant over the $V_{2}\hookrightarrow \bC^{N}$ direction. 
	By using the same argument as in the last part of the proof of Lemma~\ref{lem:counit}, we claim that the map $\iota:\cV_{0,0}\to \cV_{0,0}$ is a scalar multiple of the identity. 
	
	The differential $d_{1}$ is defined by including $\bbf^{(N-3)}$ into its lowest degree copy in the decomposition of $\bbf^{(N-4)}\bbf^{(1)}$, followed by the counit $\varepsilon$. 
	The differential is thus a nonzero scalar multiple of the identity on each $\cF_{1}^{-n}$ for $n\geq 0$ following the argument we used in Lemma~\ref{lem:la}.  
	The remaining summands $\bS^{n,1} V_{2}^{\prime} \otimes \extp^{N-4-n} V_{N-2} $ in (\ref{eq:F10V00}) are cancelled with the same terms in the resolution of $\pi_{N-2,*}(F^{0}\cV_{0,1})$ in the iterated cones (\ref{eq:itcone2}) by the procedure from the previous Lemma~\ref{lem:neq0-1}. 
	In the end, the convolution (\ref{eq:itcone2}) for $k=0$ is of the form $\on{Conv}(\cV_{0,n^{+}_{0}}^{\prime}[-n^{+}_{0}] \to \cdots \to \cV_{0,0}^{\prime})$.
	\end{proof}

\begin{lem}\label{lem:conv-1}
	We have \[\bT(\cE_{-1}) = \on{Conv}\left( \cV_{-1,n^{+}_{-1}}^{\prime}[-n^{+}_{-1}] \to \cdots \to  \cV_{-1,i}^{\prime}[-i] \to \cdots \to \cV_{-1,0}^{\prime} \right).\]
\end{lem}

\begin{proof}
	In this case, the sheaf $\pi_{N-2,*}(\cV_{-1,0})$ has a resolution of a different form (see Lemma~\ref{lem:fil}), and the cancellation will happen between this resolution, $\Theta_{1}(\cV_{-1,0})$, and $\pi_{N-2,*}(F^{j/j-1}\cV_{-1,1})$, $j=0,1$ in the cone of $\pi_{N-2,*}(d)$. 
	We first calculate the convolution 
	\begin{equation}\label{eq:T-10}
		\bT(\cV_{-1,0}) = \on{Conv} \left( \bbf^{(N-2)} \be^{(2)}  [-2] \langle 2 \rangle \xrightarrow{d_{2}} \bbf^{(N-3)} \be^{(1)}  [-1] \langle 1 \rangle \xrightarrow{d_{1}} \bbf^{(N-4)} \right) (\cV_{-1,0}).
	\end{equation}
	By Lemma~\ref{lem:eimage}, we have $\be^{(2)}(\cV_{-1,0}) = \det(\bC^{N})^{\vee} \langle 1-N \rangle$. 
	The kernel of $\bbf^{(N-2)}$ is the line bundle $\det(V_{N-2})^{2} \langle 2(N-2)\rangle$ on the zero section $\fZ_{0,N-2} = \Gr(N-2,N)$, and $\pi_{N-2}: \Gr(N-2,N)\to \TGr{N-2}{N}$ is the inclusion. 
	Thus, the direct image $\bbf^{(N-2)} \be^{(2)}(\cV_{-1,0})  [-2] \langle 2\rangle$ is resolved by the Koszul complex 
	\begin{align}
		\cF^{-n} = & \det(V_{N-2})^{2} \otimes \det(\bC^{N})^{-1} \otimes \extp^{n+2} (V_{2}^{\prime} \otimes V_{N-2}^{\vee}) \notag \\
		= & \left( \bS^{n+1,-1}V_{2}^{\prime} \otimes \extp^{N-4-n} V_{N-2} \right) \oplus \bigoplus_{|\mu| =n} \bS^{\mu} V_{2}^{\prime} \otimes \bS^{(\nu,-1)} V_{N-2} \label{eq:resolnTheta2},
	\end{align}
	where the term is shifted by $\langle N-5-2n \rangle$ for $-2 \leq n \leq 2N-6$.
	Here the direct sum is over all $\mu$ such that $0 \leq \mu_{2} \leq \mu_{1} \leq N-3$ and $\nu_{i} = 1$ for $i\leq N-3-\mu_{1}$, $\nu_{i} = -1$ for $i\geq N-2-\mu_{2}$, and zero otherwise.
	Note that these terms have also appeared in the resolution of $\bbf^{(N-3)}(V_{1}^{\vee})$ in Lemma~\ref{lem:fimage}.
	Recall that we have 
	\[
	\be^{(1)}(\cV_{-1,0}) = \left\{V_{1}^{\vee} \langle -N \rangle \to \underline{(\bC^{N}/V_{1})}^{\vee} \langle 2-N \rangle \right\}
	\]
	by Lemma~\ref{lem:eimage}.
	Further apply $\be^{(1)}$, we get $\be^{(1)}(V_{1}^{\vee}[1]\langle -N \rangle) = \det(\bC^{N})^{-1}[1]\langle -N \rangle$ and $\be^{(1)}((\bC^{N}/V_{1})^{\vee} \langle 2-N \rangle) = \det(\bC^{N})^{-1}[-1] \langle 2-N \rangle$.
	By definition, the differential $d_{2}$ includes $\be^{(2)}(\cV_{-1,0}) =\det(\bC^{N})^{\vee}  \langle 1-N \rangle$ into the lower degree factor $\be^{(2)}[1]\langle -1\rangle $ of $\be^{(1)}\be^{(1)} \simeq \be^{(2)}[1]\langle -1\rangle \oplus \be^{(2)}[-1]\langle 1\rangle $, which coincides with $\be^{(1)}(V_{1}^{\vee}[1]\langle -N \rangle)$ as we can see now. 
	Next we spell out the counit $\varepsilon: \bbf^{(1)}\be^{(1)}[1-N]\langle N-1 \rangle \to \on{id}$ on that factor. 
	The direct image $\bbf^{(1)}(\det(\bC^{N})^{-1})$ is supported on the zero section $\Gr(1,N)$, so it has a Koszul resolution 
	\[
	\left\{ V_{1}^{\vee} \to (\bC^{N}/V_{1})^{\vee} \to \cdots \to V_{1}^{\otimes N-2-n} \otimes \extp^{N-1-n} (\bC^{N}/V_{1})^{\vee} \to \cdots  \right\}_{0\leq n \leq N-1},
	\]
	where the term in degree $n$ is shifted by $\langle N-1-2n \rangle$. 
	Following the argument in Lemma~\ref{lem:counit} (as well as in Lemma~\ref{lem:conv0}), we claim that the counit map on this complex (further shifted  by $[1]\langle -N \rangle$) 
	with target $\be^{(1)}(\cV_{-1,0})$, is determined by two nonzero maps $\iota_{1}: V_{1}^{\vee} \to V_{1}^{\vee}$ and $\iota_{0}: (\bC^{N}/V_{1})^{\vee} \to (\bC^{N}/V_{1})^{\vee}$, which come from the unique map $\cO \to \cO$ over $\fZ_{0,1} = \Gr(1,N)$ via the adjunction $\pi_{1,*} \dashv \pi_{1}^{!}$. 
	Again, this implies both maps $\iota_{1},\iota_{0}$ are nonzero scalar multiples of the identity. 
	
	By Lemma~\ref{lem:fimage}, we have a resolution of the sheaf $\bbf^{(N-3)}(V_{1}^{\vee} )$ by 
	\begin{equation}\label{eq:resoln0}
		\cF^{-n} = \bigoplus_{|\mu| = n} \bS^{\mu} V_{2}^{\prime} \otimes \bS^{(\nu,-1)} V_{N-2} \langle 2N-6-2n \rangle, \quad 0\leq n \leq 2N-6,
	\end{equation}
	where $0 \leq \mu_{2} \leq \mu_{1} \leq N-3$ and $\nu_{i} = 1$ for $i\leq N-3-\mu_{1}$, $\nu_{i} = -1$ for $i\geq N-2-\mu_{2}$, and zero otherwise. 
	On the other hand, the image $\bbf^{(N-3)} ((\bC^{N}/V_{1})^{\vee}[-1])$ is an extension of
	$\pi_{N-2,*}((V_{N-2}/V_{1})^{\vee}\otimes \det(V_{N-2}/V_{1}))[-1]$ by $\pi_{N-2,*}((V_{2}^{\prime})^{\vee} \otimes \det(V_{N-2}/V_{1}))[-1]$. 
	By Lemma~\ref{lem:fimage} again, these two sheaves are resolved respectively by
	\begin{align}\label{eq:resoln1}
		\cF^{-n} =& \bS^{n+1,0}V_{2}^{\prime}\otimes \extp^{N-5-n} V_{N-2} \langle -2n -2 \rangle\\ &\oplus  \bigoplus_{|\mu| = n+2} \bS^{\mu} V_{2}^{\prime} \otimes \bS^{(\nu,-1)} V_{N-2} \langle -2n -4 \rangle \label{eq:resoln1-1}
	\end{align}
	for $-1\leq n \leq 2N-8$, where $1 \leq \mu_{2}\leq \mu_{1}\leq N-3$ and $\nu_{i} = 1$ for $i\leq N-3-\mu_{1}$, $\nu_{i} = -1$ for $i\geq N-2-\mu_{2}$, and zero otherwise; and 
	\begin{equation}\label{eq:resoln2}
		\cF^{-n} = \left( \bS^{n+1,-1}V_{2}^{\prime} \oplus \bS^{n,0}V_{2}^{\prime} \right) \otimes \extp^{N-4-n} V_{N-2}\langle -2n-2 \rangle, \quad -1 \leq n \leq N-4.
	\end{equation} 
	The identity map $\iota_{1}:V_{1}^{\vee} \to V_{1}^{\vee}$ induces a morphism $\bbf^{(N-3)}(\iota_{1})$ from the complex (\ref{eq:resolnTheta2}) to (\ref{eq:resoln0}) (further shifted by $\langle 1-N\rangle$, where $\langle -N \rangle$ comes from $V_{1}^{\vee}\langle -N\rangle$ and $\langle 1 \rangle$ comes from $\Theta_1$).
	Following the argument in Lemma~\ref{lem:la}, we claim that in the cone of $d_{2}: \Theta_{2}[-2](\cV_{-1,0}) \to \Theta_{1}[-1](\cV_{-1,0})$, the common terms in  (\ref{eq:resolnTheta2}) and (\ref{eq:resoln0}) are cancelled. 
	For the same reason, the map $\iota_{0}: (\bC^{N}/V_{1})^{\vee} \to (\bC^{N}/V_{1})^{\vee}$ induces a morphism from (\ref{eq:resolnTheta2}) to (\ref{eq:resoln2}) (further shifted by $\langle N-3 \rangle$ with $\langle 2(N-3) \rangle$ coming from the kernel of $\bbf^{(N-3)}$, $\langle 2-N \rangle$ coming from $(\bC^{N}/V_1 )^{\vee} \langle 2-N \rangle$, and $\langle 1\rangle $ coming from $\Theta_1$),
	and the common summands $\bS^{n+1,-1}V_{2}^{\prime} \otimes \extp^{N-4-n} V_{N-2}$ in (\ref{eq:resolnTheta2}) and (\ref{eq:resoln2}) are cancelled in $\on{Cone}(d_{2})$. 
	In conclusion, except the highest degree term $\cV_{-1,0}^{\prime} \langle N-1 \rangle$,
	the resolution (\ref{eq:resolnTheta2}) of $\Theta_{2}[-2](\cV_{-1,0})$ is cancelled in the cone of $d_{2}$. 
	We are left with the resolution (\ref{eq:resoln1}, \ref{eq:resoln1-1}) and the remaining summands $\bS^{n,0}V_{2}^{\prime}\otimes \extp^{N-4-n}V_{N-2}$ from (\ref{eq:resoln2}).    
	
	On $\fZ_{1,2}$, there is a short exact sequence $0\to V_{2}/V_{1} \to \bC^{N}/V_{1} \to \bC^{N}/V_{2} \to 0$. 
	So, the sheaf $\bbf^{(1)}((\bC^{N}/V_{1})^{\vee}\langle 2-N \rangle)$ is an extension of $\pi_{2,*}((V_{2}/V_{1})^{\vee} \otimes (V_{2}/V_{1})^{N-3})$ by $\pi_{2,*}((\bC^{N}/V_{2})^{\vee} \otimes (V_{2}/V_{1})^{N-3})$. 
	Using the resolution for $\bbf^{(1)}(V_{1})$ from Lemma~\ref{lem:fimage} and the projective formula, we obtain a resolution of $\pi_{2,*}((V_{2}/V_{1})^{\vee} \otimes (V_{2}/V_{1})^{N-3})$ by 
	\begin{equation*}
		\cF^{-n} = 
		\begin{cases}
			\bS^{N-4-n,0}V_{2}\otimes \extp^{n} \bC^{N}/V_{2} \langle -2n \rangle, & 0\leq n \leq N-4,\\
			\det(V_{2})^{-1}\otimes \det(\bC^{N}/V_{2}) \langle 4-2N \rangle, & n=N-3. 
		\end{cases}
	\end{equation*}
	The counit $\bbf^{(1)}\be^{(1)}[3-N]\langle N-3 \rangle \to \on{id}$ restricts to a map from this complex to $\cV_{-1,0}$, 
	which, by adjunction, actually comes from the unique map $\cO \to \cO$ on $\on{Fl}(1,2;N)$ (see the paragraph of (\ref{eq:dV00}) where we took the exterior power $\extp^{\ell=0}$ when computing the term $(\bC^{N}/V_{1})^{\vee}$ of the resolution of $\be^{(1)}(\cV_{-1,0})$).
	This leads to the cancellation of the summands (\ref{eq:resoln1-1}) (further shifted by $\langle N-3 \rangle$) with the same terms from the resolution of $\bbf^{(N-4)}(\cV_{-1,0})$ in Lemma~\ref{lem:fil}. 
	The convolution (\ref{eq:T-10}) is now reduced to an extension between (\ref{eq:resoln1}) and the remaining summands from (\ref{eq:resoln2}), i.e.\ it sits in a triangle of the form
	\begin{align*}
		& \left\{ \bS^{n+1,0}V_{2}^{\prime}\otimes\extp^{N-5-n}V_{N-2} \langle N-5-2n \rangle \right\}_{-1\leq n \leq N-5} \longrightarrow \\ & \left\{  \bS^{n,0}V_{2}^{\prime}\otimes \extp^{N-4-n}V_{N-2} \langle N-5-2n \rangle \right\}_{0\leq n\leq N-4}[1] \oplus  \cV_{-1,0}^{\prime}[0] 
		\longrightarrow \bT(\cV_{-1,0}). 
	\end{align*}
	The map $\bT(d): \bT(\cV_{-1,0})[-1]\to \bbf^{(N-4)}(\cV_{-1,1})$ is uniquely determined by two $\tG$-equivariant morphisms 
	\begin{align*}
		\left\{ \bS^{n+1,0}V_{2}^{\prime}\otimes\extp^{N-5-n}V_{N-2} \langle N-5-2n \rangle \right\}_{n} & \longrightarrow \pi_{N-2,*} (F^{1/0}\cV_{-1,1}) \langle 2N-8\rangle, \\
		\left\{\bS^{n,0}V_{2}^{\prime}\otimes \extp^{N-4-n}V_{N-2}\langle N-5-2n \rangle  \right\}_{n} & \longrightarrow \pi_{N-2,*}(F^{0}\cV_{-1,1})\langle 2N-8\rangle,
	\end{align*}
	where the target sheaves have resolutions from Lemma~\ref{lem:fil}. 
	In the totalization of $\bT(d)$, most of the terms are cancelled and only the cone of $\cV_{-1,1}^{\prime} [-1]\to \cV_{-1,0}^{\prime}$ from $\bT(\cV_{-1,0})$ is left. 
	Combining with our previous Lemma~\ref{lem:neq0-1} on what happens in the iterated cones taken before the last $\cV_{-1,1}$, we claim that the convolution (\ref{eq:itcone2}) for $k=-1$ is equivalent to $\on{Conv}( \cV_{-1,n^{+}_{-1}}^{\prime}[-n^{+}_{-1}] \to \cdots \to \cV_{-1,0}^{\prime})$.
\end{proof}

\begin{proof}[Proof of Theorem~\ref{thm:flop}]
	This is a combination of Lemma~\ref{lem:la}, Lemma~\ref{lem:neq0-1}, Lemma~\ref{lem:conv0} and Lemma~\ref{lem:conv-1}. 
\end{proof}

\subsection{Bar involution}
We briefly recall the definition of the bar involution on the equivariant K-theory of the cotangent bundle of a Grassmannian. 
The construction was proposed by Lusztig \cite{Lusztig} and achieved by Varagnolo and Vasserot \cite{VV03} for quiver varieties of simply-laced type.

Let $H \subset G$ be the maximal torus of diagonal matrices and $\tH = H\times \Cx$.  
Consider the Grothendieck group 
\[K_{\tH}(\TGr{n}{N})
\]
 of $\tH$-equivariant coherent sheaves on $\TGr{n}{N}$.
This is a module over the representation ring $K_{\tH}(\on{pt})$ of $\tH$. 
Suppose $q$ is the natural representation of $\Cx$, then we may identify $K_{\Cx}(\on{pt}) \cong \mathbb{Z}[q, q^{-1}]$ with the ring of Laurent polynomials in $q$. 
Recall that we introduced in \S\ref{sec:grading} an internal degree shift $\langle \cdot \rangle$ on the gradings of $\Cx$-equivariant sheaves. 
With this convention, for any equivariant coherent sheaf $F \langle d \rangle$ on $\TGr{n}{N}$, its K-theory class is denoted by $q^d [F]$.  
In particular, we have classes of the tautological vector bundles
\begin{equation}\label{eq:kconv}
	\sV_n = [V_n], \quad q\sW - \sV_n  = [\bC^N/V_{n}],
\end{equation}
where $\sV_n$ and $\sW$ are the elements in the equivariant K-theory $K_{\tH}(\TGr{n}{N})$ that are induced by the natural representations of $\GL(n)\cong \GL(V_{n})$ and $\GL(N)\cong \GL(\bC^N)$ respectively\footnote{More precisely, for any representation $V$ of $\GL(n)\times \GL(N)$, we have an induced $\GL(N)$-equivariant vector bundle $\mu^{-1}(0)^{\chi\text{-ss}}\times_{\GL(n)} V$ on the GIT quotient $\TGr{n}{N}$ (see \S\ref{sec:quiver}).}. 
The parameter $q$ in (\ref{eq:kconv}) is due to the fact that the tautological direction $a:V_n \hookrightarrow \bC^N$ in $\TGr{n}{N}$ has degree one for our choice of the $\Cx$-action.
The convention is also compatible with the degrees of the noncommutative algebra $A$ that we computed in \S\ref{sec:Koszul}. 
For example, we have $\Ext^{0}(V_2,\cO)_{1} \cong (\bC^{N})^{\vee}$, whose K-theory class is $q^{-1}\sW^{\vee}$, sitting exactly in degree one.  

Inspired by Lusztig's work \cite{L1,L2} on the equivariant K-theory of Springer resolutions, one is led to look for a $K_{H}(\on{pt})$-linear involution $\beta$ on $K_{\tH}(\TGr{n}{N})$, which is $K_{\Cx}(\on{pt})$-antilinear with respect to the involution 
\[
\bar{(\cdot)}:K_{\Cx}(\on{pt}) \longrightarrow K_{\Cx}(\on{pt}), \quad q \longmapsto q^{-1}. 
\] 
It turns out that this \textit{bar involution} is given by the composition
\[
\beta = T_{w_0} \circ b \cdot c \otimes (-) \circ \omega^* \circ \bD  
\]
of the following operators \cite[\S 6.1]{VV03}.
\begin{enumerate}[wide, labelindent=0pt, label = (\alph*)]
	\item The dualizing functor $\bD$, which is given by $\bD([F]) = (-1)^{d} q^{-d} [F^{\vee}]$ for any locally free $F$. Here $d$ is the dimension of ${\TGr{n}{N}}$.
	
	\item The isomorphism of quiver varieties 
	\[\omega: \TGr{N-n}{N} \xlongrightarrow{\dagger} \TGr{N}{N-n} \xlongrightarrow{S_{w_0}} \TGr{n}{N},\]
	where the isomorphisms $\dagger$ and $S_{w_0}$ are defined in \S\ref{sec:quiver}. 
	
	\item The constant $b = (-q)^{-n} q^{d}$. 
	
	\item The normalization
		 \begin{align*}
		 	c= (-1)^{n(N-1)} q^{n^2 +nN - 2N^2}  & \det(q\sW-\sV_{N-n})^{N-n} \\
		 	  \otimes &  \det(\sV_{N-n})^{2N-n}
		 	  \otimes \det(\sW)^{N(2n-N-1)}.
		 \end{align*}
	This is a specialization of \cite[\S5.7]{VV03} in the type $A_1$ case. 
	
	\item The braid group action
	\[
	T_{w_0}: K_{\tH}(\TGr{N-n}{N}) \simrightarrow K_{\tH}(\TGr{n}{N})
	\]
	as defined in \cite[\S 3]{VV03}. 
    As explained in Lemma~\ref{lem:TT} below, it is equivalent to the flop equivalence $[\bT^{-1}_{n}]$ only up to a twist by line bundles.
	For readers not familiar with quantum groups, this identification also serves as a definition. 
\end{enumerate}

\begin{lem}\label{lem:TT}
	We have 
	\[
	T_{w_0} = \sL_{n} \otimes \det(\sV_{n}) \otimes (-) \circ [\bT_{n}^{-1}] \circ  q^{2n}  \det(q\sW- \sV_{N-n})^{-1} \otimes \sL_{N-n}^{-1} \otimes (-),
	\]
	where $\sL_{m} = q^{m^{2}-2mN} \det(\sV_{m})^{m}\otimes \det(\sW)^{-mN}$.
\end{lem}

\begin{proof}
	The operator $T_{w_0}$ is defined following \cite[\S 5.2]{quantum} (denoted by $T_{i,1}^{\prime \prime}$ therein).
	In our type $A_{1}$ case, it has the form\footnote{In fact, this is a reduced form thanks to \cite[Proposition 5.2.2]{quantum} and \cite[Lemma 5.3]{CKL-sl2}. }
	\[
	T_{w_0}  = \sum_{i=0}^{n} (-q)^{i} f^{(i)} e^{(N-2n+i)},
	\]
	where $e^{(N-2n+i)}$, $f^{(i)}$ stand for the divided powers of the Kac--Moody generators $e,f$. 
	These operators act on the K-theory of the quiver variety by the convolution product with the elements (see \cite[Lemma 5.4]{VV03})
	\begin{align*}
		q^{(n+N)(N-2n+i)}    \det(\sV_{N-n})^{n-N}  \otimes &  \det(\sV_{n-i})^{N-n} 
		\otimes  \det(\sW)^{N(N-2n+i)}\\ 
		\otimes & \det(q\sW-\sV_{N-n})^{-(N-2n+i)},\\
		q^{-i(N-n+i)} \det(\sV_{n})^{N-n+i} \otimes & \det(\sV_{n-i})^{n-N} \otimes \det(\sW)^{-iN}
	\end{align*}
	of $K_{\tH}(\fZ_{n-i,N-n})$ and $K_{\tH}(\fZ_{n-i,n})$ respectively. 
	A direct comparison with the kernels $\cE^{N-n,n-i}$, $\cF^{n-i,n}$ shows that 
	\begin{align*}
		e^{(N-2n+i)} & =  \sL_{n-i} \otimes (-)  \circ [\be^{(N-2n+i)}] \circ \sL_{N-n}^{-1} \otimes (-) , \\
		f^{(i)}  & =  \sL_{n} \otimes (-)  \circ [\bbf^{(i)}] \circ \sL_{n-i}^{-1} \otimes (-).
 	\end{align*}
 	This then implies $T_{w_0} =  \sL_{n} \otimes (-)  \circ [\bT_{N-n}] \circ  \sL_{N-n}^{-1}\otimes (-)$. 
 	Combining this equation with the following one from \cite[Lemma 7.4]{CKL}
 	 \[
 	 [\bT_{N-n}] = \det(\sV_{n}) \otimes (-)  \circ [\bT_{n}^{-1}] \circ q^{2n}  \det(q\sW- \sV_{N-n})^{-1}\otimes (-),
 	 \]
 	 we prove the lemma. 
\end{proof}

Define a pairing $(- \mmid -): K_{\tH}(\TGr{n}{N}) \times  K_{\tH}(\TGr{n}{N}) \to \on{Frac}(K_{\tH}(\on{pt}))$ such that\footnote{This is a simplified expression; see \cite[p.247]{VV03}.}
\[
\overline{( F \mmid G)} = \left[ R\Gamma ( \bD(F) \otimes^{L} \beta(G) ) \right]^{\dagger}.
\]
Following Lusztig, we consider 
\[
B = \left\{ b\in K_{\tH}(\TGr{n}{N}) \mid \beta(b) = b, \,\, (b \mmid b) \in 1+ q^{-1} K_{H}(\on{pt}) \ldb q^{-1} \rdb  \right\}. 
\]
The set $B$ is a \textit{signed basis} of the $K_{\Cx}(\on{pt})$-module $K_{\tH}(\TGr{n}{N})$, which means there exists a basis $B^{\circ}$ of $K_{\tH}(\TGr{n}{N})$ such that $B = B^{\circ} \sqcup -B^{\circ}$. 
Moreover, one finds that the basis $B$ is `orthonormal' with respect to the above pairing in degree zero: for any $b,b'\in B$, we have 
\[
(b \mmid b') \in \delta_{b,b'} + q^{-1} K_{H}(\on{pt}) \ldb q^{-1} \rdb.
\]
We refer to \cite[\S 7]{VV03} for more details. 

\subsection{Invariance}\label{sec:inv}
Now we are ready to show that the K-theory classes 
$[\cE_{\la}]$ and $[\cE_{k}] = \sum_{i}[\cV_{k,i}]$
of the indecomposable summands of our tilting bundle are invariant under the bar involution and that they provide a categorical lift of the canonical basis up to shifts by $K_{H}(\on{pt})$. 

\begin{thm}\label{thm:inv}
	Tensoring with $\cO(-3) \langle 7-2N \rangle$, the indecomposable summands of $\cE$ have K-theory classes that are invariant under the bar involution $\beta$. 
\end{thm}

\begin{proof}
	From \cite[Lemma 4.6]{VV03}, we know that
	\[
	\omega^* (\sV_{2}^{\vee}) = q^{-2} (q\sW - \sV_{N-2}), \quad \omega^* (q^{-1}\sW^{\vee} - \sV_{2}^{\vee}) = q^{-2} \sV_{N-2}.  
	\]
	Spelling out the action of the involution $\beta$, we find that
	\begin{align*}
		\beta ([\cE_{\la} \otimes \det(V_2)^3]) & =  q^{14-4N} \det(\sV_2)^{3} \otimes [\bT_{2}^{-1}] \left( [\bT_{2}(\cE_{\la})] \right), \\
		\beta ([\cV_{k,i}\otimes \det(V_2)^3]) &=  q^{14-4N} \det(\sV_2)^{3} \otimes [\bT_{2}^{-1}] \left( [\cV_{k,i}^{\prime}] \right). 
	\end{align*}
	Because $\beta$ is $\mathbb{Z}[q,q^{-1}]$-antilinear and by Theorem~\ref{thm:flop}, the classes $[\cE_{\la} \otimes \cO(-3)]$ and $[\cE_{k} \otimes \cO(-3)]$ become invariant under $\beta$ once we shift them by $q^{7-2N}$. 
\end{proof}

\begin{cor}
	Tensoring with $q^{7-2N}[\cO(-3)]$, the K-theory classes of $\{\cE_{\la},\cE_k \}_{\la,k}$ belong to the signed basis $B$. 
\end{cor}

\begin{proof}
	We have already proved the invariance. 
	Now, we compute the pairing 
	\begin{align*}
		\overline{([\cE_{a}] \mmid [\cE_{b}])} &=  \left[R\Gamma \left( \bD R\mathcal{H}om \left(\cE_{b}, \cE_{a} \right)   \right) \right]^{\dagger} \\
		& = \left[\Ext^{0}(\cE_{b}, \cE_{a})^{\vee} \right]^{\dagger}. 
	\end{align*}
	By Theorem~\ref{thm:Koszul}, we immediately have $([\cE_{a}] \mmid [\cE_{b}]) \in \delta_{a,b} + q^{-1} K_{H}(\on{pt}) \ldb q^{-1} \rdb$. 
\end{proof}

\subsection{Resolutions}\label{sec:cal}

\begin{lem}\phantomsection\label{lem:eimage} 
	\begin{enumerate}[wide, labelindent=0pt, label = (\arabic*)]
		\item When $\la_{2}\geq 0$, we have $\be^{(1)}( \cE_{\la}) = 0$.
		When $\la_{2}= -1$, 
		\begin{equation*}
			\be^{(1)} (\cE_{\la}) = \bS^{\la_{1}+1} V_{1} \otimes \det(\bC^{N}/V_{1})^{-1} \langle 1\rangle, \quad \be^{(2)} (\cE_{\la}) = 0. 
		\end{equation*}
		
		\item For $i\neq 0,-1$, we have $\be^{(1)}(\cV_{k,i}) = 0$. 
		Otherwise, 
		\begin{align*}
			& \be^{(1)}(\cV_{k,0})[k] = \bS^{k} V_{1} \langle k+1-N \rangle = \be^{(1)}(\cV_{k,-1})[k-1], \quad k \neq 0,-1, \\
			&  \be^{(1)}(\cV_{0,0}) = \bS^{0} V_{1} \langle 1-N \rangle, \quad \be^{(1)}(\cV_{-1,0}) = \{V_{1}^{\vee}\langle -N \rangle \to \underline{ (\bC^{N}/V_{1})}^{\vee}  \langle 2-N\rangle \}.
		\end{align*}
		Moreover, $\be^{(2)}(\cV_{k,i})=0$ for all $k, i$ except $\be^{(2)}(\cV_{-1,0}) = \det(\bC^{N})^{\vee} \langle 1-N \rangle$. 
	\end{enumerate}
\end{lem}

Note that $\bS^{k}V_{1}$, $0\leq k \leq N-2$ are `highest weight' objects in the sense that they are killed by $\be^{(1)}$.  

\begin{proof}
	For a Schur functor of the form $\bS^{\la}V_{2}\otimes \bS^{\mu}\bC^{N}/V_{2}$, the direct image 
	\[
	\be^{(1)}(\bS^{\la}V_{2}\otimes \bS^{\mu}\bC^{N}/V_{2}) = \pi_{1,*} \left( \bS^{\la} V_{2}\otimes \bS^{\mu}\bC^{N}/V_{2} \otimes \det(\bC^{N}/V_{2})^{-1}\otimes V_{1} \right)
	\]
	is a torsion sheaf supported on the locus where the tautological map $\bC^{N}/V_{1} \to V_{1}$ vanishes on $V_{2}$.
	Thus, by applying Theorem~\ref{thm:torsion} with 
	\begin{equation*}
		\cT = \Hom(\bC^{N}/V_{1},V_{1}), \quad \cU = \Hom(\bC^{N}/V_{2},V_{1}), \quad G/P = \Gr(V_{2}/V_{1}, \bC^{N}/V_{1})
	\end{equation*}
	relatively over $\Gr(V_{1},\bC^{N})$,
	we can resolve this torsion sheaf by
	\begin{equation*}
		\cF^{-n} = \bigoplus_{\ell \geq n} H^{\ell-n}_{G/P} \left( \bS^{\la}V_{2}\otimes \bS^{\mu}\bC^{N}/V_{2} \otimes \det(\bC^{N}/V_{2})^{-1}\otimes V_{1} \otimes  \extp^{\ell} V_{2}/V_{1} \otimes  \bS^{\ell} V_{1}^{\vee} \right). 
	\end{equation*}
	Here the Schur functor $\bS^{\la} V_{2}$ has a natural filtration with associated graded pieces \cite[\S 2.3]{Wey}
	\begin{equation}\label{eq:assocgr}
		\bigoplus_{d=0}^{\la_{1}-\la_{2}} \bS^{\la_{1}-d} V_{1} \otimes \bS^{\la_{2}+d} V_{2}/V_{1}. 
	\end{equation}
	
	Now we compute the cohomology by using the Borel--Weil--Bott theorem over $G/P$. 
	For the Schur functor $\cE_{\la}$, consider the concatenation weight 
	\[(-1, \cdots, -1, \la_{2}+d+\ell). \]
	When $\la_{2}\geq 0$, we have $0\leq \la_{2}+d+\ell \leq \la_{1}+\ell \leq N-3$.
	So, the weight is not in a dominant orbit. 
	On the other hand, if $\la_{2}= -1$, the only nonvanishing term is $\cF^{0} = H^{0} =  \bS^{\la_{1}+1}V_{1}\otimes \det(\bC^{N}/V_{1})^{-1}$ at $d= 0$ and $\ell =0$. 
	
	For the bundle $\cV_{k,i}$, we consider the concatenation weight 
	\begin{equation}\label{eq:dV00}
		(0,\cdots,0,-1,\cdots,-1,k+i+\ell)
	\end{equation}
	whose first $N-3-k-2i$ coordinates are zero, followed by the next $k+2i+1$ coordinates $-1$. 
	If $k=-1$ and $i=0$, it is already dominant for either $\ell = 0$ or $1$, and the resolution has terms $\cF^{0} = H^{0} = (\bC^{N}/V_{1})^{\vee}$ and $\cF^{-1} = H^{0} =  V_{1}^{\vee}$. 
	Otherwise, we need to permute the last coordinate of (\ref{eq:dV00}) towards the left by either $k+2i+1$ or $N-2$ consecutive transpositions to obtain a dominant weight.
	In that case, it is necessary to have $(k+i+\ell) - (k+2i+1) = \ell -i-1 = 0$ or $(k+i+\ell)-(N-2) \geq 1$ respectively.
	However, the latter is impossible since $k+i\leq N-3$.
	Thus, when $i=\ell-1$, the concatenation weight (\ref{eq:dV00}) is permuted into the zero weight, which contributes to the only nonvanishing term $\cF^{k+i} = H^{k+2i+1} = \bS^{k+i+1-\ell}V_{1} = \bS^{k}V_{1}$. 
	
	The kernel of $\be^{(2)}$ is supported on the zero section $\Gr(2,N)$, so we have 
	\begin{equation*}
		\be^{(2)}(\cV_{k,i}) = H^{\bullet} \left( \Gr(2,N), \cV_{k,i} \otimes \det(\bC^{N}/V_{2})^{-2} \right). 
	\end{equation*}
	It is straightforward to check that $\be^{(2)}(\cV_{k,i}) = 0$ except $\be^{(2)}(\cV_{-1,0}) = \det(\bC^{N})^{\vee}$. 
	
	The internal degree shift on the image can be computed by adding the one coming from the kernel of $\be$ with the shift $\langle -2\ell \rangle$ contributed by the power $\extp^{\ell}$ of the Koszul complex. 
\end{proof}

\begin{lem}\phantomsection\label{lem:fimage}
	\begin{enumerate}[wide, labelindent=0pt, label = (\arabic*)]
		
		\item For $-1\leq d \leq N-2$, the sheaf $\bbf^{(1)}(\bS^{d}V_{1})$ has a resolution 
		\begin{equation*}
			\cF^{-n} = \begin{cases}
				\bS^{N-3-n,d} V_{2} \otimes \extp^{n} \bC^{N}/V_{2} \langle N-2-2n \rangle, & 0\leq n \leq N-3 - d, \\
				\bS^{d-1,N-3-n} V_{2} \otimes  \extp^{n+1} \bC^{N}/V_{2} \langle N-4-2n \rangle, & N-2-d\leq n \leq N-3.
			\end{cases}
		\end{equation*} 
		
		\item For $0 \leq d \leq N-2$, the sheaf $\bbf^{(N-3)}(\bS^{d}V_{1})$ has a resolution  
		\begin{equation*}
			\cF^{-n} = \begin{cases}
				\bS^{d-1, n} V_{2}^{\prime} \otimes \extp^{N-2-n} V_{N-2}\langle 2N-4-2d-2n \rangle, &0\leq n \leq d-1,\\
				\bS^{n,d}V_{2}^{\prime}\otimes \extp^{N-3-n} V_{N-2} \langle 2N-6-2d-2n \rangle, & d\leq n \leq N-3.
			\end{cases}
		\end{equation*}
		In contrast, the sheaf $\bbf^{(N-3)}(\bS^{-1}V_{1})$ has a resolution 
		\begin{equation*}
			\cF^{-n} = \bigoplus_{|\mu| = n} \bS^{\mu} V_{2}^{\prime} \otimes \bS^{(\nu,-1)} V_{N-2} \langle 2N-6-2n \rangle, \quad 0\leq n \leq 2N-6,
		\end{equation*}
		where $0 \leq \mu_{2} \leq \mu_{1} \leq N-3$, and $\nu_{i} = 1$ for $i\leq N-3-\mu_{1}$, $\nu_{i} = -1$ for $i\geq N-2-\mu_{2}$, and zero otherwise. 
		
		\item Let $\pi_{N-2}: \fZ_{1,N-2} \to \TGr{N-2}{N}$ be the projection from the Hecke correspondence. 
		Then, the sheaf $\pi_{N-2,*}((V_{N-2}/V_{1})^{\vee}\otimes \det(V_{N-2}/V_{1}))$ has a resolution
		\begin{equation*}
			\cF^{-n} = \left( \bS^{n,0}V_{2}^{\prime}\otimes \extp^{N-4-n} V_{N-2} \langle -2n\rangle \right) \oplus  \bigoplus_{|\mu| = n+1} \bS^{\mu} V_{2}^{\prime} \otimes \bS^{(\nu,-1)} V_{N-2}\langle -2-2n \rangle
		\end{equation*}
		for $0\leq n \leq 2N-7$ whenever the terms are well-defined. 
		Here $1 \leq \mu_{2}\leq \mu_{1}\leq N-3$, and $\nu_{i} = 1$ for $i\leq N-3-\mu_{1}$, $\nu_{i} = -1$ for $i\geq N-2-\mu_{2}$, and zero otherwise.
		
		\item The sheaf $\pi_{N-2,*}((V_{2}^{\prime})^{\vee} \otimes \det(V_{N-2}/V_{1}))$ has a resolution
		\begin{equation*}
			\cF^{-n} = \left( \bS^{n,0}V_{2}^{\prime} \otimes  \bS^{0,-1}V_{2}^{\prime} \right) \otimes \extp^{N-3-n} V_{N-2} \langle -2n \rangle, \quad 0 \leq n \leq N-3.
		\end{equation*}
	\end{enumerate}
\end{lem}

\begin{proof}
	The direct image $\bbf^{(1)}(\bS^{d}V_{1}) = \pi_{2,*} (\bS^{d}V_{1}\otimes \det(V_{2}/V_{1})^{N-3})$ is a torsion sheaf supported on the locus where the tautological map $\bC^{N}/V_{2} \to V_{2}$ has image in $V_{1}$. 
	By applying Theorem~\ref{thm:torsion} with 
	\begin{equation*}
		\cT = \Hom(\bC^{N}/V_{2},V_{2}), \quad \cU = \Hom(\bC^{N}/V_{2},V_{1}), \quad G/P = \Gr(V_{1},V_{2})
	\end{equation*}
	relatively over $\Gr(V_{2},\bC^{N})$, 
	we can resolve the torsion sheaf by 
	\begin{equation*}
		\cF^{-n}  = \bigoplus_{\ell\geq n} H^{\ell-n} \left( \bP^{1}, \bS^{d}V_{1}\otimes \bS^{N-3-\ell} V_{2}/V_{1} \otimes \extp^{\ell} \bC^{N}/V_{2} \right).
	\end{equation*}
	By the Borel--Weil--Bott theorem, the nonvanishing cohomology groups are $H^{0}$ at each $\ell \leq N-3-d$ and $H^{1}$ at each $\ell\geq N-1-d$. 
	
	The direct image $\bbf^{(N-3)}(\bS^{d}V_{1}) = \pi_{N-2,*} (\bS^{d}V_{1}\otimes \det(V_{N-2}/V_{1}))$ is supported on the locus where the tautological map $\bC^{N}/V_{N-2} \to V_{N-2}$ has image in $V_{1}$. 
	By applying Theorem~\ref{thm:torsion} with 
	\begin{equation*}
		\cT = \Hom(V_{2}^{\prime},V_{N-2}), \quad \cU = \Hom(V_{2}^{\prime},V_{1}), \quad G/P = \Gr(V_{1},V_{N-2})
	\end{equation*}
	relatively over $\Gr(V_{N-2},\bC^{N})$, 
	the torsion sheaf is resolved by 
	\begin{equation*}
		\cF^{-n}  = \bigoplus_{\ell\geq n} H^{\ell-n} \left( \bP^{N-3}, \bS^{d}V_{1}\otimes \det(V_{N-2}/V_{1}) \otimes \extp^{\ell} \left( V_{2}^{\prime} \otimes (V_{N-2}/V_{1})^{\vee} \right) \right).
	\end{equation*}
	Here the exterior power $\extp^{\ell}$ can be decomposed
	\begin{equation*}
		\bigoplus_{\mu} \bS^{\mu} V_{2}^{\prime} \otimes \bS^{\mu^{\prime}} (V_{N-2}/V_{1})^{\vee}
	\end{equation*}
	over all Young diagrams $\mu = (\mu_{1},\mu_{2})$ where $0 \leq \mu_{2}\leq \mu_{1}\leq N-3$ and $\mu_{1} + \mu_{2} = \ell$.
	We can rewrite the tensor product $\det(V_{N-2}/V_{1})\otimes \bS^{\mu^{\prime}} (V_{N-2}/V_{1})^{\vee}$ as a Schur functor $\bS^{\nu} V_{N-2}/V_{1}$ with 
	\begin{equation}\label{eq:nu}
		\nu_{1} = \cdots = \nu_{N-3-\mu_{1}} = 1,\,\, 
		\nu_{N-2-\mu_{1}} = \cdots = \nu_{N-3-\mu_{2}} = 0, \,\,
		\nu_{N-2-\mu_{2}} = \cdots = \nu_{N-3} = -1. 
	\end{equation}
	To apply the Borel--Weil--Bott theorem on $\bP^{N-3}$, we consider the concatenation weight $(\nu, d)$. 
	When $d=-1$, the weight is already dominant, and the resolution has all possible terms 
	\begin{equation*}
		\cF^{-\ell} = \bigoplus_{\mu} \bS^{\mu} V_{2}^{\prime} \otimes \bS^{(\nu,-1)} V_{N-2}. 
	\end{equation*}
	When $d\geq 0$, we need to permute the last coordinate $d$ towards the left to obtain a dominant weight. 
	As $d\leq N-2$, this is only possible when the permutation stops at $\nu_{N-3-\mu_{2}}$ or further at $\nu_{N-3-\mu_{1}}$.
	In the former case, we should have $d - \mu_{2} = 0$, which gives rise to the terms 
	\begin{equation*}
		\cF^{-n} = H^{d} = \bS^{n,d}V_{2}^{\prime}\otimes \extp^{N-3-n} V_{N-2}, \quad d\leq n \leq N-3.
	\end{equation*}
	In the latter case, we have $d-\mu_{1} = 1$, which gives rise to the terms 
	\begin{equation*}
		\cF^{-n} = H^{d-1} = \bS^{d-1, n} V_{2}^{\prime} \otimes \extp^{N-2-n} V_{N-2}, \quad 0\leq n \leq d-1. 
	\end{equation*}
	
	By the projection formula, we can tensor each term of the above resolution $\cF^{\bullet}$ for $d=0$ with $(V_{2}^{\prime})^{\vee}$ to obtain a resolution of $\pi_{N-2,*}((V_{2}^{\prime})^{\vee} \otimes \det(V_{N-2}/V_{1}))$. 
	
	Finally, to calculate $\pi_{N-2,*}((V_{N-2}/V_{1})^{\vee}\otimes \det(V_{N-2}/V_{1}))$, we first decompose 
	\[(V_{N-2}/V_{1})^{\vee}\otimes \bS^{\nu} V_{N-2}/V_{1} = \bigoplus_{p=0,1,2} \bS^{\nu(p)} V_{N-2}/V_{1}\] 
	by using Pieri's formula.
	Here, the weights $\nu(1)$ and $\nu(2)$ are the above $\nu$ in (\ref{eq:nu}) that correspond to the weights $(\mu_{1}+1,\mu_{2})$ and $(\mu_{1},\mu_{2}+1)$  respectively, and $\nu(0)$ is obtained from $\nu$ by subtracting the last coordinate $-1$ (whenever $\mu_{2} \neq 0$) by $1$. 
	By a similar argument as above, the concatenation $(\nu(p),0)$ can be permuted into a dominant weight only when $p=1$ or $p=0$. 
	In the former case, we have nothing new but a shift in the degree of the exterior power in the above $d=0$ case: we have the term $\bS^{n,0}V_{2}^{\prime}\otimes \extp^{N-4-n} V_{N-2}$ contributing to $\cF^{-n}$ for each $0\leq n \leq N-4$. 
	In the latter case, the new concatenation $(\nu(0),0) = (\cdots, -2,0)$ gives rise to terms of the form 
	\[\bS^{\mu}V_{2}^{\prime} \otimes \bS^{(\nu_{1},\cdots, \nu_{N-4}, -1,-1)} V_{N-2},\]
	which also contribute to $\cF^{-n}$ where $n+1 = \ell = |\mu|$. 
	
	Again, computing the internal degree shift is a bookkeeping. 
\end{proof}

On the Hecke correspondence $\fZ_{2,N-2}$, there is a short exact sequence of tautological vector bundles
\begin{equation*}
	0 \longrightarrow V_{N-2}/V_{2} \longrightarrow \bC^{N}/V_{2} \longrightarrow V_{2}^{\prime} \longrightarrow 0. 
\end{equation*}
So, any exterior power $\extp^{d} \bC^{N}/V_{2}$ has a natural filtration with associated graded pieces \cite[\S 2.3]{Wey}
\begin{equation}\label{eq:filtration}
	\bigoplus_{j} F^{j/j-1}, \quad  F^{j/j-1} = \extp^{d-j} V_{N-2}/V_{2} \otimes \extp^{j} V_{2}^{\prime}. 
\end{equation}
This also realizes $\extp^{d} \bC^{N}/V_{2}$ as the convolution of a complex \cite[p.262]{GM}
\begin{equation*}
	F^{2/1}[-2] \longrightarrow  F^{1/0}[-1] \longrightarrow F^{0}, \quad 2 \leq d \leq N-4;
\end{equation*}
the cone of a morphism $F^{1/0}[-1] \to F^{0}$ when $d=1$; or the cone of $F^{2/1}[-1] \to F^{1/0}$ when $d=N-3$.

\begin{lem}\phantomsection\label{lem:fil}
	\begin{enumerate}[wide, labelindent=0pt, label = (\arabic*)]
		\item The image $\bbf^{(N-4)}(\cE_{\la})$ is $\bS^{\la}V_{2}^{\prime}\langle 2(N-4-\la_{1} -\la_{2})\rangle$ when $\la_{2}\neq -1$. 
		Otherwise, it has a resolution
		\begin{equation*}
			\cF^{-n} = \begin{cases}
				\bS^{\la_{1},n} V_{2}^{\prime} \otimes \extp^{n+1} V_{N-2}^{\vee}\langle 2(N-4-\la_{1}-n)\rangle, & 0 \leq n \leq \la_{1},\\
				\bS^{n,\la_{1}+1}V_{2}^{\prime} \otimes \extp^{n+2} V_{N-2}^{\vee}\langle 2(N-5-\la_{1}-n)\rangle,  & \la_{1}+1\leq n \leq N-4. 
			\end{cases}
		\end{equation*}
		
		\item The sheaf $\bbf^{(N-4)}(\cV_{-1,0})$ has a resolution
		\begin{equation*}
			\cF^{-n} = \bigoplus_{|\mu| = n+2} \bS^{\mu} V_{2}^{\prime} \otimes \bS^{(\nu,-1)} V_{N-2}\langle N-7-2n \rangle, \quad 0\leq n \leq 2N-8,
		\end{equation*}
		where $1 \leq \mu_{2} \leq \mu_{1} \leq N-3$ and $\nu_{i} = 1$ for $i\leq N-3-\mu_{1}$, $\nu_{i} = -1$ for $i\geq N-2-\mu_{2}$, and zero otherwise. 
		
		\item Let $\pi_{N-2}: \fZ_{2,N-2} \to \TGr{N-2}{N}$ be the projection from the Hecke correspondence and consider 
		\begin{equation*}
			F^{j/j-1} \cV_{k,i} := \det(V_{2})^{k+i}  \otimes \extp^{N-3-k-2i-j} V_{N-2}/V_{2} \otimes \extp^{j} V_{2}^{\prime}\langle -(N-2-k-2i) \rangle
		\end{equation*}
		for $ j=0,1,2$ whenever it is well-defined. 
		Then, except for $F^{2/1}\cV_{-1,0} = \cV_{-1,0}$, the sheaf $\pi_{N-2,*}(F^{j/j-1}\cV_{k,i})$ has a resolution 
		\begin{equation*}
			\cF^{-n} = \bS^{k+i-1,k+i-1+n} V_{2}^{\prime}\otimes \extp^{j}V_{2}^{\prime}\otimes \extp^{N-1-k-2i-j-n} V_{N-2} \langle 6-N-3k-2i-2n \rangle
		\end{equation*}
		for $1-k-2i-j\leq n \leq 0$ if $i+j\leq 0$, and otherwise if $i+j \geq 1$, it has a resolution    
		\begin{equation*}
			\cF^{-n} = \bS^{k+i+n, k+i} V_{2}^{\prime} \otimes \extp^{j}V_{2}^{\prime}\otimes  \extp^{N-3-k-2i-j-n} V_{N-2} \langle 2-N-3k-2i-2n \rangle
		\end{equation*} 
		for $0\leq n \leq N-3-k-2i-j$. 
	\end{enumerate}
\end{lem}

\begin{proof}
	We first calculate the direct image $\pi_{N-2,*}(\cE_{\la})$. 
	It is a torsion sheaf supported on the locus where the tautological map $\bC^{N}/V_{N-2} \to V_{N-2}$ has image in $V_{2}$. 
	By applying Theorem~\ref{thm:torsion} with 
	\begin{equation*}
		\cT= \Hom(V_{2}^{\prime},V_{N-2}), \quad \cU=\Hom(V_{2}^{\prime},V_{2}), \quad G/P = \Gr(V_{2},V_{N-2})
	\end{equation*}
	relatively over $\Gr(V_{N-2},\bC^{N})$, 
	the torsion sheaf can be resolved by 
	\begin{equation*}
		\cF^{-n} = \bigoplus_{\ell\geq n} H^{\ell - n} \left( \Gr(2,N-2),\bS^{\la}V_{2} \otimes \extp^{\ell} \left(V_{2}^{\prime}\otimes ( V_{N-2}/V_{2})^{\vee}\right) \right). 
	\end{equation*}
	The exterior power $\extp^{\ell}$ is decomposed into 
	\begin{equation*}
		\bigoplus_{\mu} \bS^{\mu} V_{2}^{\prime} \otimes \bS^{\nu} V_{N-2}/V_{2},
	\end{equation*}
	where $0 \leq \mu_{2}\leq \mu_{1}\leq N-4$ and $\mu_{1} + \mu_{2} = \ell$, and the coordinates of $\nu = - \mu^{\prime}$ satisfy
	\begin{equation}\label{eq:nu2}
		\nu_{1} = \cdots = \nu_{N-4-\mu_{1}} = 0,\,\,
		\nu_{N-3-\mu_{1}} = \cdots = \nu_{N-4-\mu_{2}} = -1,\,\,
		\nu_{N-3-\mu_{2}} = \cdots = \nu_{N-4} = -2. 
	\end{equation}
	Now we consider the concatenation weight
	\begin{equation*}
		(\underbrace{0,\cdots,0}_{N-4-\mu_{1}}, \underbrace{-1,\cdots,-1}_{\mu_{1}-\mu_{2}}, \underbrace{-2,\cdots,-2}_{\mu_{2}},\la_{1},\la_{2}). 
	\end{equation*}
	When $\la_{2}\geq 0$, to obtain a dominant weight, we need to permute both $\la_{1}$ and $\la_{2}$ towards the left. 
	The permutation must stop at zero because $0 \leq \la_{1},\la_{2} \leq N-4$.  
	This can only happen when $\la_{2}=\mu_{2}$ and $\la_{1} = \mu_{1}$, and it contributes to the only nonvanishing term $\cF^{0} = H^{\la_{1}+\la_{2}}  = \bS^{\la} V_{2}^{\prime}$. 
	On the other hand, if $\la_{2}=-1$, we only need to permute $\la_{1}$. 
	There are two situations: the permutation stops at $-1$ (resp.\ $0$) and has length $\mu_{2}$ (resp.\ $\mu_{1}$). 
	In the former case, we have $\la_{1}-\mu_{2} = -1$, which gives rives to the terms
	\begin{equation*}
		\cF^{-n} = H^{\la_{1}+1} = \bS^{n,\la_{1}+1}V_{2}^{\prime} \otimes \extp^{n+2} V_{N-2}^{\vee}, \quad \la_{1}+1\leq n \leq N-4. 
	\end{equation*}
	In the latter case, we have $\la_{1}-\mu_{1} = 0$, which gives rise to the terms 
	\begin{equation*}
		\cF^{-n} = H^{\la_{1}} = \bS^{\la_{1},n} V_{2}^{\prime} \otimes \extp^{n+1} V_{N-2}^{\vee}, \quad 0 \leq n \leq \la_{1}. 
	\end{equation*}
	
	All other direct images in the lemma are  of the form 
	\[
	\pi_{N-2,*}\left( \det(V_{2})^{d_{1}} \otimes \extp^{d_{2}} V_{N-2}/V_{2}\right). 
	\]
	By Pieri's formula, we have a decomposition 
	\[
	\bS^{\nu}V_{N-2}/V_{2} \otimes \extp^{d_{2}} V_{N-2}/V_{2} \cong \bigoplus_{a+b+c = d_{2}} \bS^{\nu(a,b,c)} V_{N-2}/V_{2},
	\]
	where $\nu(a,b,c)$ is obtained from $\nu$ by raising respectively the first $a$, $b$ and $c$ coordinates in the three equations in (\ref{eq:nu2}) by $1$.
	That is to say, the concatenation weight $(\nu(a,b,c),d_{1},d_{1})$ is of the form
	\begin{equation}\label{eq:conwt}
		(\underbrace{1,\cdots,1}_{a}, \underbrace{0, \cdot\dots\cdot\dots\cdot,0}_{N-4-\mu_{1}-a+b}, \underbrace{-1,\cdot\dots\cdot,-1}_{\mu_{1}-\mu_{2}-b+c}, \underbrace{-2,\cdots,-2}_{\mu_{2}-c}, d_{1},d_{1}). 
	\end{equation}
	
	For $\cV_{-1,0}$, we take $d_{1}=-1$ and $d_{2} = N-4$, and the weight (\ref{eq:conwt}) is always dominant.
	Take into account the factor $\extp^{2}V_{2}^{\prime}$ of $\cV_{-1,0} = F^{2/1}\cV_{-1,0}$, we obtain a resolution 
	\begin{equation*}
		\cF^{-n} = \bigoplus_{|\mu| = n} \bS^{(\mu_{1}+1,\mu_{2}+1)} V_{2}^{\prime} \otimes \bS^{(\nu,-1,-1)} V_{N-2}, \quad 0\leq n \leq 2N-8,
	\end{equation*}
	where $0 \leq \mu_{2} \leq \mu_{1} \leq N-4$ and $\nu_{i} = 1$ for $i\leq N-4-\mu_{1}$, $\nu_{i} = -1$ for $i\geq N-3-\mu_{2}$, and zero otherwise.
	By reordering the index, we obtain the asserted formula. 
	
	For all other $\cV_{k,i}$, we have $0\leq d_{1} \leq N-3$, so we assume this condition in the following. 
	The only possible way to obtain a dominant weight is to permute the last two coordinates $d_{1}$ in (\ref{eq:conwt}) to the left until they are reduced to $1$ or $0$.
	In the former case, we must have $N-4-\mu_{1}-a=0, b = 0 $ and $d_{1}-(\mu_{1}-b)=1$. 
	This implies $\mu_{1}=d_{1}-1$ and $d_{1}+d_{2} = N-3+c \geq N-3$, and the corresponding cohomology is 
	\[
	H^{2(d_{1}-1)} = \extp^{d_{1}+d_{2}+1-\mu_{2}} V_{N-2}.
	\]
	In the latter case, it is necessary to have $\mu_{1}-\mu_{2}-b = 0, c=0$ and $d_{1}-(\mu_{2}-c) = 0$, which implies $\mu_{2}=d_{1}$ and $d_{1}+d_{2} = \mu_{1}+a \leq N-4$.
	The corresponding cohomology in this situation is 
	\[
	H^{2d_{1}} = \extp^{d_{1}+d_{2}-\mu_{1}}V_{N-2}. 
	\]
	In summary, when $d_{1}+d_{2} \geq N-3$, we obtain a resolution 
	\[
	\cF^{-n} = \bS^{d_{1}-1,d_{1}-1+n}V_{2}^{\prime}\otimes \extp^{d_{2}+2-n} V_{N-2},
	\] 
	and when $d_{1}+d_{2} \leq N-4$, we have 
	\[ \cF^{-n} = \bS^{d_{1}+n, d_{1}} V_{2}^{\prime} \otimes   \extp^{d_{2}-n} V_{N-2}.\]
	In both cases, the degree $n\in\mathbb{Z}$ takes all possible values. 
	The formula in the lemma is then obtained by simply taking $d_{1} = k+i$ and $d_{2} = N-3-k-2i-j$. 
	
	As before, obtaining the internal degree shift is a bookkeeping. 
\end{proof}

\printbibliography

@Article{Kap2,
	Author = {Kapranov, Mikhail},
	Title = {On the derived categories of coherent sheaves on some homogeneous spaces},
	FJournal = {Inventiones Mathematicae},
	Journal = {Invent. Math.},
	ISSN = {0020-9910},
	Volume = {92},
	Number = {3},
	Pages = {479--508},
	Year = {1988},
	Language = {English},
	DOI = {10.1007/BF01393744},
	URL = {https://eudml.org/doc/143579},
	zbMATH = {4061472},
	Zbl = {0651.18008}
}

@Book{quantum,
Author = {Lusztig, George},
Title = {Introduction to quantum groups},
Edition = {Reprint of the 1994 ed.},
FSeries = {Modern Birkh{\"a}user Classics},
Series = {Mod. Birkh{\"a}user Class.},
Year = {2010},
Publisher = {Boston, MA: Birkh{\"a}user},
Language = {English},
DOI = {10.1007/978-0-8176-4717-9},
zbMATH = {5288789},
Zbl = {1246.17018}
}

@Article{L1,
	Author = {Lusztig, George},
	Title = {Bases in equivariant {{\(K\)}}-theory},
	FJournal = {Representation Theory},
	Journal = {Represent. Theory},
	ISSN = {1088-4165},
	Volume = {2},
	Pages = {298--369},
	Year = {1998},
	Language = {English},
	DOI = {10.1090/S1088-4165-98-00054-5},
	zbMATH = {1201815},
	Zbl = {0901.20034}
}

@Article{L2,
	Author = {Lusztig, George},
	Title = {Bases in equivariant {{\(K\)}}-theory. {II}},
	FJournal = {Representation Theory},
	Journal = {Represent. Theory},
	ISSN = {1088-4165},
	Volume = {3},
	Pages = {281--353},
	Year = {1999},
	Language = {English},
	DOI = {10.1090/S1088-4165-99-00083-7},
	zbMATH = {1460684},
	Zbl = {0999.20036}
}

@Article{LMN-L,
	Author = {Lusztig, George},
	Title = {Quiver varieties and {Weyl} group actions},
	FJournal = {Annales de l'Institut Fourier},
	Journal = {Ann. Inst. Fourier},
	ISSN = {0373-0956},
	Volume = {50},
	Number = {2},
	Pages = {461--489},
	Year = {2000},
	Language = {English},
	DOI = {10.5802/aif.1762},
	URL = {https://eudml.org/doc/75426},
	zbMATH = {1448497},
	Zbl = {0958.20036}
}

@Article{LMN-N,
	Author = {Nakajima, Hiraku},
	Title = {Reflection functors for quiver varieties and {Weyl} group actions.},
	FJournal = {Mathematische Annalen},
	Journal = {Math. Ann.},
	ISSN = {0025-5831},
	Volume = {327},
	Number = {4},
	Pages = {671--721},
	Year = {2003},
	Language = {English},
	DOI = {10.1007/s00208-003-0467-0},
	zbMATH = {2078202},
	Zbl = {1060.16017}
}

@Article{LMN-M,
	Author = {Maffei, Andrea},
	Title = {A remark on quiver varieties and {Weyl} groups.},
	FJournal = {Annali della Scuola Normale Superiore di Pisa. Classe di Scienze. Serie V},
	Journal = {Ann. Sc. Norm. Super. Pisa, Cl. Sci. (5)},
	ISSN = {0391-173X},
	Volume = {1},
	Number = {3},
	Pages = {649--686},
	Year = {2002},
	Language = {English},
	URL = {https://eudml.org/doc/84483},
	zbMATH = {2217018},
	Zbl = {1143.14309}
}

@Article{VV03,
	Author = {Varagnolo, Michela and Vasserot, Eric},
	Title = {Canonical bases and quiver varieties},
	FJournal = {Representation Theory},
	Journal = {Represent. Theory},
	ISSN = {1088-4165},
	Volume = {7},
	Pages = {227--258},
	Year = {2003},
	Language = {English},
	zbMATH = {2115421},
	Zbl = {1055.17007}
}

@Article{Lusztig,
	Author = {Lusztig, G.},
	Title = {Remarks on quiver varieties},
	FJournal = {Duke Mathematical Journal},
	Journal = {Duke Math. J.},
	ISSN = {0012-7094},
	Volume = {105},
	Number = {2},
	Pages = {239--265},
	Year = {2000},
	Language = {English},
	DOI = {10.1215/S0012-7094-00-10523-6},
	zbMATH = {1820762},
	Zbl = {1017.20040}
}

@Article{KR20,
	Author = {Keller, Bernhard and Krause, Henning},
	Title = {Tilting preserves finite global dimension},
	FJournal = {Comptes Rendus. Math{\'e}matique. Acad{\'e}mie des Sciences, Paris},
	Journal = {C. R., Math., Acad. Sci. Paris},
	ISSN = {1631-073X},
	Volume = {358},
	Number = {5},
	Pages = {563--570},
	Year = {2020},
	Language = {English},
	DOI = {10.5802/crmath.72},
	zbMATH = {7267916},
	Zbl = {1451.18032}
}

@Article{BVdB,
 Author = {Bondal, A. and Van den Bergh, M.},
 Title = {Generators and representability of functors in commutative and noncommutative geometry},
 FJournal = {Moscow Mathematical Journal},
 Journal = {Mosc. Math. J.},
 ISSN = {1609-3321},
 Volume = {3},
 Number = {1},
 Pages = {1--36},
 Year = {2003},
 Language = {English},
 zbMATH = {2069670},
 Zbl = {1135.18302}
}

@Article{CK1,
 Author = {Cautis, Sabin and Kamnitzer, Joel},
 Title = {Braiding via geometric {Lie} algebra actions},
 FJournal = {Compositio Mathematica},
 Journal = {Compos. Math.},
 ISSN = {0010-437X},
 Volume = {148},
 Number = {2},
 Pages = {464--506},
 Year = {2012},
 Language = {English},
 DOI = {10.1112/S0010437X1100724X},
 zbMATH = {6028614},
 Zbl = {1249.14005}
}

@Article{Bei,
 Author = {Beilinson, A.},
 Title = {Coherent sheaves on {{\({\mathbb{P}}^n\)}} and problems of linear algebra},
 FJournal = {Functional Analysis and its Applications},
 Journal = {Funct. Anal. Appl.},
 ISSN = {0016-2663},
 Volume = {12},
 Pages = {214--216},
 Year = {1979},
 Language = {English},
 DOI = {10.1007/BF01681436},
 zbMATH = {3659709},
 Zbl = {0424.14003}
}

@Book{3264,
 Author = {Eisenbud, David and Harris, Joe},
 Title = {3264 and all that. {A} second course in algebraic geometry},
 Year = {2016},
 Publisher = {Cambridge: Cambridge University Press},
 Language = {English},
 DOI = {10.1017/CBO9781139062046},
 zbMATH = {6562439},
 Zbl = {1341.14001}
}

@Article{Kaneda,
 Author = {Kaneda, Masaharu},
 Title = {Another strongly exceptional collection of coherent sheaves on a {Grassmannian}},
 FJournal = {Journal of Algebra},
 Journal = {J. Algebra},
 ISSN = {0021-8693},
 Volume = {473},
 Pages = {352--373},
 Year = {2017},
 Language = {English},
 DOI = {10.1016/j.jalgebra.2016.10.043},
 zbMATH = {6671130},
 Zbl = {1368.14030}
}

@Article{RSVdB,
 Author = {Raedschelders, Theo and {\v{S}}penko, {\v{S}}pela and Van den Bergh, Michel},
 Title = {The {Frobenius} morphism in invariant theory},
 FJournal = {Advances in Mathematics},
 Journal = {Adv. Math.},
 ISSN = {0001-8708},
 Volume = {348},
 Pages = {183--254},
 Year = {2019},
 Language = {English},
 DOI = {10.1016/j.aim.2019.03.013},
 zbMATH = {7055759},
 Zbl = {1430.14098}
}

@Article{Kr,
 Author = {Krause, Henning},
 Title = {Koszul, {Ringel} and {Serre} duality for strict polynomial functors.},
 FJournal = {Compositio Mathematica},
 Journal = {Compos. Math.},
 ISSN = {0010-437X},
 Volume = {149},
 Number = {6},
 Pages = {996--1018},
 Year = {2013},
 Language = {English},
 DOI = {10.1112/S0010437X12000814},
 zbMATH = {6200374},
 Zbl = {1293.20046}
}

@Article{Hara-Abuaf,
 Author = {Hara, Wahei},
 Title = {On the {Abuaf}-{Ueda} flop via non-commutative crepant resolutions},
 Journal = {SIGMA, Symmetry Integrability Geom. Methods Appl.},
 ISSN = {1815-0659},
 Volume = {17},
 Pages = {paper 044, 22},
 Year = {2021},
 Language = {English},
 DOI = {10.3842/SIGMA.2021.044},
 zbMATH = {7342622},
 Zbl = {1469.14036}
}

@Article{Nak94,
 Author = {Nakajima, Hiraku},
 Title = {Instantons on {ALE} spaces, quiver varieties, and {Kac}-{Moody} algebras},
 FJournal = {Duke Mathematical Journal},
 Journal = {Duke Math. J.},
 ISSN = {0012-7094},
 Volume = {76},
 Number = {2},
 Pages = {365--416},
 Year = {1994},
 Language = {English},
 DOI = {10.1215/S0012-7094-94-07613-8},
 zbMATH = {727829},
 Zbl = {0826.17026}
}

@Article{BGS,
 Author = {Beilinson, Alexander and Ginzburg, Victor and Soergel, Wolfgang},
 Title = {Koszul duality patterns in representation theory},
 FJournal = {Journal of the American Mathematical Society},
 Journal = {J. Am. Math. Soc.},
 ISSN = {0894-0347},
 Volume = {9},
 Number = {2},
 Pages = {473--527},
 Year = {1996},
 Language = {English},
 DOI = {10.1090/S0894-0347-96-00192-0},
 zbMATH = {868359},
 Zbl = {0864.17006}
}

@Article{TU,
 Author = {Toda, Yukinobu and Uehara, Hokuto},
 Title = {Tilting generators via ample line bundles},
 FJournal = {Advances in Mathematics},
 Journal = {Adv. Math.},
 ISSN = {0001-8708},
 Volume = {223},
 Number = {1},
 Pages = {1--29},
 Year = {2010},
 Language = {English},
 DOI = {10.1016/j.aim.2009.07.006},
 zbMATH = {5643955},
 Zbl = {1187.18009}
}

@article{Web24,
      title={Tilting generator for the {$T^*Gr(2,4)$} {C}oulomb Branch}, 
      author={Aiden Suter and Ben Webster},
      year={2024},
      eprint={2409.01379},
      archivePrefix={arXiv},
      url={https://arxiv.org/abs/2409.01379}
}

@Book{Huy,
 Author = {Huybrechts, D.},
 Title = {Fourier-{Mukai} transforms in algebraic geometry},
 FSeries = {Oxford Mathematical Monographs},
 Series = {Oxford Math. Monogr.},
 ISBN = {0-19-929686-3},
 Year = {2006},
 Publisher = {Oxford: Clarendon Press},
 Language = {English},
 zbMATH = {5028970},
 Zbl = {1095.14002}
}

@Article{BM,
 Author = {Bezrukavnikov, Roman and Mirkovi{\'c}, Ivan},
 Title = {Representations of semisimple {Lie} algebras in prime characteristic and the noncommutative {Springer} resolution},
 FJournal = {Annals of Mathematics. Second Series},
 Journal = {Ann. Math. (2)},
 ISSN = {0003-486X},
 Volume = {178},
 Number = {3},
 Pages = {835--919},
 Year = {2013},
 Language = {English},
 DOI = {10.4007/annals.2013.178.3.2},
 zbMATH = {6220725},
 Zbl = {1293.17021}
}

@Article{symduality,
 Author = {Kamnitzer, Joel},
 Title = {Symplectic resolutions, symplectic duality, and {Coulomb} branches},
 FJournal = {Bulletin of the London Mathematical Society},
 Journal = {Bull. Lond. Math. Soc.},
 ISSN = {0024-6093},
 Volume = {54},
 Number = {5},
 Pages = {1515--1551},
 Year = {2022},
 Language = {English},
 DOI = {10.1112/blms.12711},
 zbMATH = {7738371},
 Zbl = {1531.16013}
}

@Article{Kaledin,
 Author = {Kaledin, Dmitry},
 Title = {Derived equivalences by quantization},
 FJournal = {Geometric and Functional Analysis. GAFA},
 Journal = {Geom. Funct. Anal.},
 ISSN = {1016-443X},
 Volume = {17},
 Number = {6},
 Pages = {1968--2004},
 Year = {2008},
 Language = {English},
 DOI = {10.1007/s00039-007-0623-x},
 zbMATH = {5275294},
 Zbl = {1149.14009}
}

@article{Hikita,
      title={Elliptic canonical bases for toric hyper-{K}{\"a}hler manifolds}, 
      author={Tatsuyuki Hikita},
      year={2020},
      eprint={2003.03573},
      archivePrefix={arXiv}
}

@Article{CK,
 Author = {Cautis, Sabin and Kamnitzer, Joel},
 Title = {Knot homology via derived categories of coherent sheaves. {I}: {The} {{\(\mathfrak{sl}(2)\)}}-case},
 FJournal = {Duke Mathematical Journal},
 Journal = {Duke Math. J.},
 ISSN = {0012-7094},
 Volume = {142},
 Number = {3},
 Pages = {511--588},
 Year = {2008},
 Language = {English},
 DOI = {10.1215/00127094-2008-012},
 zbMATH = {5278251},
 Zbl = {1145.14016}
}

@book{GM,
    AUTHOR = {Gelfand, Sergei I. and Manin, Yuri I.},
     TITLE = {Methods of homological algebra},
    SERIES = {Springer Monographs in Mathematics},
   EDITION = {2nd ed.},
 PUBLISHER = {Springer-Verlag, Berlin},
      YEAR = {2003},
     PAGES = {xx+372},
      ISBN = {3-540-43583-2},
   MRCLASS = {18-02 (18Exx 18Gxx 55U35)},
  MRNUMBER = {1950475},
       DOI = {10.1007/978-3-662-12492-5},
       URL = {https://doi.org/10.1007/978-3-662-12492-5},
}

@InCollection{Kawa,
 Author = {Kawamata, Yujiro},
 Title = {Derived equivalence for stratified {Mukai} flop on {{\(G(2,4)\)}}},
 BookTitle = {Mirror symmetry V. Proceedings of the BIRS workshop on Calabi-Yau varieties and mirror symmetry, December 6--11, 2003},
 ISBN = {0-8218-4251-X},
 Pages = {285--294},
 Year = {2006},
 Publisher = {Providence, RI: American Mathematical Society (AMS); Somerville, MA: International Press},
 Language = {English},
 zbMATH = {5153039},
 Zbl = {1137.14305}
}

@Article{Kap,
 Author = {Kapranov, Mikhail},
 Title = {On the derived category of coherent sheaves on {Grassmann} manifolds},
 FJournal = {Mathematics of the USSR. Izvestiya},
 Journal = {Math. USSR, Izv.},
 ISSN = {0025-5726},
 Volume = {24},
 Pages = {183--192},
 Year = {1985},
 Language = {English},
 DOI = {10.1070/IM1985v024n01ABEH001221},
 zbMATH = {3899068},
 Zbl = {0564.14023}
}

@Article{Web,
      title={Coherent sheaves and quantum {Coulomb} branches {I}: tilting bundles from integrable systems}, 
      author={Ben Webster},
      year={2019},
      eprint={1905.04623},
      archivePrefix={arXiv}
}

@article{Web2,
      title={Coherent sheaves and quantum {Coulomb} branches {II}: quiver gauge theories and knot homology}, 
      author={Ben Webster},
      year={2022},
      eprint={2211.02099},
      archivePrefix={arXiv}
}

@Article{Hara,
 Author = {Hara, Wahei},
 Title = {Non-commutative crepant resolution of minimal nilpotent orbit closures of type {A} and {Mukai} flops},
 FJournal = {Advances in Mathematics},
 Journal = {Adv. Math.},
 ISSN = {0001-8708},
 Volume = {318},
 Pages = {355--410},
 Year = {2017},
 Language = {English},
 DOI = {10.1016/j.aim.2017.08.010},
 zbMATH = {6769057},
 Zbl = {1387.14016}
}

@Article{DS,
 Author = {Donovan, Will and Segal, Ed},
 Title = {Window shifts, flop equivalences and {Grassmannian} twists},
 FJournal = {Compositio Mathematica},
 Journal = {Compos. Math.},
 ISSN = {0010-437X},
 Volume = {150},
 Number = {6},
 Pages = {942--978},
 Year = {2014},
 Language = {English},
 DOI = {10.1112/S0010437X13007641},
 zbMATH = {6333838},
 Zbl = {1354.14028}
}

@Book{Wey,
 Author = {Weyman, Jerzy M.},
 Title = {Cohomology of vector bundles and syzygies},
 FSeries = {Cambridge Tracts in Mathematics},
 Series = {Camb. Tracts Math.},
 ISSN = {0950-6284},
 Volume = {149},
 ISBN = {0-521-62197-6},
 Year = {2003},
 Publisher = {Cambridge University Press, Cambridge},
 Language = {English},
 zbMATH = {1983902},
 Zbl = {1075.13007}
}

@Article{Cautis,
 Author = {Cautis, Sabin},
 Title = {Equivalences and stratified flops},
 FJournal = {Compositio Mathematica},
 Journal = {Compos. Math.},
 ISSN = {0010-437X},
 Volume = {148},
 Number = {1},
 Pages = {185--208},
 Year = {2012},
 Language = {English},
 DOI = {10.1112/S0010437X11005616},
 zbMATH = {6007114},
 Zbl = {1236.14016}
}

@Article{CKL-Duke,
 Author = {Cautis, Sabin and Kamnitzer, Joel and Licata, Anthony},
 Title = {Coherent sheaves and categorical {{\(\mathfrak{sl}_2\)}} actions},
 FJournal = {Duke Mathematical Journal},
 Journal = {Duke Math. J.},
 ISSN = {0012-7094},
 Volume = {154},
 Number = {1},
 Pages = {135--179},
 Year = {2010},
 Language = {English},
 DOI = {10.1215/00127094-2010-035},
 zbMATH = {5769089},
 Zbl = {1228.14011}
}

@Article{CKL,
 Author = {Cautis, Sabin and Kamnitzer, Joel and Licata, Anthony},
 Title = {Coherent sheaves on quiver varieties and categorification},
 FJournal = {Mathematische Annalen},
 Journal = {Math. Ann.},
 ISSN = {0025-5831},
 Volume = {357},
 Number = {3},
 Pages = {805--854},
 Year = {2013},
 Language = {English},
 DOI = {10.1007/s00208-013-0921-6},
 zbMATH = {6226052},
 Zbl = {1284.14016}
}

@Article{CKL-sl2,
 Author = {Cautis, Sabin and Kamnitzer, Joel and Licata, Anthony},
 Title = {Derived equivalences for cotangent bundles of {Grassmannians} via categorical {{\(\mathfrak{sl}_2\)}} actions},
 FJournal = {Journal f{\"u}r die Reine und Angewandte Mathematik},
 Journal = {J. Reine Angew. Math.},
 ISSN = {0075-4102},
 Volume = {675},
 Pages = {53--99},
 Year = {2013},
 Language = {English},
 DOI = {10.1515/CRELLE.2011.184},
 zbMATH = {6146432},
 Zbl = {1282.14034}
}

\vspace{1em}
\begin{flushleft}
    {\fontsize{9.5}{11}\selectfont Department of Mathematics, Imperial College, London, SW7 2AZ, United Kingdom \\
    \href{mailto:w.zhou21@imperial.ac.uk}{\texttt{w.zhou21@imperial.ac.uk}}}
\end{flushleft}
\end{document}